\documentclass[a4paper,10pt,oneside]{article}
\usepackage{latexsym,amsfonts,amssymb,amsmath}
\usepackage{amsthm}
\usepackage{mathrsfs}

\usepackage[english]{babel}

\usepackage{diagrams}

\usepackage[hmargin=0.9in, vmargin=0.9in]{geometry}

\DeclareMathOperator{\codim}{codim}

\DeclareMathOperator{\Spec}{Spec}

\DeclareMathOperator{\Pic}{Pic}

\DeclareMathOperator{\coker}{coker}

\DeclareMathOperator{\sym}{sym}

\DeclareMathAlphabet{\mathpzc}{OT1}{pzc}{m}{it}

\newcommand{\mbb}{\mathbb}

\newcommand{\mc}{\mathcal}

\newcommand{\FS}{\mathcal{O}}
\newcommand{\comp}[1]{{#1}^{\bullet}}

\newcommand{\id}{{\rm id}}

\newcommand{\B}{\mathbf}

\newcommand{\Tor}{{\rm Tor}}

\newcommand{\perm}{\mathfrak{S}}

\newcommand{\trest}{{\big |} }

\newcommand{\bkrh}{\mathbf{\Phi}}

\newcommand{\tens}{\otimes}
\newcommand{\Tens}{\bigotimes}

\newcommand{\Stab}{\mathrm{Stab}}

\newcommand{\K}{\mathcal{K}}

\newtheorem{theorem}{Theorem}[section]
\newtheorem{lemma}[theorem]{Lemma}
\newtheorem{pps}[theorem]{Proposition}
\newtheorem{crl}[theorem]{Corollary}

\theoremstyle{definition}

\newtheorem{notat}[theorem]{Notation}

\theoremstyle{remark}
\newtheorem{remark}[theorem]{Remark}

\numberwithin{equation}{section}

\mathchardef\phi="0127
\mathchardef\varphi="011E

\mathchardef\alpha="710B
\mathchardef\beta="710C
\mathchardef\gamma="710D
\mathchardef\delta="710E
\mathchardef\epsilon="7122
\mathchardef\zeta="7110
\mathchardef\eta="7111
\mathchardef\theta="7112
\mathchardef\iota="7113
\mathchardef\kappa="7114
\mathchardef\lambda="7115
\mathchardef\mu="7116
\mathchardef\nu="7117
\mathchardef\xi="7118
\mathchardef\pi="7119
\mathchardef\rho="711A
\mathchardef\sigma="711B
\mathchardef\tau="711C
\mathchardef\upsilon="711D
\mathchardef\chi="711F
\mathchardef\psi="7120
\mathchardef\omega="7121
\mathchardef\varepsilon="710F
\mathchardef\vartheta="7123
\mathchardef\varpi="7124
\mathchardef\varrho="7125
\mathchardef\varsigma="7126

\theoremstyle{definition}

\newtheorem*{conventions}{Conventions}

\theoremstyle{remark}

\numberwithin{equation}{section}

\usepackage{float}
\restylefloat{table}
\usepackage{graphics}
\usepackage{array}
\usepackage{multirow}

\usepackage{caption}
\usepackage{tikz}
\usetikzlibrary{arrows}

\DeclareMathOperator{\Res}{Res}
\DeclareMathOperator{\reg}{reg}
\DeclareMathOperator{\Graphs}{\mc{G}}
\newcommand{\gammac}{\comp{\varGamma}}
\newcommand{\ccomp}[1]{\comp{\mc{#1}}}

\usepackage{hyperref}

\title{Notes on diagonals of the product and \\ symmetric variety of a surface}
\author{Luca Scala}
\date{}

\begin{document}
\maketitle

\begin{abstract}Let $X$ be a smooth quasi-projective algebraic surface and let $\Delta_n$ the big diagonal in the product variety $X^n$. We study cohomological properties of the ideal sheaves $\mc{I}^k_{\Delta_n}$ and their invariants $(\mc{I}^k_{\Delta_n})^{\perm_n}$ by the symmetric group, seen as ideal sheaves 
over the symmetric variety $S^nX$. In particular we obtain resolutions of the sheaves of invariants $(\mc{I}_{\Delta_n})^{\perm_n}$ for $n = 3,4$ in terms of invariants of sheaves over $X^n$ whose cohomology is easy to calculate. Moreover, 
we relate, via  the Bridgeland-King-Reid equivalence,  
powers of determinant line bundles over the Hilbert scheme to powers of ideals of the big diagonal $\Delta_n$. We deduce 
applications to the cohomology of double powers of determinant line bundles over the Hilbert scheme with $3$ and $4$ points and we give universal formulas for their Euler-Poincar\'e characteristic. Finally, we obtain upper bounds for the regularity of the sheaves $\mc{I}^k_{\Delta_n}$ over $X^n$ 
with respect to very ample line bundles of the form $L \boxtimes \cdots \boxtimes L$ and of their sheaves of invariants $( \mc{I}^k_{\Delta_n})^{\perm_n}$ 
on the symmetric variety $S^nX$ with respect to  very ample line bundles of the form~$\mc{D}_L$. 
\end{abstract}

\section*{Introduction}
Let $X$ be a smooth quasi-projective algebraic surface and consider the product variety $X^n$, for $n \geq 2$. The big diagonal $\Delta_n$ is the closed subscheme of $X^n$ defined as the scheme-theoretic union of all pairwise diagonals $\Delta_I$, where $I$ is a cardinality $2$ subset of $\{1, \dots, n \}$. 
The aim of this article is the study of 
of the ideal sheaf $\mc{I}_{\Delta_n}$ of the big diagonal $\Delta_n$ of the product variety $X^n$ and its invariants $(\mc{I}_{\Delta_n})^{\perm_n}$, seen as an ideal sheaf over the symmetric variety $S^n X$. 

The main reason 
for studying diagonal ideals is that their geometry is tightly intertwined 
with the geometry of the Hilbert scheme of points $X^{[n]}$ and that of the 
isospectral Hilbert scheme $B^n$. As an example of this close interplay, we mention that in  \cite{Scala2015isospectral} we related the singularities of the isospectral  Hilbert scheme in terms of the singularities of the pair $(X^n, \mc{I}_{\Delta_n})$ and  it is by studying the  latter
that we could prove that the singularities of $B^n$ are canonical if $n \leq 5$, log-canonical if $n \leq 7$ and not log-canonical if $n \geq 9$. 

In this work we are concerned with cohomological properties  of the ideal sheaf $\mc{I}_{\Delta_n}$ and its invariants $( \mc{I}_{\Delta_n} )^{\perm_n}$. A first motivation  comes from the study of symmetric powers $S^k L^{[n]}$ of tautological bundles over the Hilbert scheme of points.
Let $\mu: X^{[n]} \rTo S^n X$ be the Hilbert-Chow morphism. 
 In \cite{Scalaarxiv2015}
we built a natural filtration 
$\mc{W}^\lambda \mu_* S^k L^{[n]}$ of the push-forward $\mu_* S^k L^{[n]}$,  
indexed by partitions $\lambda$ of $k$ of length at most $n$,  
whose graded sheaves,  at least for $k$ low --- but we believe 
it is a general fact ---,  
are given, up to tensorization by some line bundles, by invariants of diagonal ideals by certain subgroups of $\perm_n$: we indicate them with $\mc{L}^\lambda (-2 \lambda \Delta)$. 
The sheaves $\mc{L}^\lambda(-2 \lambda \Delta)$ are in general more complicated than 
the invariants $( \mc{I}_{\Delta_n} )^{\perm_n}$ of the big diagonal; however, for $\lambda = 1^k$ (in exponential notation), the sheaf 
$\mc{L}^{\lambda }(-2 \lambda \Delta)$ is directly related to the sheaf of invariants $( \mc{I}_{\Delta_n} )^{\perm_n}$, as we shall see. 
Therefore, understanding the invariants $( \mc{I}_{\Delta_n} )^{\perm_n}$ is instrumental for the  investigation of 
symmetric powers of tautological bundles. 

There is, moreover, another and more direct source of interest for the cohomological study 
 of diagonal ideals. We recall that the Bridgeland-King-Reid transform 
$$ \bkrh: \B{R}p_* \circ q^* : \B{D}^b(X^{[n]}) \rTo D^b_{\perm_n}(X^n) \;,$$where $p : B^n \rTo X^n$ is the blow-up  along the diagonal $\Delta_n$ and where $q: B^n \rTo X^{[n]}$ is the (flat) quotient map by the symmetric group $\perm_n$, is an equivalence of derived 
categories  between the derived category of coherent sheaves over the Hilbert scheme $X^{[n]}$ and the $\perm_n$-equivariant derived category of the product variety $X^n$. Since the ideal $\mc{I}_{\Delta_n}$, or its powers $\mc{I}^k_{\Delta_n}$, are $\perm_n$-equivariant sheaves over $X^n$, they need to correspond, under the BKR-equivalence, to some remarkable object over the Hilbert scheme of points: indeed, it turns out that they are related, up to tensorization by the alternating representation $\epsilon_n$ of the symmetric group, to powers of 
determinant line bundles on $X^{[n]}$. Considering slightly more general $\perm_n$-equivariant objects on $X^n$
we proved the following 

\vspace{0.3cm}
\noindent
{\bf Theorem \ref{thm: projectionformula} }{\it Let $F$ be a vector bundle of rank $r$ and $A$ be a line bundle over the smooth quasi-projective surface 
$X$. Consider, over the Hilbert scheme of points $X^{[n]}$,  the rank $nr$ tautological bundle $F^{[n]}$  associated to $F$ and the natural line bundle $\mc{D}_A$ associated to $A$. 
Then, in $\B{D}^b_{\perm_n}(X^n)$ we have: 
\begin{equation}\label{eq: equivalenceintro}\tag{$\ast$} \bkrh( (\det F^{[n]} )^{\tens k} \tens \mc{D}_A) \simeq^{qis}  \mc{I}_{\Delta_n}^{rk} \tens \big( ( (\det F )^{\tens k} \tens A ) \boxtimes \cdots \boxtimes ( (\det F)^{\tens k} \tens A) ) \big)\tens \epsilon_n^{rk} \;.\end{equation}}

\vspace{-0.1cm}
\noindent
Taking $\perm_n$-invariants, the previous theorem yields 
\begin{equation}\label{eq: formulaHilbertChowintro} \tag{$\ast \ast$} \B{R} \mu_* ((\det F^{[n]})^{\tens k} \tens \mc{D}_A ) \simeq \pi_*( \mc{I}_{\Delta_n}^{rk} \tens \epsilon_n^{rk} ) ^{\perm_n} \tens \mc{D}_{\det F}^k \tens \mc{D}_A \;.\end{equation}
Therefore, by means of the BKR-equivalence $\bkrh$,  or of the derived push forward $\B{R} \mu_*$ of the Hilbert-Chow morphism, one might  use facts about  diagonal ideals and their invariants to deduce properties of determinants 
line bundles 
over Hilbert schemes;  on the other hand, one can use known results about determinants on  Hilbert schemes of points to enlighten properties of the ideal $\mc{I}_{\Delta_n}$, its powers and their invariants. This is precisely what happens, as we explain below. 

In order to obtain resolutions of invariants $(\mc{I}_{\Delta_n})^{\perm_n}$ in terms of simpler sheaves, at least from the point of view of cohomology computations, we consider, for cardinality $2$ subsets $I \subseteq \{1, \dots, n \}$, complexes 
$$ \mc{K}^\bullet_I : \FS_{X^n } \rTo \FS_{\Delta_I} \rTo 0 $$which we take as right resolutions of the ideals $\mc{I}_{\Delta_I}$. Being the ideal  $\mc{I}_{\Delta_n}$ the intersection of the ideals $\mc{I}_{\Delta_I}$, the former is isomorphic to the zero-cohomology sheaf $\mc{H}^0( \tens_{I} \mc{K}^\bullet_I)$, where $I$ runs among cardinality $2$ subsets of $\{1, \dots n \}$. However, the complex $\tens_{I} \mc{K}^\bullet_I$ is far from being exact, because the partial diagonals are not transverse: we are then led to consider the derived tensor product $\tens_I^L \mc{K}^\bullet_I$ of the complexes $\mc{K}^\bullet_I$ and its associated spectral sequence 
\begin{equation}\label{eq: spectralintro}\tag{$\star$} E^{p,q}_1 := \bigoplus_{i_1 + \dots + i_m= p} \Tor_{-q}(\mc{K}^{i_1}_{I_1}, \dots, \mc{K}^{i_m}_{I_m}) \end{equation}abutting to the sheaf cohomology 
$ \mc{H}^{p+q}(\tens_I^L \mc{K}^\bullet_I)$. Here $m = \binom{n}{2}$ and $\{I_1, \dots, I_m \}$ are all cardinality $2$ subsets of $\{1, \dots, n \}$. It is now clear that the ideal $\mc{I}_{\Delta_n}$ of the big diagonal is given by the term $E^{0,0}_2$ of the spectral sequence above. In order to deal with the latter we now face two difficulties. The first is the understanding of the multitors appearing in 
(\ref{eq: spectralintro}), which are of the form $\Tor_{-q}(\FS_{\Delta_{I_1}}, \dots, \FS_{\Delta_{I_l}})$, for some cardinality $2$ multi-indexes 
$I_1, \dots, I_l$, $l \leq m$. The second is the handling of combinatorial possibilities given by the multi-indexes $I_1, \dots, I_l$. As for the first, we establish in section \ref{section: prima} the following general formula for multitors of structural sheaves of smooth subvarieties 
$Y_1, \dots, Y_l$ of a smooth variety $M$ intersecting in a smooth variety $Z = Y_1 \cap \cdots \cap Y_l$
$$ \Tor_{j}(\FS_{Y_1}, \dots, \FS_{Y_l}) \simeq \Lambda^j (\oplus_{i=1}^l N_{Y_i} \trest_Z / N_Z)^* $$in terms of normal bundles $N_{Y_i}$ and $N_Z$ of $Y_i$ and $Z$ in $M$, respectively. For the second, we think of the multi-indexes $I_i$ 
as edges of a subgraph $\Gamma$ of the complete graph $K_n$ with $n$ vertices, such that no vertex of $\Gamma$ is isolated. Classifying all possible multitors appearing in (\ref{eq: spectralintro}) ---  up to isomorphisms 
--- is then reduced to classifying all possible graphs $\Gamma$ of this kind. The usefulness of the graph-theoretic approach extends further since several interesting properties of the multitor $\Tor_q(\FS_{\Delta_{I_1}}, \dots, \FS_{\Delta_{I_l}})$ can be understood in graph-theoretic terms. A fundamental fact is that  if $I_1, \dots, I_l$ identify a graph $\Gamma$ with $c$ independent cycles, then the associated 
multitor $\Tor_q(\Delta, \Gamma):=\Tor_q(\FS_{\Delta_{I_1}}, \dots, \FS_{\Delta_{I_l}})$ is isomorphic to the exterior power $\Lambda^q(Q^*_\Gamma)$, where $Q_\Gamma$ is 
a rank $2c$ vector bundle over the intersection $\Delta_\Gamma = \Delta_{I_1} \cap \cdots \cap \Delta_{I_l}$. Resorting to the associated graphs is very helpful also when considering the $\perm_n$-action on the naturally equivariant 
spectral sequence $E^{p,q}_1$. The $\perm_n$-action on $E^{p,q}_1$ 
induces a $\Stab_{\perm_n}(\Gamma)$-action on the multitor $\Tor_q(\Delta, \Gamma) \simeq \Lambda^q(Q^*_\Gamma)$; the group $\Stab_{\perm_n}(\Gamma)$ acts fiberwise on the vector bundle $Q_\Gamma$ via the representation $\mbb{C}^2 \tens q_\Gamma$, 
where $q_\Gamma$ is the vector space generated over $\mbb{C}$ by independent cycles. 
These facts allow us to classify, for $n = 3,4$ all isomorphism classes of multitors appearing in (\ref{eq: spectralintro}) and thoroughly study the spectral sequence of invariants $(E^{p,q}_1)^{\perm_n}$. As a consequence, we deduce the following 
right resolutions of $(\mc{I}_{\Delta_n})^{\perm_n}$ over the symmetric variety $S^n X$ for $n = 3,4$; even if we are mainly interested when $X$ is a surface, the result actually holds for $X$ a smooth algebraic variety of arbitrary dimension.  
Denote with $w_k$ the  map $X \times S^{n-k} X \rTo S^n X$ sending $(x, y)$ to $kx + y$. 

\vspace{0.3cm}
\noindent
{\bf Theorems \ref{thm: inv3}--\ref{thm: inv4}.} {\it Let $X$ be a smooth algebraic variety. We have the following natural resolutions of ideals of invariants $\mc{I}_{\Delta_3}^{\perm_3}$ and 
$\mc{I}_{\Delta_4}^{\perm_4}$ over the symmetric varieties $S^3X$ and $S^4 X$, respectively. }
\begin{gather*} 
0 \rTo \big( \mc{I}_{\Delta_3} \big)^{\perm_3} \rTo \FS_{S^3 X} {\rTo^r} {w_2}_* (\FS_{X \times X}) {\rTo^D} {w_3}_*( \Omega^1_X ) \rTo 0 \\
 0 \rTo \big( \mc{I}_{\Delta_4} \big)^{\perm_4} \rTo \FS_{S^4 X} {\rTo^r} {w_2}_*(\FS_X \boxtimes \FS_{S^2 X})_0 {\rTo^{D}} {w_3}_* (\Omega^1_X \boxtimes \mc{I}_{\Delta_2})_0 
 {\rTo^{C}} {w_4}_* (S^3 \Omega^1_X)  \rTo 0 \end{gather*}
 
 \vspace{0.2cm}
 \noindent
 The maps $r, D, C$  are explicit.  Here we indicated with 
 ${w_2}_*(\FS_X \boxtimes \FS_{S^2 X})_0$ and 
 ${w_3}_* (\Omega^1_X \boxtimes \mc{I}_{\Delta_2})_0$
 particular subsheaves of ${w_2}_*(\FS_X \boxtimes \FS_{S^2 X})$ and ${w_3}_* (\Omega^1_X \boxtimes \mc{I}_{\Delta_2})$ which will be made precise in subsection  \ref{subsection: n=4}. 
 Now,  the terms 
 appearing in the resolutions do not present any difficulty from the point of view of cohomology computations. Therefore, by 
 theorems \ref{thm: inv3}, \ref{thm: inv4} and by formula (\ref{eq: formulaHilbertChowintro}) we deduce, for $n = 3,4$, and for $X$ a surface,
 \emph{a spectral sequence $E^{p,q}_1$ abutting to the cohomology $H^{p+q}(X^{[n]}, (\det L^{[n]})^{\tens 2} \tens \mc{D}_A)$} and 
 \emph{universal formulas for the Euler-Poincar\'e characteristic $\chi(X^{[n]}, (\det L^{[n]})^{\tens 2} \tens \mc{D}_A)$ of twisted double powers of 
 determinant line bundles over Hilbert schemes of points
 }. 
 These facts have a direct application to the sheaves $\mc{L}^\lambda(-2 \lambda \Delta)$, when $\lambda = (r, \dots, r)$, $|\lambda| = rl$, since the sheaf $\mc{L}^\lambda(-2 \lambda \Delta)$ is isomorphic to 
 $$ \mc{L}^\lambda(-2 \lambda \Delta) \simeq {v_l}_* \big( ( \mc{I}^{2r}_{\Delta_{l}} )^{\perm_l} \tens \mc{D}_{L}^{\tens r} \boxtimes \FS_{S^{n-l} X} \big) \;,$$where 
 $v_l$ is the finite map $v_l : S^{l}X \times S^{n-l} X \rTo S^n X$ sending $(x, y) \rMapsto x+y$. We also obtain a right resolution 
 of the sheaf $\mc{L}^{2,1,1}(-2 \Delta)$ over $S^n X$, which is more difficult to treat, since it is not directly related to determinant line bundles 
 over Hilbert schemes. 
 
 Finally, as anticipated, we can use properties of determinant line bundles over Hilbert scheme to deduce important facts about diagonal ideals and their invariants. We use formulas (\ref{eq: equivalenceintro}) and (\ref{eq: formulaHilbertChowintro})  and the positivity properties of $\det L^{[n]}$ when $L$ is $n$-very ample to study vanishing theorems and 
 regularity for the ideal sheaves $\mc{I}^k_{\Delta_n}$ over $X^n$ and their invariants $(\mc{I}^k_{\Delta_n})^{\perm_n}$ over $S^n X$. 
In particular, we proved the following result. 

\vspace{0.3cm}
\noindent
{\bf Theorem \ref{thm: regularity1}}. {\it
Let $X$ be a smooth projective surface and $L$ be a line bundle over $X$. Suppose that, for a certain $m \in \mbb{N}^*$, $L^m \tens K_X^{-1} = \tens_{i=1}^{2 [(k+1)/2]} B_i$, with $B_1$ $n$-very ample and the other $B_i$, $i \neq 1$, $(n-1)$-very ample. Then we have the vanishing 
$$ H^i(S^n X, (\mc{I}^k_{\Delta_n})^{\perm_n} \tens \mc{D}^m_L) = 0 \qquad \qquad \mbox{for $i >0$}. $$If, moreover, $L$ is very ample over $X$, then the ideal \emph{$(\mc{I}^k_{\Delta_n})^{\perm_n}$ is $(m+2n)$-regular with respect to $\mc{D}_L$}. In particular the regularity $\reg ( (\mc{I}^k_{\Delta_n})^{\perm_n}) $ of the ideal $(\mc{I}^k_{\Delta_n})^{\perm_n}$ with respect to the line bundle $\mc{D}_{L}$ is bounded above by
$$ \reg ( (\mc{I}^k_{\Delta_n})^{\perm_n}) \leq m_0 + 2n$$where $m_0$ is the minimum of all $m$ satisfying the condition above.} 

\vspace{0.3cm}
\noindent
We proved a similar statement (theorem \ref{thm: regularity2}) about  \emph{vanishing and regularity for $\mc{I}^k_{\Delta_n}$
with respect to the line bundle $L \boxtimes \cdots \boxtimes L$ on $X^n$ for $2 \leq n \leq 7$}. These results can be written in a nicer way 
when the surface $X$ has Picard number one. Indeed, suppose that $\Pic(X) = \mbb{Z}B$ and let $r$ be the minimum positive power of $B$ such that $B^r$ is very ample and write $K_X = B^w$, for some integer $w$. Then the regularities $\reg(\mc{I}^k_{\Delta_n})$ 
and $\reg( (\mc{I}_{\Delta_n}^k)^{\perm_n}) $, with respect to the line bundle $B^r \boxtimes \cdots \boxtimes B^r$ on $X^n$ and $\mc{D}_{B^r}$ on $S^n X$, respectively, are bounded above by 
\begin{align*}
\reg(\mc{I}_{\Delta_n}^k) & \leq (k+3)n-k +\lceil w/r  \rceil  & \mbox{for $2 \leq n \leq 7$} &\\
 \reg( (\mc{I}_{\Delta_n}^k)^{\perm_n}) & \leq  2n ( [(k+1)/2] +1) - 2[(k+1)/2] +1 +\lceil w/r  \rceil 
 & \mbox{for all $n \in \mbb{N}$, $n \geq 2$} & .
 \end{align*}

\begin{conventions}\label{conv: sym}i). We work over the field of complex numbers. By point we will always mean closed point. 

\noindent
ii). Let $A$ a $\mbb{C}$-algebra and $M$ an $A$-module. For $n \in \mbb{N}\setminus \{0 \}$, consider the symmetric  power 
$S^n M$  of the module $M$. We consider $S^n M$ as the space of $\perm_n$-invariants of $M^{\tens n}$ for the action of $\perm_n$ permutating the factors in the tensor product. Throughout this article, we will use the following convention for the symmetric product $u_1. \cdots . u_n$ 
of elements $u_i \in M$: 
$$ u_1. \cdots . u_n := \sum_{\sigma \in \perm_n} u_{\sigma(1)} \tens \cdots \tens u_{\sigma(n)} \;,$$where the right hand side is seen in $M^{\tens n}$. 
We use an analogous convention for the exterior product: 
$ u_1 \wedge \cdots \wedge u_n := \sum_{\sigma \in \perm_n} (-1)^\sigma u_{\sigma(1)} \tens \cdots \tens u_{\sigma(n)} $, where $(-1)^\sigma$ is the signature of the of permutation $\sigma$ and where we see $\Lambda^n M$ as the space of anti-invariants for the action of $\perm_n$ over $M^{\tens n}$. 
\end{conventions}

\paragraph{Acknowledgements.} This work is partially supported by CNPq, grant 307795/2012-8. 

\section{The Bridgeland-King-Reid transform of diagonal ideals}\label{section: BKR}
Consider a smooth quasi-projective algebraic surface $X$. Denote with $X^{[n]}$ the Hilbert scheme of $n$ points over $X$ and with 
$B^n$ the isospectral Hilbert scheme \cite{Haiman1999, Haiman2001}, that is, the blow-up of $X^n$ along the big diagonal $\Delta_n$. We indicate with 
$p : B^n \rTo X^n$ the blow-up map, with $q: B^n \rTo X^{[n]}$ the quotient projection by the symmetric group $\perm_n$ and with 
$\mu: X^{[n]} \rTo S^n X$ the Hilbert-Chow morphism. 
The Bridgeland-King-Reid equivalence \cite{BridgelandKingReid2001, Haiman2001, Haiman2002}, in the case of the action of the symmetric group $\perm_n$ over the product variety $X^n$, provides an equivalence of derived categories 
\begin{equation} \label{eq: bkrh} \bkrh := \B{R} p_* \circ q^*:  \B{D}^b(X^{[n]}) \rTo \B{D}^b_{\perm_n}(X^n) \end{equation}from the derived category of coherent sheaves over the Hilbert scheme of $n$ points over $X$ and the $\perm_n$-equivariant derived category of the product variety $X^n$. Any power $\mc{I}^m_{\Delta_n}$  of the ideal $\mc{I}_{\Delta_n}$ is naturally a $\perm_n$-equivariant coherent sheaf  over $X^n$: it is then natural to ask what is the corresponding complex of coherent sheaves over the Hilbert schemes of points for the equivalence (\ref{eq: bkrh}). In general we can twist the ideal $\mc{I}^m_{\Delta_n}$ with the line bundle $L \boxtimes \cdots \boxtimes L$ ($n$-factors) and ask
the same question for $\mc{I}^m_{\Delta_n} \tens L \boxtimes \cdots \boxtimes L$. To give a general statement, we need to introduce the line bundle $\mc{D}_L$. 
\begin{remark}\label{rmk: DL}
If $L$ is a line bundle on $X$, the line bundle $L \boxtimes \cdots \boxtimes L$ ($n$-factors) on $X^n$ \emph{descends} to a line bundle $\mc{D}_L$ on $S^n X$, in the sense that $\pi^* \mc{D}_L = L \boxtimes \cdots \boxtimes L$ \cite[Thm 2.3]{drezetNarasimhan1989}. As a consequence, the line bundle 
$\mc{D}_L$  coincides with the sheaf of $\perm_n$-invariants, on $S^n X$, of the line bundle $L \boxtimes \cdots \boxtimes L$. Pulling-back the line bundle $\mc{D}_L$  via the Hilbert-Chow morphism $\mu: X^{[n]} \rTo S^n X$ we get a line bundle $\mu^* \mc{D}_L$ on the Hilbert scheme, called the \emph{natural line bundle} on $X^{[n]}$ associated to the line bundle $L$ on $X$. For brevity's sake, we will denote it again with $\mc{D}_L$.  
\end{remark}

We  need as well a technical lemma about the local cohomology of ideal sheaves $\mc{I}^s_{\Delta_n}$, $s \in \mbb{N}^*$
with respect to the closed subscheme $W$ given by the intersection of double diagonals in $X^n$. 
More precisely, we define $W$ as the scheme-theoretic intersection of pairwise diagonals
$$ W := \bigcap_{\substack{|I| = |J| = 2 \\ I, J \subseteq \{1, \dots, n \}, I \neq J}} \Delta_I \cap \Delta_J \;.$$It is a closed subscheme of $X^n$ of  codimension $4$. \begin{notat}\label{notat: opensets}We denote with $X^n_{**}$ the open subset $X^n \setminus W$ and with $B^n_{**}$, $S^n_{**}X$, $X^{[n]}_{**}$ 
the open subsets $ B^n_{**} := p^{-1}(X^n_{**})$, $S^n_{**}X := \pi(X^n_{**})$, $X^{[n]}_{**}:= \mu^{-1}(S^n_{**} X)$. 
We denote moreover, with $j: X^n_{**} \rTo X^n$ the open immersion of $X^n_{**}$ into $X^n$. We also denote with $j$ the open immersion of each of the open sets $B^n_{**}$, $S^n_{**}X$, $X^{[n]}_{**}$ into their closure $B^n$, $S^nX$, $X^{[n]}$, respectively; 
it will be clear from the context over which variety we are working. 
\end{notat}

\begin{remark}\label{rmk: bigdiagonal}The following is an important result about powers of the ideal sheaf of the diagonal $\Delta_n$ in $X^n$, and it has been proven by Haiman in \cite[Corollary 3.8.3]{Haiman2001}.  Over $X^n$, for all $s \in \mbb{N}$ one has:  
\begin{equation} \label{eq: bigdiagonal}
 \bigcap_{\substack{I \subseteq \{1, \dots, n \} \\ | I | =2}} \mc{I}_{\Delta_I}^s = 
 \Bigg ( \bigcap_{\substack{I \subseteq \{1, \dots, n \} \\ | I | =2}} \mc{I}_{\Delta_I} \Bigg) ^s = \mc{I}_{\Delta_n}^s 
\end{equation}
\end{remark}

The local cohomology property of the ideal sheaves we want to prove is the following. 
 \begin{lemma}\label{lmm: extensionideali}Let $\ell: \{ (i, j) \: | \; i, j \in \mbb{N} \; , \; 1 \leq i < j \leq n \} \rTo \mbb{N}$ be a function.   
 Then $$j_* j^* \bigcap_{1 \leq i < j \leq n} \mc{I}_{\Delta_{ij}}^{\ell (i,j)} = \bigcap_{1 \leq i < j \leq n} \mc{I}_{\Delta_{ij}}^{\ell (i,j)}  \;.$$
 \end{lemma}
 \begin{proof}
  We begin by proving recursively that, for any $s \in \mbb{N}$ and for any fixed natural numbers $i, j$, $1 \leq i < j \leq n$,  we have 
 \begin{equation}\label{eq: localcohom} \mc{H}^l_W(\mc{I}^s_{\Delta_{i, j}}) = 0  \qquad \mbox{for all $0 \leq l \leq 2$} \;. \end{equation}Indeed it is true for $s = 0$, since $\mc{I}^0_{\Delta_{i, j}} \simeq \FS_{X^n}$, $X^n$ is normal, and $W$ is of codimension $4$ in $X^n$. Consider now $s \in \mbb{N}$. 
 The local cohomology long exact sequence applied to the short exact sequence $$ 0 \rTo \mc{I}^{s+1}_{\Delta_{ij}} \rTo \mc{I}^s_{\Delta_{ij}} \rTo 
 \mc{I}^s_{\Delta_{ij}}/\mc{I}^{s+1}_{\Delta_{ij}} \rTo 0$$yields: 
 \begin{multline}  0 \rTo \mc{H}^0_W (\mc{I}^{s+1}_{\Delta_{ij}}) \rTo \mc{H}^0_W(\mc{I}^s_{\Delta_{ij}}) \rTo \mc{H}^0_W( \mc{I}^s_{\Delta_{ij}}/\mc{I}^{s+1}_{\Delta_{ij}} ) \rTo  \\ \rTo \mc{H}^1_W (\mc{I}^{s+1}_{\Delta_{ij}}) \rTo \mc{H}^1_W(\mc{I}^s_{\Delta_{ij}}) \rTo \mc{H}^1_W( \mc{I}^s_{\Delta_{ij}}/\mc{I}^{s+1}_{\Delta_{ij}} ) \rTo \\ 
 \rTo \mc{H}^2_W (\mc{I}^{s+1}_{\Delta_{ij}}) \rTo \mc{H}^2_W(\mc{I}^s_{\Delta_{ij}}) \;. 
 \end{multline}Note that $\mc{H}^l_W ( \mc{I}^s_{\Delta_{ij}}/\mc{I}^{s+1}_{\Delta_{ij}} ) = 0$ for $l =0,1$ because the sheaf $\mc{I}^s_{\Delta_{ij}}/\mc{I}^{s+1}_{\Delta_{ij}}$ is a vector bundle over the smooth subvariety $\Delta_{ij}$, in which $W$ is of codimension $2$ and because of  \cite[Lemma 3.1.9]{Scala2009D}. Now, if $\mc{H}^l_W(\mc{I}^s_{\Delta_{ij}}) = 0$ for $l =0, 1, 2$, 
 the local cohomology long exact sequence above yields the vanishing for $\mc{I}^{s+1}_{\Delta_{ij}}$ and $l = 0,1,2$. Induction on $s$ then yields (\ref{eq: localcohom}) for any $s \in \mbb{N}$. 
 
 Since $\mc{H}^l_W(\mc{I}^s_{\Delta_{ij}}) = 0$ for $l =0, 1, 2$ and for any $s \in \mbb{N}$, an analogous argument via the long exact sequence in local cohomology applied to the short exact sequence 
 $$ 0 \rTo \mc{I}^s_{\Delta_{ij}} \rTo \FS_{X^n} \rTo \FS_{X^n}/\mc{I}^s_{\Delta_{ij}} \rTo 0$$proves that $\mc{H}^l_W( \FS_{X^n}/\mc{I}^s_{\Delta_{ij}}) = 0$ for $l =0,1$ and for any $s$; this last fact is equivalent to the isomorphism $$ j_* j^*  \FS_{X^n}/\mc{I}^s_{\Delta_{ij}} \simeq  \FS_{X^n}/\mc{I}^s_{\Delta_{ij}} \;.$$
 
 The above isomorphism, together with the following commutative diagram 
 \begin{diagram} 0 & \rTo & \bigcap_{1 \leq i < j \leq n} \mc{I}_{\Delta_{ij}}^{\ell (i,j)}  & \rTo & \FS_{X^n} & \rTo & \bigoplus_{1 \leq i<j \leq n } \FS_{X^n}/ \mc{I}_{\Delta_{ij}}^{\ell (i,j)} \\ 
 & & \dTo & & \dTo^{\simeq} & & \dTo^{\simeq} \\ 
  0 & \rTo & j_* j^* \bigcap_{1 \leq i < j \leq n} \mc{I}_{\Delta_{ij}}^{\ell (i,j)}  & \rTo &  j_* j^* \FS_{X^n} & \rTo & \bigoplus_{1 \leq i<j \leq n }  j_* j^*\FS_{X^n}/ \mc{I}_{\Delta_{ij}}^{\ell (i,j)} 
 \end{diagram}where the second and third vertical arrows are isomorphisms, because we proved so above, yields the statement of the lemma. \end{proof}

 \begin{lemma}\label{lmm: inv2k-1}Let $k \in \mbb{N}^*$. We have the isomorphism of sheaves 
 of invariants over $S^n X$: 
 $$ \big( \mc{I}^{2k-1}_{\Delta_n}\big)^{\perm_n} \simeq \big( \mc{I}^{2k}_{\Delta_n}\big)^{\perm_n} \;.$$\end{lemma}\begin{proof}
 The statement follows using (\ref{eq: bigdiagonal}),  taking $\perm_n$-invariant in the exact sequence 
 $$ 0 \rTo \mc{I}^{2k-1}_{\Delta_n} \rTo \FS_{X^n} \rTo \oplus_{i<j} \FS_{X^n}/\mc{I}^{2k-1}_{\Delta_{ij}} $$and noting that 
 $ \pi_* \big( \oplus_{i<j} \FS_{X^n}/\mc{I}^{2k-1}_{\Delta_{ij}} \big)^{\perm_n}  \simeq \pi_* \big( \FS_{X^n}/ \mc{I}^{2k-1}_{\Delta_{12}} \big)^{\Stab_{\perm_n}(\{1, 2 \})}  \simeq \pi_* \big( \FS_{X^n}/ \mc{I}^{2k}_{\Delta_{12}} \big)^{\Stab_{\perm_n}(\{1, 2 \})} \simeq \pi_* \big( \oplus_{i<j} \FS_{X^n}/\mc{I}^{2k}_{\Delta_{ij}} \big)^{\perm_n}$, since 
 $\pi_*(\mc{I}^{2k-1}_{\Delta_{1,2}})^{\perm (\{1, 2 \}) } \simeq \pi_*(\mc{I}^{2k}_{\Delta_{1,2}})^{\perm (\{1, 2 \})} $. 
 \end{proof}
  
  \begin{remark}
  Indicate now with $E$ the exceptional divisor (or the boundary) of $X^{[n]}$: it is the exceptional divisor for the Hilbert-Chow morphism and the branching divisor for the map $q: B^n \rTo X^{[n]}$. It is well known \cite[Lemma 3.7]{Lehn1999} that 
$$ \FS_{ X^{ [n] } } (-E) \simeq (\det \FS_{X}^{[n]} )^{\tens 2} \:. $$As a consequence there exists a divisor $e$ on the Hilbert scheme $X^{[n]}$ such that $E = 2e$, and such that $\FS_{X^{[n]}}(-e) = \det \FS_X^{[n]}$. It is also well known that $\det L^{[n]} \simeq \mc{D}_L \tens \det \FS_X^{[n]}$, which can be rewritten, with the notations just explained, as \begin{equation} \label{eq: detL[n]}\det L^{[n]} \simeq \mc{D}_L(-e) \;.\end{equation}
Denote now with $E_B$ the exceptional divisor over the isospectral Hilbert scheme, that is, the exceptional divisor 
for the blow-up map $p : B^n \rTo X^n$. We have $\FS_{B^n}(-E_B) \simeq q^*\FS_{X^{[n]}}(-e)$. 
\end{remark}
Denoting with $\epsilon_n$ the alternating representation of $\perm_n$, 
we can now prove 
\begin{theorem}\label{thm: bkrhideali} Let $X$ be a smooth quasi-projective algebraic surface. Then 
\begin{gather*}  \bkrh( \FS_{X^{[n]}}(-l e)) \simeq \bkrh( (\det \FS_X^{[n]})^{\tens l}) \simeq \mc{I}^l_{\Delta_n} \tens \epsilon_n^{l } \\
 \B{R}\mu_*\FS_{X^{[n]}}(-l e) \simeq  \B{R}\mu_*  (\det \FS_X^{[n]})^{\tens l}) \simeq \pi^{\perm_n}_*(  \mc{I}^l_{\Delta_n} \tens \epsilon_n^{l} ) \;.
 \end{gather*}
\end{theorem}The proof of theorem \ref{thm: bkrhideali} is a consequence of the following more general result. 
\begin{theorem}\label{thm: projectionformula} Let $F$ be a vector bundle of rank $r$ and $A$ be a line bundle over the smooth quasi-projective surface 
$X$. Consider, over the Hilbert scheme of points $X^{[n]}$,  the rank $nr$ tautological bundle $F^{[n]}$  associated to $F$ and the natural line bundle $\mc{D}_A$ associated to $A$. 
Then, in $\B{D}^b_{\perm_n}(X^n)$ we have: 
$$ \bkrh( (\det F^{[n]} )^{\tens k} \tens \mc{D}_A) \simeq^{qis}  \mc{I}_{\Delta_n}^{rk} \tens ( ( (\det F )^{\tens k} \tens A ) \boxtimes \cdots \boxtimes ( (\det F)^{\tens k} \tens A) ) \tens \epsilon_n^{rk} \;.$$
\end{theorem}
\begin{proof}
{\it Step 1. Case $A$ trivial}. Let's prove the formula for $A$ trivial first. It is sufficient to prove the formula for $k=1$; for arbitrary $k$ it follows by applying the formula for $k=1$ to the sheaf $F^{\prime} = F^{\oplus k}$. Note first that the vanishing 
$$ R^i p_* (q^* \det F^{[n]}) = 0 \qquad \qquad \mbox{for all $i>0$} $$is a consequence of the more general vanishing \cite[Thm. 2.3.1, Cor. 4.13]{Scala2009D}, \cite[Prop. 21]{Scala2009D}: 
$$  R^i p_* (q^*S^\lambda F^{[n]}) = 0 \qquad \qquad \mbox{for all $i>0$} $$for any Schur functor $S^\lambda F^{[n]}$ of any tautological bundle $F^{[n]}$ associated to a vector bundle $F$ on $X$. Hence we just have to prove $$p_* q^* \det F^{[n]} \simeq 
 \mc{I}_{\Delta_n}^{r} \tens (\det F \boxtimes \cdots \boxtimes \det F) \tens \epsilon_n^{r}$$as $\perm_n$-equivariant sheaves on $X^n$. Consider the open set $B^n_{**}$ defined in notation \ref{notat: opensets}; it is the complementary of the closed subscheme $p^{-1}(W)$, which is of \emph{codimension $2$ in $B^n$}. Recall that
we indicate with $j$ both the open immersions $j: B^n_{**} \rInto B^n$ and $j: X^n_{**} \rInto X^n$. The projection $p |_{B^n_{**}}: B^n_{**} \rTo X^n_{**}$ coincides with the smooth blow up of the (now disjoint) pairwise diagonals $\Delta_{ij}$ in $X^n_{**}$; we denote with $E_{ij}$ the irreducible component of the exceptional divisor $E_B$ dominating $\Delta_{ij}$. Over $B^n_{**}$ we have the exact sequence \cite{Danila2001} of $\perm_n$-equivariant sheaves: 
$$ 0 \rTo q^* F^{[n]} \rTo \oplus_{i} p^* F_i \rTo \oplus_{i<j} p^* F_{i} \trest_{E_{ij}} \rTo 0 \;.$$Over $B^n_{**}$ the Weil divisors $E_{ij}$ are Cartier: hence, over $B^n_{**}$ we get 
$$ q^* \det F^{[n]} \simeq \det q^* F^{[n]} \simeq (\det ( \oplus_i p^* F_i )) \tens \det  (\oplus_{i<j} p^* F_{i} \trest_{E_{ij}})^{-1} \;.$$
Let's compute first the second factor; for each sheaf $p^* F_i \trest_{E_{ij}}$, we have, over $B^n_{**}$: 
$$ 0 \rTo p^* F_i (-E_{ij}) \rTo p^* F_ i \rTo p^* F_i  \trest_{E_{ij}} \rTo 0 \;.$$Hence, over $B^n_{**}$
$$\det  ( p^* F_{i} \trest_{E_{ij}})  \simeq (\det p^* F_i ) \tens (\det p^* F_i (-E_{ij}))^{-1} \simeq \FS_{B^n} ( r E_{ij}) $$and 
\begin{equation}\label{eq: secondfactor} \det  (\oplus_{i<j} p^* F_{i} \trest_{E_{ij}})  \simeq \tens_{i<j} \det  ( p^* F_{i} \trest_{E_{ij}})  \simeq \tens_{i<j}  \FS_{B^n} ( r E_{ij}) \simeq \FS_{B^n}(rE_B) \;.\end{equation}As for the first factor, we have, \emph{just as coherent sheaves, without considering the $\perm_n$-action}: 
$$ \det (\oplus_i p^* F_i) \simeq p^* (\tens_i \det F_i) \simeq p^* (\det F \boxtimes \cdots \boxtimes \det F) \;.$$However, 
it is clear that a consecutive transposition $\tau_{i, i+1}$ acts on the sheaf on the left hand side with the sign $(-1)^r$, while it acts trivially on the right hand side: hence, to have an isomorphism \emph{as $\perm_n$-equivariant sheaves} we have to correct the previous formula by the representation $\epsilon_n^r$; that is, as $\perm_n$-sheaves: 
\begin{equation}\label{eq: firstfactor} \det (\oplus_i p^* F_i) \simeq p^* (\det F \boxtimes \cdots \boxtimes \det F) \tens \epsilon_n^r \;. \end{equation}From (\ref{eq: firstfactor}) and (\ref{eq: secondfactor}) we get that, as $\perm_n$-equivariant sheaves, over $B^n_{**}$: 
$$ q^* \det F^{[n]} \simeq p^* (\det F \boxtimes \cdots \boxtimes \det F) \tens \FS_{B^n}(-rE_B) \tens \epsilon_n^r \;.$$Since this is an isomorphism of vector bundles, since $B^n$ is normal and since the complementary of $B^n_{**}$ in $B^n$ is a closed subscheme of codimension $2$, the previous isomorphism extends to the whole variety $B^n$ as an isomorphism of $\perm_n$-equivariant vector bundles. Therefore, by  projection formula: 
$$\B{R} p_* q^* \det F^{[n]} \simeq (\det F \boxtimes \cdots \boxtimes \det F) \tens p_* \FS_{B^n}(-r E_B) \tens \epsilon_n^r \;.$$
To finish the proof we just have to show that $p_*  \FS_{B^n}(-r E_B)  \simeq \mc{I}_{\Delta_n}^r$. Since $ \FS_{B^n}(-r E_B)$ is a line bundle, since $B^n$ is normal and $B^n_{**}$ is the complementary of a closed of codimension $2$, we have $ \FS_{B^n}(-r E_B) \simeq j_* j^*  \FS_{B^n}(-r E_B)$; hence 
\begin{align*} p_*  \FS_{B^n}(-r E_B) \simeq \: & p_*  j_* j^* \FS_{B^n}(-r E_B)  \simeq \:  j_* ( p\trest_{B^n_{**}} )_* j^*   \FS_{B^n}(-r E_B) \\ \simeq \: & j_* ( p\trest_{B^n_{**}} )_*    \FS_{B^n}(-r E_B) \trest_{B^n_{**}} \simeq  j_* j^* \mc{I}_{\Delta_n}^r \simeq \mc{I}_{\Delta_n}^r \end{align*}since, over $X^n_{**}$, $p$ is a smooth blow-up and thanks to lemma \ref{lmm: extensionideali}. 

{\it Step 2. Arbitrary $A$}. If $A$ is non trivial we write: 
 \begin{align*} \bkrh((\det F^{[n]} )^{\tens k} \tens \mc{D}_A ) \simeq & \: \B{R}p_* q^* (  (\det F^{[n]} )^{\tens k} \tens \mc{D}_A )  \simeq \B{R}p_* \big(  q^* (\det F^{[n]} )^{\tens k}  \tens q^*\mc{D}_A \big) \\ \simeq &  \:  \B{R}p_* \big( q^* (\det F^{[n]} )^{\tens k}  \tens p^*(A \boxtimes \cdots \boxtimes A)\big ) \\ 
\simeq &  \:  \big( \B{R}p_*  q^* (\det F^{[n]} )^{\tens k} \big)   \tens (A \boxtimes \cdots \boxtimes A)
\end{align*}where in the third isomorphism we used that $q^* \mc{D}_A \simeq p^*(A \boxtimes \cdots \boxtimes A) $, and in the third we used projection formula. Now the formula follows immediately from the previous case. 
\end{proof}

\begin{crl}\label{crl: projectionformula}Let $F$ be a vector bundle  of rank $r$ and $A$ a line bundle on the smooth quasi-projective surface $X$. Then 
$$  \B{R} \mu_* ( (\det F^{[n]} )^{\tens k} \tens \mc{D}_A ) \simeq \pi_* \Big ( \mc{I}_{\Delta_n}^k  \tens \epsilon_n^k  \Big)^{\perm_n} \tens \mc{D}_{\det F}^{\tens k} \tens \mc{D}_A \;.$$
\end{crl}

\section{Multitors of pairwise diagonals in \protect $X^n$}
Let $X$ be a smooth algebraic variety and let $n \in \mbb{N}$, $n \geq 2$. 
In this section we will study multitors of the form $\Tor_q^{\FS_{X^n}}(\FS_{\Delta_{I_1}}, \dots, \FS_{\Delta_{I_l}})$, where $\Delta_{I_j}$ are pairwise diagonals in $X^n$. This study will be useful  in order to prove the resolutions of 
$\perm_n$-invariants of diagonal ideals of 
subsections \ref{subsection: n=3} and \ref{subsection: n=4} 

\subsection{A general formula for multitors} \label{section: prima}
In this section
we will prove a general formula for multitors $\Tor_q^{\FS_M}(\FS_{Y_1}, \dots, \FS_{Y_l})$ of structural sheaves of smooth subvarieties $Y_i$ of a smooth algebraic variety $M$.

\begin{theorem}\label{thm: multitor}
Let $M$ be a smooth algebraic variety 
 and $Y_1, \dots, Y_l$ be smooth subvarieties of $M$ such that the intersection $Z := Y_1 \cap \dots \cap Y_l$ is smooth. 
 Then: 
\begin{equation}\label{eq: formula} \Tor_q^{\FS_M}(\FS_{Y_1}, \dots, \FS_{Y_l}) = \Lambda^q ( \oplus_{i=1}^l N_{Y_i} \trest_Z / N_Z)^* \;,\end{equation}where 
$N_{Y_i}$ and $N_Z$ denote the normal bundles of $Y_i$ and $Z$ in $M$,  respectively. 
\end{theorem} 
\begin{proof}
Let $x$ be a point of $Z$, and $U$ an affine open neighbourhood of $x$. Restricting $U$ if necessary,  we can find generators 
$f_{j_1}, \dots, f_{j_{c_j}}$ 
of $\mc{I}_{Y_j}(U)$, such that $c_j = \codim Y_j$. It is possible to find them since $Y_j$ is complete intersection in $U$; we can moreover find
 $g_1, \dots, g_c$ generators of $\mc{I}_{Z}(U)$, with $c= \codim Z  \leq \sum_j c_j=d$. Denote simply with $g$ the vector $(g_1, \dots, g_c) \in \FS_M(U)^{\oplus c}$, with $f$ the vector $(f_1, \dots, f_{c_1}, \dots, f_{l_1}, \dots, f_{l_{c_l}}) \in \FS_M(U)^{ \oplus d}$ and with $f_j$ the vector 
 $(f_{j_1}, \cdots, f_{j_{c_j}}) \in \FS_M(U)^{c_j}$. 
  Denote with $E$ the vector bundle $ \FS_M^c$, with $F_j$ the bundle $\FS_M^{c_j}$ and with $F$ the  bundle $ \FS_M^d$. It is clear that, over $U$, $g$ defines a section of $E$, $f_j$ defines a section of $F_j$ and $f$ a section of $F$. 
 We can identify the conormal bundle $N^*_Z$ with the restriction  $E^*\trest_Z$ and 
$N^*_{Y_j}$ with   $ F^*_j\trest_{Y_j}$.

{\it Step 1.} Since all the varieties $Y_i$ and $Z$ are smooth, by the exactness of the conormal sequence, we can identify conormal bundles $N^*_{Y_i}$ and $N^*_{Z}$ with differentials in $\Omega^1_M$ vanishing over $Y_i$ and over $Z$, respectively. 
Since both $f_{j_i}$ and $g_h$ are generators of $I_Z(U)$, there is a $d \times c$-matrix $B \in M_{d \times c}(\FS_M(U))$ and a 
$c \times d$-matrix $B \in M_{c \times d}(\FS_M(U))$ such that $g=Bf$, $f=Ag$. 
This means that, taking differentials: 
$dg = (dB)f + B df$ and $df = (dA)g + A dg$. On points $y \in Z \cap U$ 
we have just: $dg= B df$, $df = A dg$ and $dg = BA dg$. Now, since $Z$, over $U$, is smooth and complete intersection of $Z(g_1), \dots, Z(g_c)$, we have that $[g_i]$ are a basis of $I_Z(U)/I_Z(U)^2$, that is $dg_i$ are a local frame for 
$N^*_{Z}$ over $Z \cap U$ and $df_{j_i}$ are a local frame for $N^*_{Y_j}$. Hence $dg_i$ are linearly indipendent in $\Omega^1_{M} \trest_Z (y) \supseteq N^*_{Z}(y)$ for any $y \in Z\cap U$. 
Now on $\Omega^1_{M} \trest_Z (y)$ we have the relation $dg =  B(y)A(y) dg$, meaning that the matrix $B(y)A(y)$ takes linearly independent into linearly independent, which implies that for any $y \in Z \cap U$, the matrix $B(y)$ is surjective and $A(y)$ is injective. Hence, $B(z)$ and $A(z)$ have to be surjective and injective, respectively, in a
neighbourhood of $x$. Restricting $U$, we can suppose that $A$ and $B$ are injective and surjective, respectively, in any point of $U$.  

{\it Step 2.} Let $E$, $F_i$, $F$, $g$, $f$, $U$ built as in the previous step. 
The matrix $A$ allows to define  an injective morphism of vector bundles over $U$: 
$$ 0 \rTo E \rTo^A F_1 \oplus \cdots \oplus F_l \rTo^{p_Q} Q \rTo 0$$whose cokernel we call $Q$. It is a locally free sheaf on $U$ of rank $d-c$. 
Note that $A$ takes the section $g$ into the section $f$, and hence defines an injective morphism  of pairs $A: (E, g) \rTo (F, f)$. Since we are on an affine open set, the sequence splits; hence we have a morphism $p_E: F \rTo E$ such that $p_E \circ A = \id_E$; the splitting yields an isomorphism: $F \rTo E \oplus Q$, given by $(p_E, p_Q)$. Under this isomorphism the section $f$ of $F$ is carried onto the section $g \oplus 0$ of $Q$, since $(p_E, p_Q)f = p_E f \oplus p_Q f = 
p_E A g \oplus p_Q A g= g \oplus 0$. Hence we have an isomorphism of pairs 
$(F, f) \simeq (E \oplus Q, g \oplus 0)$. 

{\it Step 3.} The previous step yields an isomorphism of Koszul complexes: 
$K^{\bullet}(F, f) \simeq K^{\bullet}(E \oplus Q, g \oplus 0) \simeq 
K^{\bullet}(E, g) \tens K^{\bullet}(Q, 0)$. Note that $$Q\trest_Z \simeq (\oplus_j F_j  / E) \trest_Z \simeq (\oplus_j F_j\trest_Z ) / E\trest_Z  \simeq \oplus_j N_{Y_j} \trest_Z / N_Z \;.$$
Hence: 
\begin{align*}
\Tor_q^{\FS_M}(\FS_{Y_1} , \dots, \FS_{Y_l}) \simeq &\; H^{-q}(K^{\bullet}(F, f)) \simeq H^{-q}(K^{\bullet}(E \oplus Q, g \oplus 0))  \\
\simeq &\; H^{-q}(K^{\bullet}(E, g) \tens K^{\bullet}(Q, 0)) \\
\simeq &\; \bigoplus_{-q=r+s } H^r(K^{\bullet}(E, g)) \tens K^s(Q, 0) \\
\simeq &\; H^0(K^{\bullet}(E, g)) \tens \Lambda^q Q^* \\
\simeq &\; \Lambda^q Q^* \trest_Z \\
\simeq &\; \Lambda^q (\oplus_{i=1}^l N_{Y_i} \trest_Z / N_{Z})^* \;,
\end{align*}where in the fourth isomorphism we used the fact that $K(E \oplus Q, g \oplus 0)$ is a tensor product of $K(E, g)$ with $K(Q, 0)$, which is a complex of locally free sheaves with zero differentials.

{\it Step 4.} We obtained the wanted isomorphism locally. That these local 
isomorphisms glue to a global one is an easy exercise and we leave it to the reader. 
\end{proof}

\begin{notat}Let $k_1, \dots, k_l$ be positive integers. Let $M$ be an algebraic variety and $F_1, \dots, F_l$ coherent sheaves over $M$. We denote with $\Tor^{k_1, \dots, k_l}_q(F_1, \dots, F_l)$ the multitor $\Tor^{\FS_M}_q(F_1, \dots , F_1, \dots, F_l, \dots, F_l)$, where, for all $i$, the factor $F_i$ is repeated $k_i$ times. \end{notat}

\sloppy
The product of symmetric groups $\perm_{k_1} \times \cdots \times \perm_{k_l}$ obvioulsy acts on multitors of the form $\Tor^{k_1, \dots, k_l}_q(F_1, \dots, F_l)$, defined above. Details of this action are described in \cite[Appendix B]{Scala2009D}. Therefore one can study them as $\perm_{k_1} \times \cdots \times \perm_{k_l}$-representations. We have the following. 

\begin{pps}Let $M$ be a smooth algebraic variety and $Y_1, \dots, Y_l$ locally complete intersection subvarieties of $M$ such that 
the intersection $Z = Y_1 \cap \cdots \cap Y_l$ is locally complete intersection. Then for $k_1, \dots, k_l$ positive integers we have, 
as $\perm_{k_1} \times \cdots \times \perm_{k_l}$ representations: 
$$ \Tor_q^{k_1, \dots, k_l}(\FS_{Y_1}, \dots, \FS_{Y_l}) \simeq \bigoplus_{q_1 + q_2 = q} \Tor_{q_1}^{\FS_M}(\FS_{Y_1}, \dots, \FS_{Y_l})
\tens \Lambda^{q_2}( \oplus_{i=1}^l N^*_{Y_i/M} \tens \rho_{k_i}) \;,$$where $\rho_{k_i}$ is the standard representation of the symmetric group $\perm_{k_i}$. 
\end{pps}
\begin{proof}We solve locally, on adequate open affine subsets $U_i$, the structural sheaves $\FS_{Y_i}$ with Koszul complex $\comp{K}(F_i, 
s_i)$, where $F_i$ is a vector bundle of rank $\codim_{U_i} Y_i$ and $s_i$ is a section of $F_i$ transverse to the zero section. 
Then, over $U = U_1 \cap \dots \cap U_l$, we have: 
\begin{align*}
\Tor_{q}^{k_1, \dots, k_l}(\FS_{Y_1}, \dots, \FS_{Y_l}) = H^{-q}(\tens_{i=1}^l \tens_{j=1}^{k_i} \comp{K}(F_i, s_i)) = 
H^{-q}( \tens_{i=1}^l \comp{K}(F_i \tens R_{k_i} , \sigma_{k_i} \tens s_i ))
\end{align*}where $R_{k_i} \simeq \mbb{C}^{k_i}$ is the natural representation of  $\perm_{k_i}$, with canonical basis $e_i$,   and
$\sigma_{k_i}$ is its invariant element, that is $\sigma_{k_i}
= \sum_{h=1}^{k_i} e_h \in R_{k_i}$. Then, since $\comp{K}(F_i \tens R_{k_i} , \sigma_{k_i} \tens s_i )= \comp{K}(F_i \tens \rho_{k_i}, 0) \tens \comp{K}(F_i, 
s_i)$ by \cite[Remark B.5]{Scalaarxiv2015},  we have: 
\begin{align*}
\Tor_{q}^{k_1, \dots, k_l}(\FS_{Y_1}, \dots, \FS_{Y_l})  = & \:
H^{-q}(\tens_{i=1}^l \comp{K}(F_i, s_i) \tens \tens_{i=1}^l \comp{K}(F_i \tens \rho_{k_i}, 0  ) )\\ =& \:
H^{-q}(\tens_{i=1}^l \comp{K}(F_i, s_i) \tens \comp{K}( \oplus_{i=1}^l F_i \tens \rho_{k_i}, 0  ) ) \\
 = & \: \bigoplus_{q_1 + q_2=q}\Tor^{\FS_M}_{q_1}(\FS_{Y_1}, \dots, \FS_{Y_l}) \tens 
\Lambda^{q_2}( \oplus_{i=1}^l N^*_{Y_i/X} \tens \rho_{k_i})
\end{align*}as $\times_{i=1}^l \perm_{k_i}$-representations. Now the open sets of the form $U$ cover the algebraic variety $M$. 
It is an easy exercise to prove that the local isomorphism shown above glue to give a global isomorphism over~$M$. 
\end{proof}

\begin{crl}Let $M$ be a smooth algebraic variety 
 and $Y_1, \dots, Y_l$ smooth subvarieties of $M$ such that the intersection $Z := Y_1 \cap \dots \cap Y_l$ is smooth. Let $k_1, \dots, k_l$ positive integers. Then we have, as $\perm_{k_1} \times \cdots \times \perm_{k_l}$-representations
 $$ \Tor_q^{k_1, \dots, k_l}(\FS_{Y_1}, \dots, \FS_{Y_l}) \simeq \Lambda^q ( [(\oplus_{i=1}^l N_{Y_i} \trest_Z)  / N_Z ]^* \bigoplus \oplus_{i=1}^l N^*_{Y_i} \tens \rho_{k_i} ) \;.$$
\end{crl}

\subsection{Multitors of pairwise diagonals}\label{subsection: multitorspairwise}\label{section: multitorspairwise}
In this subsection we apply the previous general formula to the case of pairwise diagonals. We are concerned with  multitors of the form 
$\Tor_q^{\FS_{X^n}}(\FS_{\Delta_{I_1}}, \dots, \FS_{\Delta_{I_l}})$, where $X$ is a smooth algebraic variety, and where $I_j$ are subsets of $\{1, \dots, n \}$ of cardinality $2$ such that $I_i \neq I_j$ if $i \neq j$. It is therefore convenient to think of the multi-indexes $I_j$ as \emph{edges of a simple graph}. 
We refer to \cite{DiestelGT}  for basic concepts in graph theory. 
More precisely, given $l$ distinct multi-indexes $I_1, \dots, I_l$ of cardinality $2$, we can build the simple graph $\Gamma$, whose set of vertices 
$V_\Gamma$ is defined as the set $I_1 \cup \cdots \cup I_l$ and whose edges are $E_\Gamma = \{ I_1, \dots, I_l \}$. All the vertices of the graph 
$\Gamma$ 
are non isolated, that is, they have degree greater or equal than $1$. On the other hand, given a simple graph $\Gamma$, such that its vertices $V_\Gamma$ are a subset of $\{1, \dots, n \}$ and are non isolated, 
its edges are a set of $l$ distinct  cardinality-$2$ multi-indexes $\{I_1, \dots, I_l \}$ such that $V_\Gamma = I_1 \cup \cdots \cup I_l$. 
For such a  graph $\Gamma$, we denote with $\Delta_\Gamma$ the intersection of  diagonals $\Delta_{I}$, $I \in E_\Gamma$ and with $\Tor_q(\Delta, \Gamma)$ the multitor $\Tor_q^{\FS_{X^n}}(\Delta_{I_1}, \dots, \Delta_{I_l})$. The isomorphism class of this multitor does not depend on the order in which the edges $I_j$ are taken; however, the order of the diagonals is important when dealing with permutation of factors in a multitor: therefore, in the following section, we will always consider $l$-uples of cardinality $2$-multi-indexes $(I_1, \dots, I_l)$, ordered via the lexicographic order. 

\begin{remark}Let $\Gamma$ be a simple graph. Let $v$ the number of vertices, $l$ the number of edges and $k$ the number of connected components. The \emph{number of independent cycles} $c$ of the graph $\Gamma$ is given by $c = l -v + k $.  \end{remark}

\begin{remark}Let $X$ be a smooth algebraic variety of dimension $d$. Let $\Gamma$ be a simple graph without isolated vertices. The subvariety $\Delta_\Gamma$ of $X^n$, intersection of the distinct pairwise diagonals $\Delta_{I}$, $I \in E_\Gamma$ is  smooth of codimension $d(v-k)$, where $v = |V_\Gamma|$. This fact, together the possibility of using formula (\ref{eq: formula}), since all varieties $\Delta_I$, $I \in E_\Gamma$ and $\Delta_\Gamma$ are smooth, allows us to translate properties of the graph $\Gamma$ 
into properties of the multitor $\Tor_q(\Delta, \Gamma)$. In particular, it is clear that 
\begin{itemize} \item 
\emph{the pairwise diagonals $\Delta_I$, $I \in E_\Gamma$, 
intersect transversely in the subvariety $\Delta_\Gamma$} if and only if 
$d(v-k) = \codim_{X^n} \Delta_\Gamma = \sum_{I \in E_\Gamma} \codim_{X^n} \Delta_I = dl$, that is, if and only if $c = l-v+k= 0$, that is, \emph{if and only if the graph $\Gamma$ is acyclic}; in this case $\Tor_q(\Delta, \Gamma) = 0$ for all $q >0$.  
\item  the sheaf $Q_\Gamma :=\big [ \oplus_{I \in E_\Gamma} N_{\Delta_I} \trest_{\Delta_\Gamma} \big] / N_{\Delta_\Gamma}$ is a vector bundle over $\Delta_\Gamma$ of rank $dc$;  hence $\Tor_q(\Delta, \Gamma) =0$ for $q>dc$
\end{itemize}
\end{remark}
For $A \subseteq \{1, \dots, n \}$ denote with $\perm(A)$ the symmetric group of the set $A$ and with $\rho_A$ its the standard representation. Let $\perm_\Gamma :=\Stab_{\perm_n}(\Gamma)$ be the subgroup of $\perm_n$ transforming the graph $\Gamma$ into itself. 
It is a subgroup of $\perm (V_\Gamma) \times \perm(\overline{V_{\Gamma}})$. Indicate with $\widehat{\perm_\Gamma}$ the subgroup 
$\perm (V_\Gamma) \cap \perm_\Gamma$ of $\perm_\Gamma$. 

Suppose now that $\Gamma_1, \dots, \Gamma_k$ are the connected components of the graph $\Gamma$; let now $S_1, \dots, S_t$ be the partition of $\{1, \dots, k \}$ induced by the equivalence relation defined by 
$i \sim j$ if and only if $\Gamma_i$ is isomorphic to $\Gamma_j$. Denote with $\perm_k^\Gamma$ the subgroup 
$\perm(S_1) \times \cdots \times \perm(S_t)$ of $\perm_k$ and with $\widetilde{\perm}_{\Gamma_i} = \Stab_{\perm_n}(\Gamma_i) \cap \perm(V_{\Gamma_i})$, where $\perm(V_{\Gamma_i})$ is naturally seen as a subgroup of $\perm_n$. 
Then there is a split exact sequence 
\begin{equation}\label{eq: split} 1 \rTo \widetilde{\perm}_{\Gamma1} \times \cdots \times \widetilde{\perm}_{\Gamma_k} \rTo \widehat{\perm_\Gamma} \rTo 
\perm_k^\Gamma
\rTo 1 \;.\end{equation}
In other words, the subgroup $\widehat{\perm_\Gamma}$ of the stabilizer $\perm_\Gamma$ is a semi-direct product $(\widetilde{\perm}_{\Gamma1} \times \cdots \times \widetilde{\perm}_{\Gamma_k} ) \rtimes \perm^\Gamma_k$; the proof of this fact is analogous to \cite[Lemma 2.12]{Scalaarxiv2015}.  The full stabilizer $\perm_\Gamma$ is isomorphic to 
$$ \perm_\Gamma \simeq \widehat{\perm_\Gamma} \times \perm(\overline{V_\Gamma}) \simeq  \Big( (\widetilde{\perm}_{\Gamma1} \times \cdots \times \widetilde{\perm}_{\Gamma_k} ) \rtimes \perm^\Gamma_k \Big) \times \perm(\overline{V_\Gamma}) \;.$$

The multitor $\Tor_q(\Delta, \Gamma)$ is naturally $\perm_\Gamma$-linearized. Consider the standard representation $\rho_{V_\Gamma}$ of $\perm(V_\Gamma)$. It can naturally be seen as a $\perm (V_\Gamma) \times \perm(\overline{V_{\Gamma}})$-representation, since the second factor acts trivially on $V_\Gamma$. 
Denote as $\rho_\Gamma := \Res_{\perm_\Gamma} \rho_{V_\Gamma}$ the restriction of $\rho_{V_\Gamma}$ to $\perm_\Gamma$. 

\begin{notat}\label{notat: igamma}If $\Gamma$ is a subgraph of $K_n$  without isolated points and with $k$ connected components $\Gamma_1, \dots, \Gamma_k$, we denote with
 $i_{\Gamma}: X^k \rInto X^{V_{\Gamma_1}} \times \cdots \times X^{V_{\Gamma_k}}$ the immersion defined by embedding 
each factor $X$ in the factor $X^{V_{\Gamma_i}}$ diagonally; for any connected component $\Gamma_i$ we indicate with $p_{\Gamma_i}: X^n \rTo X^{V_{\Gamma_i}}$ the projection onto the factors in $V_{\Gamma_i}$; the morphism 
$p_{\Gamma}: X^n \rTo X^{V_{\Gamma_1}} \times \cdots \times X^{V_{\Gamma_k}}$ 
is defined as $p_\Gamma := p_{\Gamma_1} \times \cdots \times p_{\Gamma_k}$. 
Finally, If $F$ is a sheaf over $X^k$,  
we indicate with $F_\Gamma$ the sheaf over $X^n$ defined as $p_\Gamma^* {i_\Gamma}_* F$. If $\Gamma$ has a single edge, say $E_\Gamma = \{I \}$, we will denote, for brevity's sake $F_\Gamma$ with $F_I$. 
\end{notat}
\begin{notat}We will indicate with $W_\Gamma$ and $q_\Gamma$ the representations of $\perm_\Gamma$ defined by  
$W_\Gamma := \oplus_{I \in E_\Gamma} \rho_I$ and 
 by the exact sequence 
$$ 0 \rTo \rho_\Gamma \rTo W_\Gamma  \rTo q_\Gamma \rTo 0 \;,$$respectively. It is clear that, if $\Gamma_i$ are the connected components of the graph $\Gamma$, the vector space $\oplus_{i=1}^k q_{\Gamma_i}$ is naturally a $\perm_\Gamma$-representation isomorphic to $q_\Gamma$. \end{notat}
\begin{notat}\label{notat: etaIgamma}Let $\Gamma$ be a simple graph without isolated vertices, and let $\gamma$ be an oriented cycle in $\Gamma$. 
If $I$ is an edge of $\gamma$, $I = \{i, j \}$, $i < j$, then we define the sign $\eta_{I, \gamma}$ to be $+1$ if the vertex $j$ is the immediate successor of $i$ in $\gamma$, and $\eta_{I, \gamma} = -1$ otherwise. 
\end{notat}
\begin{remark}\label{rmk: cycles}The representation $q_\Gamma$ is generated over $\mbb{C}$ by independent cycles. More precisely, 
we can consider $q_\Gamma$ as a subrepresentation of $\oplus_{I \in E_\Gamma} R_I$, where $R_I \simeq \mbb{C}^I$, with basis $e_i, i \in I$,  is the natural $\perm(I)$-representation. If $I = \{i, j \}$, $i < j$, let's indicate with $e_I$ the vector $e_j - e_i$. 
If $\gamma$ is an oriented cycle in $\Gamma$, 
then we can consider the vector $e_\gamma :=\sum_{I \in \gamma} \eta_{I, \gamma} e_I$ in 
$W_\Gamma = \oplus_{I \in E_\Gamma} \rho_I \subseteq \oplus_{I \in E_\Gamma} R_I$. Now  $\gamma_1, \dots, \gamma_c$ are independent cycles in $\Gamma$ if and only if $e_{\gamma_1} \dots, e_{\gamma_c}$ are independent in $W_\Gamma$ and project to a basis in $q_\Gamma$. Hence we can identify $q_{\Gamma}$ with the subspace $ \mbb{C} e_{\gamma_1} \oplus \cdots \oplus \mbb{C} e_{\gamma_c}$
of $W_\Gamma$. 
\end{remark}
\begin{remark}Consider the vector bundles 
$\boxplus_{i=1}^k TX \tens q_{\Gamma_i}$ and $\boxplus_{i=1}^k \Omega^1_{X} \tens q_{\Gamma_i}$ over $X^k$. 
 They are naturally $\perm_\Gamma$-equivariant; indeed
the $\perm_\Gamma$ acts on the variety $X^k$ via the surjective composition $\perm_\Gamma \rTo \widehat{\perm_{\Gamma}} \rTo \perm_k^\Gamma$; 
this action lifts to the tangent and cotangent bundle. On the other hand
the action of $\perm_\Gamma$  over the representations $q_{\Gamma_i}$ is induced by its action on $q_\Gamma$. 
\end{remark}

We have the following interpretation of a multitor of pairwise diagonals $\Tor_q(\Delta, \Gamma)$ in terms of data attached to the 
graph $\Gamma$.  
\begin{pps}\label{pps: torgamma}As $\perm_\Gamma$-sheaves, we have that   $ \Tor_q(\Delta, \Gamma) \simeq  
  \Lambda^q(  \boxplus_{i=1}^k \Omega^1_{X} \tens q_{\Gamma_i} )_{\Gamma} \;.$ 
\end{pps}
\begin{proof}Consider first the case in which $\Gamma$ is connected. Then $\Delta_\Gamma \simeq X \times X^{\overline{V_\Gamma}}$. It is then sufficient to prove the proposition when $n = v$, since the case $n >v$ can be then obtained by flat base change. 
The morphism $i_{\Gamma}: X \rTo X^n$ is an isomorphism over the image $\Delta_\Gamma$. Hence pulling back the exact sequence over $\Delta_\Gamma$
$$ 0 \rTo N_{\Delta_\Gamma} \rTo \oplus_{I \in E_\Gamma} N_I \trest_{\Delta_\Gamma}  \rTo Q_\Gamma \rTo 0 $$
 via $i_{\Gamma}$, we obtain over $X$, by definition of $\rho_\Gamma$ and $q_\Gamma$, a $\perm_\Gamma$-equivariant  exact sequence 
 $$ 0 \rTo TX \tens \rho_{\Gamma} \rTo TX \tens W_\Gamma \rTo TX \tens q_\Gamma \rTo 0 \;.$$ 
 Hence the vector bundle $Q_\Gamma = \big [ \oplus_{I \in E_\Gamma} N_{\Delta_I} \trest_{\Delta_\Gamma} \big] / N_{\Delta_\Gamma}$ over $\Delta_\Gamma$  is isomorphic to ${i_{\Gamma}}_* (TX \tens q_\Gamma) = (TX \tens q_\Gamma)_{\Gamma}$ and we conclude by formula (\ref{eq: formula}). 
 
Consider now a general $\Gamma$. We have that 
$\Delta_\Gamma \simeq X^k \times X^{\overline{V_\Gamma}}$: it is then sufficient to consider the case $n = v$ for the same reasone as above. In this case $i_{\Gamma} : X^k \rTo X^n$ is an isomorphism over the image $\Delta_{\Gamma}$. 
Consider the connected components $\Gamma_i$, $i = 1, \dots, k$, of the graph $\Gamma$. 
We have $N_{\Delta_{\Gamma}} \simeq \oplus_{i=1}^k N_{\Delta_{\Gamma_i}} \trest_{\Delta_\Gamma}$, since the partial diagonals $\Delta_{\Gamma_i}$ intersect transversely. For each $i = 1, \dots, k$, we have sequences 
$$ 0 \rTo N_{\Delta_{\Gamma_i}} \rTo \oplus_{I \in E_{\Gamma_i}} N_{I} \trest_{\Delta_{\Gamma_i}} \rTo Q_{\Gamma_i} \rTo 0 \;.$$Hence over 
$\Delta_\Gamma$ we have an exact sequence 
$$ 0 \rTo N_{\Delta_\Gamma} \rTo \oplus_{I \in E_{\Gamma}} N_I \trest_{\Delta_I} \rTo \oplus_{i=1}^k Q_{\Gamma_i} \rTo 0 $$and hence 
pulling everything back to $X^k$ we get the exact sequence 
$$ 0 \rTo \boxplus_{i=1}^k TX \tens \rho_{\Gamma_i}  \rTo \boxplus_{i=1}^k TX \tens W_{\Gamma_i}  \rTo \boxplus_{i=1}^k TX \tens  q_{\Gamma_i} \rTo 0 \;.$$
Hence, as $\perm_\Gamma$-vector bundles over $\Delta_\Gamma$, 
$Q_{\Gamma} \simeq {i_{\Gamma} }_*(\boxplus_{i=1}^k TX \tens q_{\Gamma_i} ) \simeq (\boxplus_{i=1}^k TX \tens q_{\Gamma_i})_{\Gamma}$, and we conclude by formula (\ref{eq: formula}). 
\end{proof}

\begin{notat}Let $\Gamma$ a simple graph without isolated vertices, such that  $V_\Gamma \subseteq \{1, \dots , n \}$. If $J \subseteq \{1, \dots, n \}$, $|J|=2$,  and $J \not \in E_\Gamma$, we will indicate with $\Gamma \cup J$ the graph obtained by $\Gamma$ adding the edge $J$, that is, the graph defined by $V_{\Gamma \cup J} := V_\Gamma \cup J$, $E_{\Gamma \cup J} := E_\Gamma \cup \{ J \}$. 
\end{notat}

\begin{pps}\label{pps: iGamma}Let $X$ a smooth algebraic variety of dimension $d$. Let $\Gamma$ be a simple graph without isolated vertices such that $V_\Gamma \subseteq \{1, \dots, n \}$ and with edges $E_\Gamma = \{I_1, \dots, I_l \}$. Let $J \subseteq \{1, \dots, n \}$, $|J|=2$,  and $J \not \in E_\Gamma$. 
Identifying $\Tor_q(\Delta, \Gamma)$ with $\Tor_q(\Delta_{I_1}, \dots, \Delta_{I_l}, \FS_{X^n})$, 
the $\perm_\Gamma \cap \perm_{\Gamma \cup J}$-equivariant map 
$$ i_{\Gamma, \Gamma \cup J}: \Tor_q(\Delta, \Gamma) \rTo \Tor_q(\Delta, \Gamma \cup J) \;, $$induced by the 
restriction $\FS_{X^n} \rTo \Delta_J$, can be identified with 
\begin{itemize}\item the restriction 
$ \Lambda^q(Q_\Gamma^*) \rTo  \Lambda^q(Q_\Gamma^*) \trest_{\Delta_{\Gamma \cup J}} $
if $J \not 
\subseteq V_\Gamma$
\item  the natural injection
$ \Lambda^q(Q_\Gamma^*) \rTo  \Lambda^q(Q_{\Gamma \cup J}^*) $, induced by the injection of vector bundles $Q^*_\Gamma \rTo Q_{\Gamma \cup J}^*$ over $\Delta_\Gamma$, if $J \subseteq V_\Gamma$. \end{itemize}
\end{pps}
\begin{proof}
Let $x $ be a point in $\Delta_\Gamma$. On a small affine open neighbourhood of $x$, we can find vector bundles $F_I$, $I \in E_\Gamma$, of rank $d$, such that $\Delta_I$ are the zero locus of sections $s_I$ of $F_I$ transverse to the zero section. The same is true for $\Delta_J$. The structural sheaves $\FS_{\Delta_I}$ can then be resolved with Koszul complexes $K^\bullet (F_I, s_I)$.  Denote with $F_\Gamma$ the vector bundle $\oplus_{I \in E_\Gamma} F_I$ and with $s_\Gamma$ the section $\oplus_I s_I$. Therefore 
\begin{gather*} \Tor_q(\Delta, \Gamma) \simeq H^{-q}(\tens_{ I \in E_\Gamma } K^\bullet(F_I, s_I) ) \simeq 
H^{-q}( K^\bullet( F_\Gamma, s_\Gamma) ) \simeq \Lambda^q (Q_\Gamma^*) \\
\Tor_q (\Delta, \Gamma \cup J) 
\simeq 
H^{-q}( K^\bullet( F_{\Gamma \cup J}, s_{\Gamma \cup J} ) ) \simeq 
\Lambda^q (Q^*_{\Gamma \cup J} ) \;.
\end{gather*}
The morphism $\Tor_q (\Delta, \Gamma) {\rTo} \Tor_q (\Delta, \Gamma \cup J)$ of the statement is induced by 
the injection of vector bundles $i_J : F^*_{\Gamma} =  \oplus_{I \in E_\Gamma} F^*_I \rTo \oplus_{I \in E_{\Gamma \cup J}} F^*_I = F^*_{\Gamma \cup J} $; 
one then sees, as in the proof of theorem \ref{thm: multitor} that the injection $i_J$ induces 
the natural map:   
$$Q^*_\Gamma =  [ \oplus_{I \in E_\Gamma} N_{\Delta_I} \trest_{\Delta_\Gamma} 
/N_{\Delta_\Gamma} ]^*
{\rTo}  [ (  \oplus_{I \in E_{\Gamma }}N_{\Delta_I } \oplus 
 N_{\Delta_J}  ) \trest_{\Delta_{\Gamma \cup J} }
/N_{\Delta_{\Gamma \cup J}} ]^* = Q^*_{\Gamma \cup J} \;.$$
Now if $J \not \subseteq V_\Gamma$, we just have that $N_{\Delta \cup J} \simeq N_{\Delta_\Gamma } \oplus N_{\Delta_J}$, and hence
$Q^*_{\Gamma \cup J} \simeq Q^*_{\Gamma} \trest_{\Delta_{\Gamma \cup J}}$ and the previous map is the restriction; on the other hand, if $J  \subseteq V_\Gamma$, $Q^*_{\Gamma \cup J}$ is a vector bundle over $
\Delta_{\Gamma \cup J} = \Delta_\Gamma$ and the previous map is the natural injection. This proves the statement.  
\end{proof}

\begin{pps}\label{pps: qKn} Let $K_v$ the complete graph on $v$ vertices and suppose that $V_{K_v} \subseteq \{1, \dots, n \}$. Then, as $\perm_{K_v}$-representations, 
$$ \Tor_q(\Delta_, K_v) \simeq \Lambda^q(\Omega^1_X \tens \Lambda^2 \rho_v)_{K_v} \;,$$where $\rho_v$ denotes the standard representation of $\widehat{\perm_{K_v}} \simeq \perm_v$. 
\end{pps}
\begin{proof}
It is sufficient to consider the case $v = n$, since the case $v <n$ follows by flat base change. 
By propositon \ref{pps: torgamma}, it is sufficient to prove that the representation $q_{K_n}$ is isomorphic to $\Lambda^2 \rho_n$. 
Since $\rho_{K_n} \simeq \rho_n$, it is sufficient to prove that the representation 
$W_{K_n} = \oplus_{|I|=2, \; I \subseteq \{1, \dots, n \} } \rho_I \simeq \rho_n \oplus \Lambda^2 \rho_n$. In order to achieve this, it is sufficient to prove that the characters of the two representations are the same. Let $\bold i$ the $n$-uple $(i_1, \dots, i_n)$, with $\sum_j ji_j =n$. 
Denote with $C_{\bold i}$ the conjugacy class in $\perm_n$ of permutations 
having $i_j$ $j$-cycles. By Frobenius formula 
\cite[exercise 4.15]{FultonHarrisRT}, the character of $\rho_n \oplus \Lambda^2 \rho_n$
is valued, on $C_{\bold i}$: 
$$( \chi_{\rho_n} + \chi_{\Lambda^2 \rho_n})
(C_{\bold i})= i_1 -1 + \frac{1}{2}(i_1 -1)(i_1-2) -i_2 = \binom{i_1}{2} - i_2 \;.$$On the other hand, 
a basis of the representation $W_{K_n}$ is given by vectors $e_J$, $J \subseteq \{1, \dots, n \}$, $|J| =2$.
 If $I = \{i, j \}$, we have that $(ij) e_I = -e_I$ and $(ij) e_{ih}= e_{jh}$ if $h \not \in I$. Hence, 
any cycle $\gamma_j$ of length $j \geq 3$ will act with trace zero. The $1$-cycles 
$(j_1)\dots(j_{i_1})$ act trivially on a $\binom{i_1}{2}$-dimensional space, and hence with  trace $\binom{i_1}{2}$. Since the cycles are disjoint, the traces add up. Hence 
$$ \chi_{W_{K_n}}{(C_{\bold i})} = \binom{i_1}{2}-i_2 = (\chi_{\rho_n} + \chi_{\Lambda^2 \rho_n})(C_{\bold i}) \;.$$\end{proof}

\subsection{The cases \protect $n=3, 4$. }We analyse here in detail the representations $q_\Gamma$ with $\Gamma$ a non-acyclic  subgraph of  $K_n$, $n = 3, 4$, with non isolated  vertices. 
\paragraph{Non acyclic graphs.} There is a unique non-acyclic subgraph of $K_3$ without isolated vertices, the complete graph $K_3$ itself. 
On the other hand, up to isomorphism, the non-acyclic subgraphs of $K_4$ without isolated vertices are the following:
\begin{itemize}\item for $l = 3$, the complete graph $K_3$ with $3$ vertices; its stabilizer is the group $\perm_3$; 
\item for $l=4$, the $4$-cycle $C_4$ and the graph $K_3 \cup J$, obtained by $K_3$ adding an edge $J \not \subseteq E_{K_3}$; 
since $C_4$ has dihedral symmetry,  $\perm_{C_4}$ is isomorphis to the dihedral group $D_4$; on the other hand, adding an edge to 
the graph $K_3$ reduces its symmetries to $\perm_2$. 
\item for $l = 5$ the graph $C_4 \cup L$, obtained adding an edge $L \not \subseteq E_{C_4}$ to the $4$-cycle $C_4$; its stabilizer is 
$\perm_{C_4 \cup J} \simeq \perm_2 \times \perm_2$; 
\item for $l = 6$, the complete graph $K_4$; its stabilizer is the full symmetric group $\perm_4$. 
\end{itemize}
The non-acyclic subgraphs of $K_4$ without isolated vertices  are represented in the figure below.

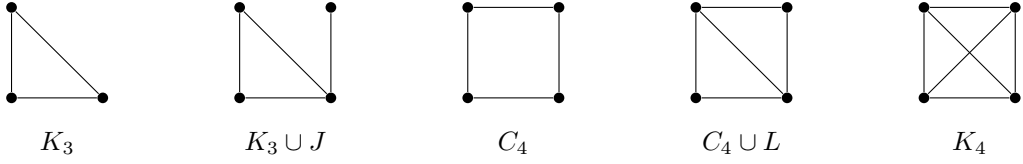
\vspace{1cm}
\begin{center}
\begin{tikzpicture}
  [-, -=stealth, scale=.3,auto=left]
  \tikzstyle{F}=[circle, minimum width=4pt, fill, 
  inner sep=0pt
  ]; 
  \node[F] (n1) at (-2,0) {}; 
  \node[F] (n2) at (2,0) {}; 
  \node[F] (n3) at (-2,4){};

\node at (0,-2) {$K_3$};
  \foreach \from/\to in {n1/n2, n2/n3, n3/n1}
    \draw (\from) -- (\to);

  \node[F] (m1) at (8,0) {}; 
  \node[F] (m2) at (12,0) {}; 
  \node[F] (m3) at (8,4){}; 
\node[F] (m4) at (12, 4) {}; 

\node at (10, -2) {$K_3 \cup J$};
  \foreach \from/\to in {m1/m2, m2/m3, m3/m1, m2/m4}
    \draw (\from) -- (\to);
    
     \node[F] (r1) at (18,0) {}; 
  \node[F] (r2) at (22,0) {}; 
  \node[F] (r3) at (22,4){}; 
\node[F] (r4) at (18, 4) {}; 

\node at (20, -2) {$C_4$};
  \foreach \from/\to in {r1/r2, r2/r3, r3/r4, r4/r1}
    \draw (\from) -- (\to);

   \node[F] (s1) at (28,0) {}; 
  \node[F] (s2) at (32,0) {}; 
  \node[F] (s3) at (32,4){}; 
\node[F] (s4) at (28, 4) {}; 

\node at (30, -2) {$C_4 \cup L$};
  \foreach \from/\to in {s1/s2, s2/s3, s3/s4, s4/s1, s4/s2}
    \draw (\from) -- (\to);
    
     \node[F] (t1) at (38,0) {}; 
  \node[F] (t2) at (42,0) {}; 
  \node[F] (t3) at (42,4){}; 
\node[F] (t4) at (38, 4) {}; 

\node at (40, -2) {$K_4$};
  \foreach \from/\to in {t1/t2, t2/t3, t3/t4, t4/t1, t4/t2, t3/t1}
    \draw (\from) -- (\to);

\end{tikzpicture}
\end{center}

\paragraph{Classification of representations $q_\Gamma$.}
In order to classify the representations $q_\Gamma$ we need the following lemma
\begin{lemma}\label{lmm: reduction} Let $n = r+s$, with $r, s \in \mbb{N} \setminus \{0 \}$.  Consider the symmetric groups $\perm_{r+s}$, $\perm_r$, 
$\perm_s$.  Then, as $\perm_r \times \perm_s$-representations, 
$ \rho_{r+s} \simeq \rho_r \oplus \rho_s \oplus 1 \;.$ As a consequence, if $n= \sum_{i=1}^r k_i$, as $ \perm_{k_1} \times \cdots \times \perm_{k_r}$-representations,  we have
$ \rho_{k_1 + \dots + k_r} \simeq \rho_{k_1} \oplus \dots \oplus \rho_{k_r} \oplus 1^{r-1} $, where $1^{r-1}$ is the $(r-1)$-dimensional trivial representation. \end{lemma} 
\begin{proof}The natural $\perm_n$-representation $R_n$ splits as $R_n \simeq R_r \oplus R_s$ 
when seen as 
as $\perm_r \times \perm_s$-represenation. The inclusion $\rho_{r+s} \rInto R_n = R_r \oplus R_{s}$ is $\perm_{r} \times \perm_s$-equivariant. Then, as $\perm_r \times \perm_s$-representations
$R_n \simeq   R_r \oplus R_s = (\rho_r \oplus 1) \oplus (\rho_s \oplus 1) \simeq \rho_{r} \oplus \rho_s \oplus 1^2$. 
Hence the statement: $\rho_{r+s} \simeq \rho_r \oplus \rho_s \oplus 1$.\end{proof}

\begin{remark}The dihedral group $D_4$, generated by the reflection $\sigma$ and the rotation $\rho$, 
has 4 finite dimensional irreducible representations: the trivial, 
the standard representation $\theta$, 
the determinantal $\det := \det \theta$, the linear $\ell(1_\sigma, -1_\rho)$, 
the linear $\ell(-1_\sigma, -1_\rho)$. 
 A description and a character table for these representation 
is given in \cite{SerreLRFG}. 
\end{remark}

We already know, by proposition \ref{pps: qKn}, that $q_{K_3} \simeq \Lambda^2 \rho_3 \simeq \epsilon$ and that $q_{K_4} \simeq \Lambda^2 \rho_4 \simeq \rho_4 \tens \epsilon$. 
As for the remaining representations $q_\Gamma$ for $\Gamma$ a subgraph of $K_4$ without isolated vertices we have the following. 
\begin{itemize}\item $\Gamma = K_3 \cup J$. Up to isomorphism we can think that $E_{K_3 \cup J} = \{ \{1, 2 \}, \{1, 3 \}, \{2, 3 \}, \{3,4 \} \}$. 
Hence $\perm_{K_3 \cup J} \simeq \perm( \{1, 2 \})$ and $W_{K_3 \cup J} = \rho_{12} \oplus \rho_{13} \oplus \rho_{23} \oplus \rho_{34}$. 
The character $\chi_{W_{K_3 \cup J}}$ is easily $(4,0)$ according to 
the conjugacy classes of $1, (12)$; but $(4,0) = (1, -1) + (1, -1) + (1, 0) + (1, 0) = 2 \chi_{\epsilon} + 2\chi_{1}$.  
Hence $\chi_{q_{K_3 \cup J}}  = \chi_{W_{K_3 \cup J}} - \chi_{\rho_{K_3 \cup J}} = 2 \chi_{\epsilon} + 2\chi_{1} - \chi_{\Res^{\perm_4}_{\perm_2} \rho_{4}} = 2 \chi_{\epsilon} + 2\chi_{1} - \chi_\epsilon - 2 \chi_{1} = \chi_{\epsilon}$ by lemma \ref{lmm: reduction}. 
Hence $q_{K_3 \cup J} \simeq \epsilon$. 

\item $\Gamma = C_4$. Up to isomorphism, we can think thak $E_{C_4} = \{ \{1,2 \}, \{1, 4\}, \{2, 3 \}, \{3, 4 \} \}$. We easily have that the stabilizer $\perm_{C_4}$ is isomorphic to the dihedral group $D_4$, where the reflection $\sigma$ and the rotation $\rho$ are identified with 
 $\sigma = (24)$ and $\rho = (1234)$, for example. Then $\chi_{W_{C_4}}$, according to conjugacy classes, $1, \sigma, \sigma \rho, 
 \rho, \rho^2$, is given by $\chi_{W_{C_4}} = (4, 0, -2, 0, 0)$, and hence $W_{C_4}$ is isomorphic to $\det \oplus \ell(1_\sigma, -1_\rho) \oplus \theta$. Computing characters we get that $\chi_{\rho_{C_4}} = (3, 1, -1, -1, -1)$ and hence $\rho_{C_4}$ is isomorphic to $\ell(1_\sigma, -1_\rho) \oplus \theta$ as 
 $D_4$-representation. Hence $q_{C_4} \simeq \det $ as $D_4$-representation. 
 
 \item $\Gamma = C_4 \cup L$.  Up to isomorphism, suppose that the graph $C_4 \cup L$ has edges 
$E_{C_4 \cup L} = \{ \{1,2\}, \{1, 4 \}, \{2,3\}, \{3,4\},  \{1,3 \} \}$, so that $\perm_{C_4 \cup L} \simeq \perm(1,3) \times \perm(2,4)$. 
Then, according to conjugacy classes $1, (13), (24), (13)(24)$, the character $\chi_{W_{C_4 \cup L}}$ is given by $(5, -1, 1, -1)$. 
By lemma \ref{lmm: reduction}, $\rho_{C_4 \cup L} \simeq \Res^{\perm_4}_{\perm_2 \times \perm_2} \rho_{5} \simeq 
(\epsilon \tens 1) \oplus (1 \tens \epsilon) \oplus 1$ and hence $\chi_{\rho_{C_4 \cup L}} = (3, 1, 1, -1)$. 
Hence $\chi_{q_{C_4 \cup L}} = (2, -2, 0, 0)$, which yields $q_{C_4 \cup L} \simeq (\epsilon \tens 1) \oplus (\epsilon \tens \epsilon)$. 
\end{itemize} 
\paragraph{Multitors as $\perm_\Gamma$-representations.} For $n =3, 4$, all non-acyclic graph are connected, hence 
the corresponding $\Tor_q(\Delta, \Gamma)$ is isomorphic to $\Lambda^q (\Omega^1_X \tens q_\Gamma)_\Gamma$. 
We decompose the exterior power $\Lambda^q (\Omega^1_X \tens q_\Gamma)$ according to the Schur-functor decomposition 
$$  \Lambda^q (\Omega^1_X \tens q_\Gamma ) = \bigoplus_{\lambda} S^\lambda \Omega^1_X \tens S^{\lambda^\prime} q_\Gamma $$where the direct sum is taken on partitions $\lambda$ of $q$ such that $\lambda$ has at most $\dim X$ rows and at most $\dim q_\Gamma$ columns \cite[Exercise 6.11]{FultonHarrisRT}. In the next section it will be important to determine $\perm_\Gamma$-invariants 
of the previous exterior power. We have the following 
\begin{lemma}Let $n = 4$,  let $\Gamma_1$ be a graph of the kind $C_4 \cup L$ and let $K_4$ the complete graph with $4$-vertices. If 
$S^\lambda q_{\Gamma_1}$ has nontrivial $\perm_{\Gamma_1}$-invariants, or if 
$S^\lambda q_{K_4}$ has nontrivial  $\perm_4$--invariants, then the partition $\lambda$ appears in the following table, which indicates as well the dimension of the space of invariants. 

\vspace{0.3cm}

\begin{table}[H]

\centering
{\renewcommand{\arraystretch}{1.35}    
\begin{tabular}{cc|cccccccccc}
   &                                                                                                 & \multicolumn{10}{c}{$\lambda$} \\ \cline{3-12}
   &                                                                                                 & (2) & (3) & (4) & (3,1) & (2,2) & (3,1,1) & (6) &  (5,1) & (4,2) &  (2,2,2) \\ \hline
\multicolumn{1}{c|}{ \multirow{2}{*}{$\dim (S^\lambda q_\Gamma)^{\perm_\Gamma}$}  }&  $ \Gamma_1$ &       2 & 0          & 3   &   1    &  1     & 0          & 4   &   2      &   2     &   0 \\
                                 \multicolumn{1}{c|}{}                                                                             &  $K_4$             &1    & 1&2 &0 &1 &1 & 3& 1&2 &1
                                                                          \end{tabular}}
\caption{} \label{table: qgamma}
\end{table}
\end{lemma}
\begin{proof}The proof is straightforward in the case of the graph $\Gamma_1$. In the case of the complete graph $ K_4$, 
we decompose
$S^\lambda (q_{K_4}) \simeq S^\lambda (\rho_4 \tens \epsilon) \simeq S^\lambda V^{2,1,1}$  into irreducible $\perm_4$-representations
via 
the script in {\tt GAP} \cite{GAP4} indicated in \cite{moflw96705}. 
\end{proof}

\section{Invariants of diagonal ideals for low \protect $n$.}
Let $X$ be a smooth algebraic variety. The ideal $\mc{I}_{\Delta_n}$ of the big diagonal $\Delta_n$  is the intersection 
$\mc{I}_{\Delta_n} = \cap_{I \subseteq \{1, \dots, n \}, |I| =2} \mc{I}_{\Delta_I}$ of ideals of pairwise diagonals $\Delta_I$, $I \subseteq \{1, \dots, n \}$, $|I|=2$: it is then isomorphic to 
the kernel of the natural morphism 
\begin{equation}\label{eq: morfismo} \FS_{X^n} \rTo \bigoplus_{I \subseteq \{1, \dots, n \}, |I| =2} \FS_{\Delta_I} \;.\end{equation}Hence it is useful to consider, for each multi-index 
$I$ of cardinality $2$, a right resolution $\comp{\mc{K}}_I$ of the ideal sheaf $\mc{I}_{\Delta_I}$: 
$$ \comp{\mc{K}}_I: \qquad 0 \rTo \FS_{X^n} \rTo \FS_{\Delta_I} \rTo 0 \;,$$
concentrated in degree $0$ and $1$: 
indeed, the first nontrivial map of the $\perm_n$-equivariant complex  
$ \Tens_{I } \ccomp{K}_I$ is indeed exactly (\ref{eq: morfismo}); here the order in which the tensor product is taken is always the lexicographic order on the cardinality $2$-multi-indexes.  However, the complex $\Tens_{I } \ccomp{K}_I$ is not exact, and in order to 
to deal with this problem, it is better to consider the derived tensor product  of complexes $$ \Tens^L_I \ccomp{K}_I \;.$$
Indeed, let $r=n(n+1)/2$ 
and consider the spectral sequence \begin{equation}\label{eq: spectralsequence} E^{p,q}_1 := \bigoplus_{i_1 +  \cdots + i_{r}=p} 
\Tor_{-q}(\mc{K}_{1,2}^{i_1}, \cdots, \mc{K}_{n_1, n}^{i_r}) \;. \end{equation}It abuts to the cohomology $\mc{H}^{p+q}(\Tens^L_I \ccomp{K}_I)$ and 
the term $E^{0,0}_2$ is clearly isomorphic to the ideal $\mc{I}_{\Delta_n}$. 
\begin{remark}\label{rmk: abut} 
Since $\mc{I}_{\Delta_I}$ are sheaves and the complexes $\ccomp{K}_{I}$ are their resolutions, $\mc{H}^{p+q}(\Tens^L_I \ccomp{K}_I) = \Tor_{-p-q}(\mc{I}_{\Delta_{12}}, \cdots, \mc{I}_{\Delta_{n-1,n}})=0$ for $p+q > 0$. Consequently, the abutment 
of the spectral sequence  is zero for $p+q>0$. 
\end{remark}
We plan to get information on the sheaf of 
invariants $\mc{I}^{\perm_n}_{\Delta_n} = ( E^{0,0}_2)^{\perm_n}$ from 
the vanishing of the abutment in positive degree and from studying the spectral sequence of invariants in detail.

\subsection{The comprehensive \protect $\perm_n$-action on the spectral sequence $E^{p,q}_1$.} We will here briefly explain the action of the symmetric group $\perm_n$ on the derived tensor product 
$\Tens^L_I \ccomp{K}_I $ or, equivalently, on the spectral sequence $E^{p,q}_1$. This is analogous to what done 
in \cite{Scala2009D} for the derived tensor power of a complex  of sheaves $\mc{C}^\bullet$. The point here is that, when considering 
the multitors $\Tor_{-q}(\mc{K}_{1,2}^{i_1}, \cdots, \mc{K}_{n_1, n}^{i_r})$, the terms $\mc{K}_{I}^{h}$ are not just sheaves, but 
terms of a complex $\ccomp{K}_I$. In the following remark we will recall what we explained in detail in 
\cite[Section 4.1]{Scala2009D} and 
\cite[Appendix B]{Scala2009D}. 
\begin{remark}\label{rmk: permutative}Let $\mc{C}^\bullet_1$, $\mc{C}^\bullet_2$ complexes of sheaves over a variety $M$.  
If we have a tensor product of complexes of sheaves
$\ccomp{C}_1 \tens \ccomp{C}_{2}$ the permutation of factors $\tau_{12}: \ccomp{C}_1 \tens \ccomp{C}_{2} \rTo \ccomp{C}_2 \tens \ccomp{C}_{1}$ is a morphism of complexes if and only if $\tau_{12}$ acts 
on the term of degree $h$, that is the term  $(\ccomp{C}_1 \tens \ccomp{C}_{2})^h :=\oplus_{i+j = h} \mc{C}^i_1 \tens \mc{C}^j_2$,  exchanging the terms $\mc{C}^{i}_1 \tens \mc{C}^j_2 \rTo \mc{C}_2^j \tens \mc{C}^{i}_1$ and \emph{twisting by the sign} $(-1)^{ij}$. 
The same argument can be applied to a tensor product of complexes $\ccomp{C}_1 \tens \cdots \tens \ccomp{C}_{r}$. 
Indeed, in order to understand how a general permutation of factors operate on a tensor product of complexes, it is sufficient to understand 
 how a consecutive transposition acts, and this is completely analogous to the case $r=2$.

If now we want to understand the effect of permutating factors  in a \emph{derived tensor product} $\mc{C}^\bullet_1 \tens^L \cdots \tens^L \mc{C}^{\bullet}_r$, we have to resolve each of the complexes $\ccomp{C}_j$ with a complex of locally free $\comp{R}_j$ (at least locally), and apply the previous reasoning to $\comp{R}_1 \tens \cdots \tens \comp{R}_r$. To be more explicit, at the level of spectral sequences, consider a consecutive transposition $\tau_{j, j+1} \in \perm_r$ and consider the spectral sequences 
\begin{gather*}
E^{p,q}_1 = \oplus_{i_1 + \cdots + i_r=p} \Tor_{-q}(\mc{C}^{i_1}_1, \dots, \mc{C}^{i_j}_j, \mc{C}^{i_{j+1}}_{j+1}, \dots, \mc{C}^{i_r}_r) \\ 
{E^{\prime}}^{p,q}_1 = \oplus_{h_1 + \cdots + h_r=p} \Tor_{-q}(\mc{C}^{h_1}_1, \dots, \mc{C}^{h_{j+1}}_{j+1}, \mc{C}^{h_j}_j,  \dots, \mc{C}^{h_r}_r) \;,
\end{gather*}abutting to $\mc{H}^{p+q}(\mc{C}^\bullet_1 \tens^L \cdots \tens^L \mc{C}^\bullet_j \tens^L \mc{C}^\bullet_{j+1} \tens^L  \cdots \tens^L \mc{C}^\bullet_r )$ and $\mc{H}^{p+q}(\mc{C}^\bullet_1 \tens^L \cdots \tens^L \mc{C}^\bullet_{j+1} \tens^L \mc{C}^\bullet_{j} \tens^L \cdots \tens^L  \mc{C}^\bullet_r )$. The consecutive transposition $\tau_{j, j+1}$ induces an isomorphism $E^{p,q}_1 {\rTo} {E^{\prime}} ^{p,q}_1$ and hence isomorphisms 
$$ \widehat{\tau_{j, j+1}}: \Tor_{-q}(\mc{C}^{i_1}_1, \dots, \mc{C}^{i_j}_j, \mc{C}^{i_{j+1}}_{j+1}, \dots, \mc{C}^{i_r}_r) \rTo \Tor_{-q}(\mc{C}^{i_1}_1, \dots, \mc{C}^{i_{j+1}}_{j+1},  \mc{C}^{i_j}_j,  \dots, \mc{C}^{i_r}_r) \;.$$Now, considering the $\mc{C}^{i_l}_{l}$ just sheaves, and not as terms of a complex $\mc{C}^\bullet_{l}$, one has the standard permutation of factors in a multitor 
$$ \widetilde{\tau_{j, j+1}}: \Tor_{-q}(\mc{C}^{i_1}_1, \dots, \mc{C}^{i_j}_j, \mc{C}^{i_{j+1}}_{j+1}, \dots, \mc{C}^{i_r}_r) \rTo \Tor_{-q}(\mc{C}^{i_1}_1, \dots, \mc{C}^{i_{j+1}}_{j+1},  \mc{C}^{i_j}_j,  \dots, \mc{C}^{i_r}_r) \;.$$We proved in \cite[section 4.1]{Scala2009D} that 
the two isomorphisms $\widetilde{\tau_{j, j+1}}$ and $\widehat{\tau_{j, j+1}}$ are related by the sign
\begin{equation}\label{eq: tautau} \widehat{\tau_{j, j+1} }= (-1)^{i_j i_{j+1}} \widetilde{\tau_{j, j+1}} \;. \end{equation}
\end{remark}

\begin{remark}\label{rmk: comprehensive}When we say that the complex $\mc{K}^\bullet_{1,2} \tens^L \cdots \tens^L \mc{K}^\bullet_{n-1, n}$ 
is $\perm_n$-equivariant, what we mean is that \emph{$\perm_n$ acts  up to permutation of factors}. More precisely, 
we can interpret 
the $\perm_n$-action in the following way. For brevity's sake, denote with $m = n(n+1)/2$ and set $E_{K_n} := \{J_1, \dots, J_m \}$ in lexicographic order. Consider the  monomorphism $\theta: \perm_n \rInto \perm(E_{K_n})$ induced by the natural action of $\perm_n$ on $E_{K_n}$: 
if $\sigma \in \perm_n$, we will briefly indicate with $\tilde{\sigma}$ its image in $\perm(E_{K_n})$. Denote with $\Delta_{\perm_n}$ the subgroup 
of $\perm_n \times \perm(E_{K_n})^{\mathrm{op}}$ given by the set of couples $(\sigma, \tilde{\sigma}^{-1})$, such that $\sigma \in \perm_n$. 
The complex 
$\mc{K}^\bullet_{J_1} \tens^L \cdots \tens^L \mc{K}^\bullet_{J_m}$ is now \emph{$\Delta_{\perm_n}$-equivariant}; any element $(\sigma, \tilde{\sigma}^{-1}) \in \Delta_{\perm_n}$ acts via a composition 
$$ \mc{K}^\bullet_{J_1} \tens^L \cdots \tens^L \mc{K}^\bullet_{J_m}  { \rTo^{\lambda_{\sigma}}} 
\sigma^* \big( \mc{K}^\bullet_{\sigma (J_1)} \tens^L \cdots \tens^L \mc{K}^\bullet_{\sigma (J_m)} \big) { \rTo^{\lambda_{\tilde{\sigma}^{-1}} }  } \sigma^* \big( \mc{K}^\bullet_{J_1} \tens^L \cdots \tens^L \mc{K}^\bullet_{J_m} \big)  $$
where $\lambda_{\sigma}$ is the \emph{geometric action} and is induced by isomorphisms 
$$ \mc{K}^\bullet_{J_1} \tens^L \cdots \tens^L \mc{K}^\bullet_{J_m} \rTo \sigma^* \mc{K}^\bullet_{\sigma(J_1)} \tens^L \cdots \tens^L  \sigma^* \mc{K}^\bullet_{\sigma( J_m) }
\simeq \sigma^* \big( \mc{K}^\bullet_{\sigma( J_1) } \tens^L \cdots \tens^L \mc{K}^\bullet_{\sigma(J_m)} \big) \;,
   $$while 
 $\lambda_{\tilde{\sigma}^{-1}}$ is the \emph{permutation of factors} induced by $\tilde{\sigma}^{-1}$: it operates  the same way as the 
 permutations described in remark \ref{rmk: permutative}.  At the level of the spectral sequence (\ref{eq: spectralsequence}) and in the identification of $\Tor_{-q}(\mc{K}^{i_1}_{J_1}, \dots, \mc{K}^{i_m}_{J_m})$ with $\Tor_{-q}(\Delta, \Gamma)$ for some 
 subgraph $\Gamma$ of $K_n$ without isolated vertices, 
  the action of an element $(\sigma, \tilde{\sigma})$ is expressed through compositions of isomorphisms 
  $$ \Tor_{-q}(\Delta, \Gamma) \rTo^{\lambda_{\sigma}} \sigma^* \Tor_{-q}(\Delta, \sigma(\Gamma)) \rTo^{\lambda_{\tilde{\sigma}^{-1}}} \sigma^* \Tor_{-q}(\Delta, \sigma( \Gamma) ) $$where $\lambda_\sigma$ is described by the geometric action seen throughout subsection \ref{subsection: multitorspairwise} and where $\lambda_{\tilde{\sigma}^{-1}}$ is the derived permutative action described in remark \ref{rmk: permutative}: in other words, after formula \ref{eq: tautau}, $\lambda_{\tilde{\sigma}^{-1}}$ operates with the sign $\epsilon_{E_{\sigma(\Gamma)}}(\tilde{\sigma}^{-1})=\epsilon_{E_{\sigma(\Gamma)}}(\tilde{\sigma})$, where 
 $\tilde{\sigma}$ is naturally seen in $\perm(E_{\sigma(\Gamma)})$ and where $\epsilon_{E_{\sigma(\Gamma)}}$ is the alternating representation of $\perm(E_{\sigma(\Gamma)})$. In particular, if $\sigma \in \perm_\Gamma$, for the comprehensive $\perm_\Gamma$ action, we have the isomorphism 
 of $\perm_\Gamma$-representations: 
 $$  \Tor_{-q}(\Delta, \Gamma) \simeq \Lambda^q(\boxplus_{i=1}^k \Omega^1_X \tens q_{\Gamma_i})_{\Gamma_i} \tens \Res_{\perm_\Gamma} \epsilon_{E_\Gamma} \;.$$From now on, we wil omit talking about the group $\Delta_{\perm_n}$ and, for brevity's sake, 
when considering the $\perm_n$-action on the derived tensor product $\mc{K}^\bullet_{1,2} \tens^L \cdots \tens^L \mc{K}^\bullet_{n-1, n}$, we will always tacitly intend the $\Delta_{\perm_n}$-action explained here above. 
\end{remark}

\begin{remark}\label{rmk: graphs}Denote with $\mc{G}_{l,n}$ the set of subgraphs of the complete graph $K_n$ without isolated vertices and with $l$ edges. 
We can form $\perm_n$-equivariant complexes of $\FS_{X^n}$-modules $( \comp{\varGamma}_q, \partial^\bullet ) $ on $X^n$ by 
setting $\varGamma_0^0 := \FS_{X^n}$, $\Gamma_q^p := \oplus_{\Gamma \in \Graphs_{p,n}} \Lambda^q(Q^*_\Gamma) $ and where 
the differential $\partial^p: \varGamma_q^p \rTo \varGamma_q^{p+1}$ is defined, over the component $\Lambda^q(Q^*_{\Gamma^\prime})$, 
$\Gamma^\prime \in \Graphs_{p+1, n}$, as the alternating sum:  
$$ \partial^p ( x)_{_{\Gamma^{\prime}}} = \sum_{\substack{\Gamma \in \Graphs_{p, n} \\ \Gamma \subseteq \Gamma^\prime}} 
\epsilon_{\Gamma, \Gamma^\prime} i_{\Gamma, \Gamma^\prime} (x)$$over the subgraphs of $\Gamma^\prime$ with $p$-edges of 
the inclusions $i_{\Gamma, \Gamma^\prime}$. The sign $\epsilon_{\Gamma, \Gamma^\prime}$ is defined as $(-1)^{a-1}$, where $a$ is 
the position in $\Gamma^\prime$ --- according to the lexicographic order --- of the only edge in $\Gamma^\prime$ which is not in $\Gamma$. It is now immediate, using proposition \ref{pps: iGamma},  to show that the complexes $\gammac_q$ are $\perm_n$-equivariant and isomorphic to $E^{\bullet, q}_1$. The complexes $\gammac_q$ are not exact in general, as we will see in the sequel; however, they seem to arise in a pretty natural way as combinatorial objects, without the need to be linked to multitors; 
they might have an interest on their own. On the other hand it seems difficult to describe a general pattern for their cohomology. 
\end{remark}

\begin{notat}\label{notat: 3cycle}In what follows, if $H \subseteq \{1, \dots, n \}$ is a cardinality $3$ multi-index, we will indicate with $K_3(H)$ the complete graph with vertices in $H$, which is a $3$-cycle. Sometimes, for brevity's sake, and when there is no risk of confusion, we will indicate this $3$-cycle directly with $H$, instead of $K_3(H)$. 
\end{notat}

\begin{notat}\label{notat: omegadelta}For $r , j \in \mbb{N}$,  $j \geq 1$, we will write the sheaf $(\Lambda^j \Omega^1_X \boxtimes \FS_{X^{n-3}}) \tens \mc{I}^r_{\Delta_{n-2}}$ over the variety $X \times X^{n-3}$ just with $\Lambda^j \Omega^1_X \boxtimes \mc{I}^r_{\Delta_{n-2}}$, or with $\Lambda^j \Omega^1_X(-r \Delta_{n-2})$. We will also indicate with $\Lambda^j \Omega^1_X \boxtimes \FS_{\Delta_{n-2}}$ the sheaf 
$(\Lambda^j \Omega^1_X \boxtimes \FS_{X^{n-3}}) \tens \FS_{\Delta_{n-2}}$. 
\end{notat}

\begin{lemma}\label{lmm: E1terms}The kernel of the first differential $d_1 : E^{1,0}_1 \rTo E^{2,0}_1$ of the spectral sequence $E^{p,q}_1$ is given~by 
$$ \ker d_1  \simeq  \Big( \oplus_{| I | = 2} \FS_{\Delta_I} \Big)_0 := \left \{ (f_I)_{I}  \in \oplus_{| I | = 2} \FS_{\Delta_I} \; | \;(  f_{J} - f_{K}) \trest_{\Delta_{J} \cap \Delta_{K}} = 0, \; \;\forall  \; J, K, \; J \neq K  \right \} \;.$$The term $E^{3, -1}_2$ is given by: 
$ E^{3,-1}_2 \simeq 
 \bigoplus_{|H|=3} Q^*_{K_3(H)} \tens \bigcap_{|J | =  2 , J \not \subseteq H} \mc{I}_{\Delta_J} \simeq \bigoplus_{|H|=3} (\Omega^1_X \boxtimes \mc{I}_{\Delta_{n-2}} )_{K_3(H)}$. 
 \end{lemma}
\begin{proof}The first statement is a consequence of the fact that the map $\partial^1: \varGamma^1_0 \rTo \varGamma^2_0$ in remark \ref{rmk: graphs} is given 
 by $$( \partial^1(f_I)_{I} )_{\Gamma} = ( \epsilon_{J, \Gamma}f_J + \epsilon_{K, \Gamma}f_K) \trest_{\Delta_J \cap \Delta_K} = 
 \epsilon_{J, \Gamma} ( f_J - f_K) \trest_{\Delta_J \cap \Delta_K} \;,$$where $\Gamma$ is the graph with two edges $J$ and $K$. 
 
 The second statement follows in a similar way, considering that $E^{3,1}_1 \simeq \varGamma^3_1 \simeq \bigoplus_{|H|=3} Q^*_{K_3(H)}$ and 
 that the differential $\partial^3: \varGamma^3_1 \rTo \varGamma^4_1$ is induced by  restrictions
 $$ \partial^3 ( (x_H)_H)_{K_3(L) \cup J} =  \epsilon_{K_3(L), K_3(L) \cup J} x_L \trest_{\Delta_L \cap \Delta_J} \;,$$where $H$ and $L$ are cardinality $3$ multi-indexes and $J$ is a cardinality $2$ multi-index. 
 Hence $(x_H)_H \in  \bigoplus_{|H|=3} Q^*_{K_3(H)}
 $ belongs to $E^{3,1}_2$ if and only each restriction $x_H \trest_{\Delta_H \cap \Delta_J}$ is zero. But this means exactly it belongs to 
 $\bigoplus_{|H|=3} Q^*_{K_3(H)} \tens \cap_{|J | =  2 , J \not \subseteq H} \mc{I}_{\Delta_J}$. Note that each sheaf $Q^*_{K_3(H)} \tens \cap_{|J | =  2 , J \not \subseteq H} \mc{I}_{\Delta_J}$ is isomorphic to $(\Omega^1_X \boxtimes \mc{I}_{\Delta_{n-2}})_{K_3(H)}$. 
\end{proof}
\begin{remark}\label{rmk: Danila}
By Danila's lemma, we have the isomorphism of sheaves of invariants over $S^nX$
$$ (E^{p,q}_1)^{\perm_n} \simeq \bigoplus_{[\Gamma] \in \Graphs_{l,n}/\perm_n } \pi_* (\Tor_{-q}(\Delta, \Gamma)) ^{\perm_\Gamma } \simeq \bigoplus_{[\Gamma] \in \Graphs_{l,n}/\perm_n } \pi_* ( \Lambda^q(Q^*_\Gamma) ) ^{\perm_\Gamma } $$
 where on the right hand sides we consider the comprehensive $\perm_n$-action. More in general, we consider a subgroup $G$ of $\perm_n$. Hence the sheaves of $G$-invariants over the symmetric variety $S^nX$
 $$ \pi_* (E^{p,q}_1)^{G} \simeq \bigoplus_{[\Gamma] \in \Graphs_{l,n}/G } \pi_* ( \Tor_{-q}(\Delta, \Gamma) ) ^{\Stab_{G}(\Gamma)} \simeq \bigoplus_{[\Gamma] \in \Graphs_{l,n}/ G} \pi_* ( \Lambda^q(Q^*_\Gamma) ) ^{\Stab_{G}(\Gamma)}  $$

\end{remark}

The following table lists the groups $\perm_\Gamma$ and the representation $\Res_{\perm_\Gamma} \epsilon_{E_\Gamma}$ for all isomorphisms classes of non empty graphs $\Gamma \subseteq K_4$ without isolated vertices. Here we indicate with $A_1$ the graph with a single edge, with $A_2$ a graph with two intersecting edges, with $B_2$ a graph with two non-intersecting edges, with $A_3$ and $B_3$ the acyclic subgraphs 
of $K_4$ with three edges and with, respectively, no vertex of degree $3$ and a single vertex of degree $3$.

\vspace{.2cm}
\begin{table}[H]
\centering
\resizebox{1.015\textwidth}{!}{
$
\hspace{-0.02\textwidth} 
\begin{array}{c| c| c| c| c| c| c| c| c|c|c}
& \Gamma = A_1 & \Gamma = A_2 &\Gamma = B_2 & \Gamma = A_3 & \Gamma = B_3 & \Gamma = K_3        & \Gamma = K_3 \cup J & \Gamma = C_4 & \Gamma = C_4 \cup L & \Gamma = K_4  \\
\perm_\Gamma & \perm_2 \times \perm_2 & \perm_2    & \perm_2 \times \perm_2 \times \perm_2  & \perm_2 & 
\perm_2 & \perm_3 & \perm_2 & D_4 & \perm_2 \times \perm_2 & 
\perm_4 \\
\Res_{\perm_\Gamma} \epsilon_{E_\Gamma} & 1  & \epsilon & 1 \tens 1 \tens \epsilon&  \epsilon &  \epsilon & \epsilon & \epsilon & \ell(-1_\rho, 1_\sigma) & 1 &  1
\end{array}
$
}
\caption{} \label{table: res}
\end{table}

\vspace{0.2cm}
The case of the graph $\Gamma = B_2$ needs a line of explanation. We can suppose that $\Gamma$ is the graph consisting of the edges $\{1, 2 \}, \{3, 4 \}$; hence $\perm_\Gamma = \langle (12) \rangle \times \langle (34) \rangle \times \langle (13)(24) \rangle$. The subgroup $\langle (13)(24) \rangle$ is isomorphic to $\perm^\Gamma_2$ and acts nontrivially in the representation 
$\Res_{\perm_2 \times \perm_2 \times \perm_2}  \epsilon_{E_\Gamma}$. 

\begin{lemma}\label{lmm: A4}The invariants $(E^{2,0}_1)^{\perm_n}$ of the term $E^{2,0}_1$ of the spectral sequence $E^{p,q}_1$ are isomorphic to 
the sheaf $\mc{A}_4(\FS_X) :=\pi_* \Big( (\FS_{\Delta_{12}} \tens \FS_{\Delta_{34}}) \tens \epsilon \tens 1 \Big)^{ \langle (13)(34) \rangle \tens \perm(\{5, \dots, n \})}$.
Over an affine open set $S^n U$ its module of sections is isomorphic to $\Lambda^2H^0(\FS_U) \tens S^{n-4}H^0(\FS_U)$. 
\end{lemma}\begin{proof}We first remark that, for any $n \geq 4$, the types $A_2$ and $B_2$ are the only isomorphism classes of subgraphs of $K_n$ with $2$ edges and without isolated vertices.  
By remark \ref{rmk: Danila}, we have that $\pi_*(E^{2,0}_1)^{\perm_n} \simeq \pi_*(\FS_{\Delta_{\Gamma_1}} \tens \Res_{\perm_{\Gamma_1}}\epsilon_{\Gamma_1})^{\perm_{\Gamma_1}} \oplus \pi_*(\FS_{\Gamma_2} \tens \Res_{\perm_{\Gamma_2}}\epsilon_{\Gamma_2})^{\perm_{\Gamma_2}}$, where $\Gamma_1$ is a graph of type $A_2$, and $\Gamma_2$ is a graph of type $B_2$. Now the first summand is zero and the second identifies to the one in the statement, when taken $\Gamma_2$ to be the graph with edges $\{1, 2 \}$, $\{ 3, 4 \}$.  
\end{proof}
\begin{notat}Let $n, l \in \mbb{N}^*$, $l <n$.  We denote with $w_l$ the morphism $X \times S^{n-l}X \rTo S^n X$ sending 
$(x, y)$ to the $0$-cycle $lx +y$. It is a finite morphism if $l=1$ and a closed immersion if $l \geq 2$. 
\end{notat}

\begin{pps}\label{pps: A4}
In the identifications $(E^{1,0}_1)^{\perm_n} \simeq {w_2}_*(\FS_X \boxtimes \FS_{S^{n-2}X})$ and $(E^{2,0}_1)^{\perm_n} \simeq \mc{A}_4(\FS_X)$, 
the  invariant differential $ d^{\perm_n}_1 : (E^{1,0}_1)^{\perm_n} \rTo  (E^{2,0}_1)^{\perm_n}$ of the spectral sequence $(E^{p,q}_1)^{\perm_n}$
is
determined locally, over an affine open set of the form $S^n U = \Spec S^n A$, by the formula 
$$ d^{\perm_n}_1 (a \tens b_1 . \dots . b_{n-2}) = \sum_{1 \leq i<j \leq n-2} ( a \tens b_i b_j - b_i b_j \tens a ) \tens \widehat{b_{ij}} \;,$$where $a, b_i \in A$. 
\end{pps}\begin{proof}The expression is obtained --- over an affine open set of the form $S^n U$ as in the statement --- by identifying 
$\pi_*(E^{2,0}_1)^{\perm_n}$ with $\pi_*(\FS_{\Delta_{\Gamma_2}})^{\perm_{\Gamma_2}}$ where $\Gamma_2$ is the 
graph with  edges $\{1, 2 \}$, $\{3, 4 \}$ we considered in the proof of lemma \ref{lmm: A4}. Hence the map of invariants $d_1^{\perm_n}$ can be identified with 
the morphism 
$$ {w_2}_*(\FS_X \boxtimes \FS_{S^{n-2}X})  \simeq \pi_*(\FS_{\Delta_{12}})^{\perm(\{3, \dots, n\})} \simeq [ \pi_*(\FS_{\Delta_{12}}) \oplus \pi_*(\FS_{\Delta_{34}}) ]^{\perm_{\Gamma_2}} \rTo \pi_*(\FS_{\Delta_{\Gamma_2}})^{\perm_{\Gamma_2}} \simeq  \mc{A}_4(\FS_X)$$
given by \begin{align*} d_1^{\perm_n} (a \tens b_1\dots b_{n-2})  = \: & d_1 (a \tens b_1 \dots b_{n-2} + (13)(24)_*(a \tens b_1 \dots b_{n-2})) \\
= \: & d_1 (a \tens b_1 \dots b_{n-2})  + (13)(24)_* d_1 (a \tens b_1 \dots b_{n-2}) \\
= \: &  \sum_{1 \leq i<j \leq n-2} ( a \tens b_i b_j - b_i b_j \tens a ) \tens \widehat{b_{ij}} \end{align*}
where we saw the element $b_1 . \dots . b_{n-2}$ as a $\perm(\{3, \dots , n \})$-invariant element  in $H^0(U)^{\tens n-2}$. 
\end{proof}
\begin{notat}\label{notat: kerd1sigma}We denote the kernel 
$\ker d_1^{\perm_n}$ with $ {w_2}_*(\FS_X \boxtimes \FS_{S^{n-2}X})_0$. 
\end{notat}
The next lemma is immediate from lemma \ref{lmm: E1terms}. 
\begin{lemma}The invariants $(E^{3,-1}_2)^{\perm_n}$ of the term $E^{3,-1}_2$ of the spectral sequence $E^{p,q}_1$ 
are isomorphic to the sheaf 
${w_3}_*( (\Omega^1_X \boxtimes \mc{I}_{\Delta_{n-2}})^{\perm_{n-3}} )$, where $\perm_{n-3}$ acts on the factor $X^{n-3}$ of the product 
$X \times X^{n-3}$. 
\end{lemma}

\subsection{The case \protect $n =3$.}\label{subsection: n=3}
\begin{theorem}\label{thm: inv3}
Let $X$ be a smooth algebraic variety. The complex of coherent sheaves over $S^3 X$ 
$$ \mbb{I}^\bullet_3 \colon \qquad 0 
\rTo \FS_{S^3 X} {\rTo^r} {w_2}_* (\FS_{X \times X}) {\rTo^D} {w_3}_*( \Omega^1_X ) \rTo 0 $$--- 
where $r$ is the restriction and $d$ is given locally by $D(a \tens b ) = 2a db - b da$ --- is a resolution of the sheaf of invariants $(\mc{I}_{\Delta_3})^{\perm_3}$.  
\end{theorem}\begin{proof}All subgraphs $\Gamma \subseteq K_3$ are connected. Hence the multitors $\Tor_{q}(\Delta, \Gamma)$ are isomorphic to $$ \Lambda^q (\Omega^1_X \tens q_\Gamma)_{\Gamma} \tens \Res_{\perm_\Gamma} \epsilon_{E_\Gamma}$$for the comprehensive $\perm_\Gamma$-action. Remembering that $q_{K_3} \simeq \epsilon_3$ and 
using  table \ref{table: res} we see immediately that the term $(E^{3,-2}_1)^{\perm_3} \simeq \Tor_2(\Delta, K_3)^{\perm_3} \simeq 
 {w_3}_*(\Lambda^2 \Omega^1_X \tens \epsilon)^{\perm_3} $ vanishes. 
 Moreover, $(E^{2, 0}_1)^{\perm_3} \simeq \pi_*( \FS_{\Delta_{123}} \tens \epsilon) ^{\perm_2}$, $(E^{3,0}_1)^{\perm_3} \simeq \pi_*(\FS_{\Delta_{123}} \tens \epsilon)^{\perm_3}$: hence there are no $\perm_3$-invariants for $q=0$, $p=2,3$. 
 Therefore, for $p+q \geq 0$, the only nonzero terms in the spectral sequence of invariants are of the form $(E^{0,0}_1)^{\perm_3}$, $(E^{1,0}_1)^{\perm_3}$, $(E^{3,-1}_1)^{\perm_3}$ and $(E^{3,-3}_1)^{\perm_3}$. The first two are easily proven to be isomorphic to the sheaves $\FS_{S^3 X}$ and ${w_2}_*(\FS_{X \times X})$, respectively. 
Moreover  $(E^{3,-1}_1)^{\perm_3} \simeq  \Tor_1(\Delta, K_3)^{\perm_3} \simeq 
{w_3}_* ( \Omega^1_X)$ and, analogously, $(E^{3,-3}_1)^{\perm_3} \simeq {w_3}_*(\Lambda^3 \Omega^1_X)$. 
Hence we have the resolution of
the statement where the map $D: {w_2}_* (\FS_{X \times X}) {\rTo} {w_3}_* ( \Omega^1_X) $ is induced by the second differential $d_2^{\perm_3}$ of the spectral sequence of invariants; the precise local expression of the map $D$ follows from proposition \ref{pps: D} in the appendix. 
\end{proof}
Let $G = \perm(\{23\})$. In section \ref{subsec: sheaves} we will need the following result about the invariants $\pi_*(\mc{I}_{\Delta_3})^G$. 

\begin{pps}\label{pps: 211}Let $X$ be a smooth algebraic variety. Over $S^3X$, the sheaf of invariants $\pi_*(\mc{I}_{\Delta_3})^G$ is resolved by the complex 
$$ 0 {\rTo} {w_1}_* (\FS_X \boxtimes \FS_{S^2 X}) {\rTo} \left [ {w_2}_*(\FS_X \boxtimes \FS_X )^{\oplus 2} \right ] _0  {\rTo} {w_3}_*(\Omega^1_X) \rTo 0 \;.$$Here the sheaf $\left [{w_2}_*(\FS_X \boxtimes \FS_X )^{\oplus 2} \right]_0$ is the kernel 
of the map ${w_2}_*(\FS_X \boxtimes \FS_X )^{\oplus 2}  {\rTo} {w_2}_*( \FS_{\Delta_2})$ given locally by 
$(a \tens u, b \tens v) \rMapsto au - bv$.  
The first map of the complex is locally defined as $a \tens u.v \rMapsto (a u \tens v + av \tens u, 
2 uv \tens a)$, while the second is determined by $(a \tens u, b \tens v) \rMapsto 2adu - vdb$. \end{pps}
\begin{proof}We consider the invariants $\pi_*(E^{p,q}_1)^G$ of the spectral sequence $E^{p,q}_1$ by the group $G = \perm(\{2 3 \})$. 
By proposition \ref{pps: torgamma}, remark \ref{rmk: comprehensive} and remark \ref{rmk: Danila}, the terms 
$\pi_*(E^{p,q}_1)$, as $G$-representations, are 
\begin{align*} \pi_*(E^{p,q}_1) \simeq \: 
&
 \bigoplus_{[\Gamma] \in \Graphs_{p,3}/G } \Lambda^{-q}(\Omega^1_X \tens \Res_{\Stab_G(\Gamma)} q_\Gamma)_\Gamma \tens \Res_{\Stab_G(\Gamma)} \epsilon_{E_{\Gamma}} \;.
 \end{align*}
 It is then immediate to prove that $\pi_*(E^{0,0}_1)^G \simeq \pi_*(\FS_{X^3})^{G} \simeq {w_1}_*(\FS_X \boxtimes \FS_{S^2 X})$, 
 $\pi_*(E^{1,0}_1)^G \simeq \pi_*(\FS_{\Delta_{12}}) \oplus \pi_*(\FS_{\Delta_{23}}) \simeq  {w_2}_*(\FS_X \boxtimes \FS_X)^{\oplus 2}$, $\pi_*(E^{2,0}_1)^G \simeq \pi_*( \FS_{\Delta_{123}}) \simeq {w_2}_*(\FS_{\Delta_2})$. 
 For $q<0$ and $p + q \geq 0$, the only nontrivial terms are 
$$ \pi_*(E^{3, -1}_1)^G \simeq 
{w_3}_*(\Omega^1_X)$$and $ \pi_*(E^{3, -3}_1)^G \simeq {w_3}_*(\Lambda^3 \Omega^1_X)$, since $ \pi_*(E^{3,-2}_1)^G 
\simeq \left[ {w_3}_*(\Lambda^2(\Omega^1_X \tens \epsilon) ) \tens \epsilon \right]^G = 0 $. It is now easy to see that   $\pi_*(E^{1, 0}_2)^G \simeq \left [ {w_2}_*(\FS_X \boxtimes \FS_X )^{\oplus 2} \right ] _0$; 
hence, drawing the page $E_2$ of the spectral sequence,  we get the complex in the statement. To prove that the maps are the ones mentioned above, one sees immediately that the first is induced by restrictions, while for the second one has just to track down the higher differential $d_2^G$, but this is done easily taking $G$-invariants in the statement of corollary \ref{crl: B}. 
\end{proof}

\subsection{The case \protect $n = 4$.}\label{subsection: n=4}
In order to understand the $\perm_4$-invariants of the sheaf $\mc{I}_{\Delta_4}$, 
we have to work out the spectral sequence $\pi_*(E^{p,q}_1)^{\perm_4}$; the first step, by virtue of remark \ref{rmk: Danila}, is to compute, for each class $[\Gamma] \in  \Gamma \in \Graphs_{p, 4} / \perm_4$,  the invariants $\pi_* (\Tor_{-q}(\Delta, \Gamma) \tens \Res_{\perm_\Gamma} \epsilon_{E_\Gamma})^{\perm_{\Gamma}}$. For $q<0$, we are just interested in  graphs with at least one cycle,
which  are all connected:  the above sheaves then have the form 
$$ \pi_* (\Tor_{-q}(\Delta, \Gamma) \tens \Res_{\perm_\Gamma} \epsilon_{E_\Gamma})^{\perm_{\Gamma}} \simeq \pi_*(\Lambda^{-q}(\Omega^1_X \tens q_\Gamma)_\Gamma  \tens \Res_{\perm_\Gamma} \epsilon_{E_\Gamma})^{\perm_{\Gamma}} $$and hence can be computed easily by combining table \ref{table: qgamma} with table \ref{table: res}.  For convenience of the reader we present the computation of the invariants $\pi_* (\Tor_{-q}(\Delta, \Gamma) \tens \Res_{\perm_\Gamma} \epsilon_{E_\Gamma})^{\perm_{\Gamma}} $ in the following table.  

\vspace{0.2cm}
\begin{table}[H]
\centering
\resizebox{\textwidth}{!}{%
 {\renewcommand{\arraystretch}{1.5} 
 \begin{tabular}{>{$}c<{$}|>{$}c<{$}|>{$}c<{$}|>{$}c<{$}|>{$}c<{$}|>{$}c<{$}}
 & \Gamma = K_3                     & \Gamma = K_3 \cup J & \Gamma = C_4 & \Gamma = C_4 \cup L & \Gamma = K_4 \\ \hline
q = -1                                  & {w_3}_*(\Omega^1_X \boxtimes \FS_X) & {w_4}_*(\Omega^1_X) & 0 &  0 & 0 \\ \hline
q = -2 &  & 0 & 0 & {w_4}_* ( \Lambda^2 \Omega^1_X \oplus \Lambda^2 \Omega^1_X) &  {w_4}_*( \Lambda^2 \Omega^1_X)  \\ \hline
 q = -3 & {w_3}_*( \Lambda^3 \Omega^1_X \boxtimes \FS_X) & {w_4}_*( \Lambda^3 \Omega^1_X) & 0 &0 &  {w_4}_* ( S^3 \Omega^1_X )  \\
 \hline 
 q = -4 & 0 & 0 & 0 &  {w_4}_* ( (\Lambda^4 \Omega^1_X)^{\oplus 3} \oplus S^{2,1,1} \Omega^1_X \oplus S^{2,2} \Omega^1_X)   &  {w_4}_*( (\Lambda^4 \Omega^1_X)^{\oplus 2} \oplus S^{2,2} \Omega^1_X)    \\ \hline
q = -5& {w_3}_*( \Lambda^5 \Omega^1_X \boxtimes \FS_X) & {w_3}_*( \Lambda^5 \Omega^1_X) & 0 & 0& {w_4}_* (S^{3,1,1} \Omega^1_X) 
\\ \hline
q = -6 & 0 & 0 & 0 & {w_4}_*( (\Lambda^6 \Omega^1_X )^{\oplus 4}  \oplus (S^{2,1,1,1,1} \Omega^1_X)^{\oplus 2}) \oplus    &  {w_4}_*( (\Lambda^6 \Omega^1_X )^{\oplus 3}   \oplus S^{2,1,1,1,1} \Omega^1_X) ) \oplus \\ &&&& {w_4}_* ((S^{2,2,1,1} \Omega^1_X)^{\oplus 2} )&  {w_4}_*(  (S^{2,2,1,1} \Omega^1_X )^{\oplus 2} \oplus    S^{3,3} \Omega^1_X  )
\end{tabular}%
}
}
\caption{} \label{table: inv}
\end{table}

\noindent
As for $q = 0$, it is clear that $$ \pi_* (E^{p, 0}_1)^{\perm_4} \simeq 
 \pi_* (\varGamma^{p}_0)^{\perm_n} \simeq \oplus_{[\Gamma] \in \Graphs_{p,4} / \perm_4} \pi_*( \FS_{\Delta_\Gamma} \tens  \Res_{\perm_\Gamma} \epsilon_{E_\Gamma})^{\perm_{\Gamma}} \;.$$
 We have the following lemmas
 \begin{lemma}\label{lmm: Ep0}The complex $\pi_* (E^{\bullet, 0}_1)^{\perm_4} \simeq  \pi_* (\varGamma^{\bullet}_0)^{\perm_n}$ is 
 quasi isomorphic to the complex 
 $$ 0 \rTo \FS_{S^4 X} {\rTo^r} {w_2}_*(\FS_X \boxtimes \FS_{S^2 X}) \rTo^{d_1^{\perm_n}} \mc{A}_4(\FS_X) \rTo 0 \;,$$
 which is exact in degree greater or equal than $2$. The first map is the restriction, while the second is given locally by $a \tens b. c \rMapsto a \wedge bc$. 
  \end{lemma}
 \begin{proof} 
 With the help of table \ref{table: res}, we immediately have that
the complex $( \pi_* E^{\bullet, 0}_1 )^{\perm_4}$ is quasi-isomorphic to the complex 
$$ 0 \rTo \FS_{S^4 X} {\rTo^r} {w_2}_*(\FS_X \boxtimes \FS_{S^2 X}) \rTo^{d_1^{\perm_n}} \mc{A}_4(\FS_X) \rTo 0 \rTo 0 \rTo \pi_* (\FS_{\Delta_{1234}})
 \rTo  \pi_* (\FS_{\Delta_{1234}}) \rTo 0 \;.$$The map $\pi_* (\FS_{\Delta_{1234}})
 \rTo  \pi_* (\FS_{\Delta_{1234}}) $ is immediately an isomorphism, being induced by the identity on the sheaf $\FS_{\Delta_{1234}}$; 
 the map $d_1^{\perm_n}: {w_2}_*(\FS_X \boxtimes \FS_{S^2 X}) \rTo \mc{A}_4(\FS_X)$ is surjective, by proposition \ref{pps: A4}, since it is given locally, on an affine open subset of the form $S^n U$, $U = \Spec(A)$, 
 by the map $A \tens S^2 A \rTo  \Lambda^2 A$, sending $a \tens b.c$ to $a \wedge bc = a \tens bc - bc \tens a$.  \end{proof}
 
 \begin{lemma}\label{lmm: inv54}Consider a graph $\Gamma$ of the kind $C_4 \cup L$. Then the vector space $\Lambda^4(\mbb{C}^2 \tens q_\Gamma)$ is completely $\perm_\Gamma$-invariant. Moreover, the composition 
 $$ c: \Lambda^4(\mbb{C}^2 \tens q_\Gamma) \rTo \Lambda^4(\mbb{C}^2 \tens q_{K_4}) \rTo \Lambda^4(\mbb{C}^2 \tens q_{K_4})^{\perm_4} $$where the first map is the injection $i_{\Gamma, K_4}$ and the second is the projection onto the invariants, is an isomorphism. 
 \end{lemma}\begin{proof}The first statement is a consequence of table \ref{table: qgamma}, case in which $\Gamma$ is of the kind $C_4 \cup L$, $q = 4$, 
 and $X = \mbb{C}^2$. By the same table we can also deduce that the vector space $ \Lambda^4(\mbb{C}^2 \tens q_{K_4})^{\perm_4}$ is one dimensonal. Therefore, to prove the second statement, we just have to prove that the map $c$ is nonzero. 
 It is not restrictive to suppose that $\Gamma$ is defined by edges $ \{1,2 \}, \{1,3 \}, \{1,4 \}, \{2,3 \}, \{3,4 \} $. 
 Indicating an oriented $3$-cycle with a sequence of its vertices, the basis of $\mbb{C}^2 \tens q_\Gamma$ is then given by vectors $e_i \tens e_{123}$ and 
 and $e_i \tens e_{134}$, $i=1,2$. Denote more briefly,  for an oriented $3$-cycle $H$, with $\gamma_{H} = e_1 \tens e_H$ and $\delta_H = e_2 \tens e_H$, where we write the elements in the set $H$ in an order according to the given orientation. 
 Hence 
  we can take as a basis of $\mbb{C}^2 \tens q_\Gamma$ the vectors $\gamma_{123}$, $\delta_{123}$, $\gamma_{134}$, $\delta_{134}$; similarly, as a basis of $\mbb{C}^2 \tens q_{K_4}$ we take the previous vectors, to which we add $\gamma_{124}, \delta_{124}$. For brevity's sake, denote with $\alpha_H := \gamma_H \wedge \delta_H$. 
  A basis of $\Lambda^4 (\mbb{C}^2 \tens q_\Gamma)$, which is fully invariant, is then given by 
  $a  = \alpha_{123} \wedge \alpha_{134}$. Since this element is invariant by $\perm_\Gamma \simeq 
  \langle (13) \rangle \times \langle (24) \rangle$,  
we have that   the map $c$ is, up to a constant, given by
$$ c(a) = \sum_{[\sigma] \in \perm_4/\perm_\Gamma} \sigma_* i_{\Gamma, K_4} a  \;.$$Now the cosets $\perm_4/\perm_\Gamma$ can be represented by the set $\{ \id, (12), (14), (23), (34), (12)(34) \}$, hence $c(a)$ is given, up to a constant by 
\begin{align*} c(a) = \: & \alpha_{123} \wedge \alpha_{134}
+ \alpha_{213}  \wedge \alpha_{234} 
+ \alpha_{423}  \wedge \alpha_{431}  
 + \alpha_{132}  \wedge \alpha_{124} 
+\alpha_{124}  \wedge \alpha_{143} 
+ \alpha_{214}  \wedge  \alpha_{243}  \\
= \: & \alpha_{123}  \wedge \alpha_{134} 
+  \alpha_{123}  \wedge \alpha_{234} 
+  \alpha_{234}  \wedge \alpha_{134} 
 + \alpha_{123}  \wedge \alpha_{124}  
+ \alpha_{124}  \wedge \alpha_{134}
+ \alpha_{124}  \wedge \alpha_{234} 
\end{align*}Now, using that $\alpha_{234} = \alpha_{123} + \alpha_{134} - \alpha_{124}$, we get easily that, up to a constant, 
  $$
  c(a) = 3  \big( \alpha_{123}  \wedge \alpha_{124}   +  \alpha_{123}  \wedge \alpha_{134}  + \alpha_{124}  \wedge \alpha_{134}  \big) 
 $$which is a nonzero element of $\Lambda^4(\mbb{C}^2 \tens q_{K_4})^{\perm_4} \subseteq \Lambda^4(\mbb{C}^2 \tens q_{K_4})$. 
 \end{proof}
 
\begin{notat}\label{notat: kerA}Consider now the second invariant differential  
$$ A:=d_2^{\perm_4}: {w_3}_*(\Omega^1_X \boxtimes \mc{I}_{\Delta_2})  \simeq \pi_*(E^{3,-1}_2)^{\perm_4} \simeq \rTo \pi_*(E^{5, -2}_2)^{\perm_4} 
 \;.
$$Its precise expression is determined in the appendix, corollary \ref{crl: A}. We denote with 
${w_3}_*(\Omega^1_X \boxtimes \mc{I}_{\Delta_2})_0$ its kernel. 
\end{notat}
Recall notation \ref{notat: kerd1sigma}. We have the following result. 
\begin{theorem}\label{thm: inv4}Let $X$ be a smooth algebraic variety. 
The sheaf of invariants $(\mc{I}_{\Delta_4})^{\perm_4}$ over the symmetric variety $S^4 X$
admits a right resolution given by a natural
complex 
$$ \mbb{I}^\bullet_4 := 0 
\rTo \FS_{S^4 X} {\rTo^r} {w_2}_*(\FS_X \boxtimes \FS_{S^2 X})_0 {\rTo^{D}} {w_3}_* (\Omega^1_X \boxtimes \mc{I}_{\Delta_2})_0 
 {\rTo^{C}} {w_4}_* (S^3 \Omega^1_X)  \rTo 0 \;,$$where $r$ is the restriction, $D$ is defined locally as 
 $D(a \tens u.v) = (2adu - uda) \tens v + (2adv - vda) \tens u$ and $C$ is determined, up to a constant, by 
 $C(\omega \tens f) = \sym(\omega \tens d^2_\Delta f)$, where $\omega \in \Omega^1_X$ and $f \in \mc{I}_{\Delta_2}$.  \end{theorem}
\begin{proof}For brevity's sake, we will do the proof in the case $\dim X \leq 2$: the proof in the general case is analogous, but slighlty more technical: we indicate in the next remark how to adapt the present proof to the general case. 

The abutment of the spectral sequence of invariants $\pi_* (E^{p,q}_1)^{\perm_4}$ is zero in positive degree, 
after remark~\ref{rmk: abut}. The complex $\pi_*(E^{\bullet, 0}_1)^{\perm_4}$ has been studied in  lemma \ref{lmm: Ep0}. 
At level $q = -1$, recalling table \ref{table: inv} and notation \ref{notat: omegadelta}, we just have the complex 
$ {w_3}_*(\Omega^1_X \boxtimes \FS_X) {\rTo} {w_3}_* (\Omega^1_X \boxtimes \FS_{\Delta_2}) \simeq  {w_4}_* (\Omega^1_X)$, in degree $3$ and $4$, 
 and at level $q = -2$ we have the complex 
 $  {w_4}_* ( K_X \oplus K_X) {\rTo}  {{w_4}_*}(K_X)$, in degree $5$ and $6$. Moreover, at level $q = -4$, we have the complex ${w_4}_*(K_X^2) {\rTo} {w_4}_*(K_X^2)$, in degree $5$ and $6$. 
 The only other nonzero terms for $q<0$ are $\pi_*( E^{6,-3}_1)^{\perm_4} \simeq  {w_4}_* (S^3 \Omega^1_X)$ and 
 $\pi_* (E^{6,-6}_1)^{\perm_4} \simeq  {w_4}_*(K^3_X)$. 
 
 We now prove that the map ${w_3}_*(\Omega^1_X \boxtimes \FS_X) { \rTo} {w_4}_*( \Omega^1_X )$ is surjective.  Indeed it can be seen as the map of $\perm_4$-invariants 
 $$ (E^{3,-1}_1)^{\perm_4} \simeq  \pi_* ( (\Omega^1_X)_{K_3(123)})^{\perm_3} \rTo 
 \pi_*( (\Omega^1_X)_{K_3(123) \cup \{3 4 \} } )^{\perm_2} \;.$$This map 
 is naturally the composition 
$$  \pi_* ( (\Omega^1_X)_{K_3(123)})^{\perm_3} \rInto  \pi_* ( (\Omega^1_X)_{K_3(123)})^{\perm_2} \rTo 
 \pi_*( (\Omega^1_X)_{K_3(123) \cup \{3 4 \} } )^{\perm_2} $$where the first map is the natural inclusion and the second 
 is the map of $\perm_2$-invariants of the map 
 $\pi_* ( (\Omega^1_X)_{K_3(123)}) \rTo \pi_*( (\Omega^1_X)_{K_3(123) \cup \{3 4 \} } )$.  
 But this last map is surjective, since the map $(\Omega^1_X)_{K_3(123)} \rTo  (\Omega^1_X)_{K_3(123) \cup \{3 4 \} }
 $ is a restriction and hence surjective, and because $\pi$ is finite.  Moreover, the natural inclusion $\pi_* ( (\Omega^1)_{K_3(123)})^{\perm_3} \rInto  \pi_* ( (\Omega^1_X)_{K_3(123)})^{\perm_2}$ is an isomorphism, since both terms coincide with $\pi_* ( (\Omega^1_X)_{K_3(123)})$. This means in particular that $\pi_* (E^{4,-1}_2)^{\perm_4} = 0$. 
 
 Let's now look at the map $ {w_4}_* (K_X \oplus K_X) {\rTo}  {w_4}_*(K_X)$. Its cokernel is 
 isomorphic to $\pi_* (E^{6,-2}_2)^{\perm_4}$. Now, because of the form of the complexes $\pi_*(E^{\bullet, q}_1)^{\perm_4}$ 
 for $q = 0, -1$, and because we proved that $\pi_* (E^{4,-1}_2)^{\perm_4} = 0$, 
  there are no nonzero higher invariant differentials $d^{\perm_4} _l$ targeting $\pi_* (E^{6,-2}_2)^{\perm_4}$: hence this term has to live till 
 the $\infty$-page, as a graded sheaf of the abutment, but the abutment in degree $4$ is zero. Hence $\pi_* (E^{6,-2}_2)^{\perm_4}$ has to vanish and the map $ {w_4}_*(K_X \oplus K_X) {\rTo}  {w_4}_*(K_X)$ has to be surjective. 
 Hence $\pi_*(E^{5, -2}_2)^{\perm_4} \simeq  {w_4}_*(K_X)$. 
 
 Finally, we look at the map $\pi_*(E^{5,-4}_1)^{\perm_4} \simeq {w_4}_*(K_X^2) {\rTo} {w_4}_* (K_X^2) \simeq \pi_*(E^{6,-4}_1)^{\perm_4}$ at level $q=-4$. 
 It coincides with the map of invariants 
\begin{equation}\label{eq: inv54}  \pi_* ( \Lambda^4 (Q^*_\Gamma) )^{\perm_{\Gamma}} {\rTo} \pi_* ( \Lambda^4(Q^*_{K_4}) )^{\perm_4} \;,\end{equation}
induced by the  inclusion $i_{\Gamma, K_4} : Q^*_{\Gamma} \rInto Q^*_{K_4}$, where $\Gamma$ is a graph of the kind $C_4 \cup L$. 
Now, since both $\perm_\Gamma$ and $\perm_4$ act trivially over $\Delta_{1234} = \Delta_\Gamma$, 
and $\pi |_{\Delta_{1234}}$ is a closed immersion, 
the map (\ref{eq: inv54}) is an isomorphism 
 if and only if the map of line bundles ${i_{\Gamma}}_* K^2_X \simeq 
 \Lambda^4 (Q^*_\Gamma) ^{\perm_{\Gamma}} \rTo \Lambda^4(Q^*_{K_4}) ^{\perm_4} \simeq {i_{\Gamma}}_* K^2_X$ 
  over $\Delta_{1234}$ is. Since $\Lambda^4(Q^*_\Gamma)$ is fully $\perm_\Gamma$-invariant,  the previous map coincide, up to constants, with the composition 
  $$  \Lambda^4(Q^*_\Gamma) \rTo   \Lambda^4(Q^*_{K_4})  \rTo   \Lambda^4(Q^*_{K_4}) ^{\perm_4} \;,$$where the first is the inclusion and 
  the second is the projection onto the invariants. Now, on the fibers, this map is precisely the map of lemma \ref{lmm: inv54}. Hence it is an isomorphism.

 Consequently, at level $E_2$, the only nonzero terms are 
 \begin{align*} & \pi_*(E^{0,0}_2)^{\perm_4} \simeq  \; ( \mc{I}_{\Delta_4}) ^{\perm_4} & \hspace{1cm} 
 & \pi_*(E^{1,0}_2)^{\perm_4} \simeq  \; \coker \big( \FS_{S^4 X} {\rTo} {w_2}_*(\FS_X \boxtimes \FS_{S^2 X})_0 \big) \\ 
 & \pi_*(E^{3,-1}_2)^{\perm_4} \simeq  \; {w_3}_*(\Omega^1_X \boxtimes \mc{I}_{\Delta_2}) \;, & \hspace{1.5cm}
  & \pi_*(E^{5,-2}_2)^{\perm_4} \simeq \;   {w_4}_*(K_X) \\ 
&  \pi_*(E^{6,-3}_2)^{\perm_4} \simeq  \;  {w_4}_*(S^3 \Omega^1_X)\;,  & \hspace{1.5cm}
 & \pi_*(E^{6,-6}_2)^{\perm_4} \simeq  \;   {w_4}_*(K^3_X) \;.
 \end{align*}
 
 \noindent
Therefore, drawing the page $E_2$ of the spectral sequence $\pi_*(E^{p,q}_1)^{\perm_4}$,  
we deduce a complex $0 \rTo \pi_*(E^{1,0}_2)^{\perm_4} {\rTo} {w_3}_*(\Omega^1_X \boxtimes \mc{I}_{\Delta_2} ) {\rTo^A} {w_4}_*( K_X)
\rTo 0 
$ which is non exact only in the middle, with cohomology isomorphic to ${w_4}_*(S^3 \Omega^1_X)$; henceforth we have an exact sequence 
$$ 0 \rTo \pi_*(E^{1,0}_2)^{\perm_4} {\rTo} {w_3}_*(\Omega^1_X \boxtimes \mc{I}_{\Delta_2} )_0 {\rTo^C} {w_4}_*(S^3 \Omega^1_X) \rTo 0  \;. $$The statement of the theorem follows, since one has as well an exact sequence
\begin{gather*} 0 \rTo (\mc{I}_{\Delta_4})^{\perm_4} \rTo \FS_{S^4 X} {\rTo} {w_2}_*(\FS_X \times \FS_{S^2 X})_0 {\rTo} \pi_*(E^{1,0}_2)^{\perm_4} \rTo 0 \;.
\end{gather*}The precise expression of map $D$ and $C$ will be determined in the appendix, proposition \ref{pps: D} and proposition \ref{pps: mapC}. \end{proof}
\begin{remark}We sketch here how to adapt the proof in the case $X$ is of higher dimension. Using table \ref{table:  inv}, we see that, for $p \leq 4$, we have to consider, for $q$ negative 
odd, the surjective maps $(E^{3, q}_1)^{\perm_4} \simeq {w_3}_*(\Lambda^{-q} \Omega^1_X \boxtimes \FS_X) {\rOnto} 
{w_4}_*(\Lambda^{-q} \Omega^1_X) \simeq (E^{4, q}_1)^{\perm_4}$. Hence, for $q$ odd, $(E^{3, q}_2)^{\perm_4} \simeq 
{w_3}_*(\Lambda^{-q} \Omega^1_X \boxtimes \mc{I}_{\Delta_2})$ and $(E^{4, q}_2)^{\perm_4}=~0$. 
Consider now $p \geq 5$. 
Analogously as in the case of a surface, 
we have a surjective map $(E^{5,-2}_1)^{\perm_4} \simeq 
{w_4}_*(\Lambda^2 \Omega^1_X \oplus \Lambda^2 \Omega^1_X) {\rTo} {w_4}_*(\Lambda^2 \Omega^1_X) \simeq (E^{6, -2}_1)^{\perm_4}$, 
with kernel ${w_4}_*(\Lambda^2 \Omega^1_X) \simeq  (E^{5,-2}_2)^{\perm_2}$. In higher dimension one can prove, in a way similar to lemma \ref{lmm: inv54}, that the map
$(E^{5,-4}_1)^{\perm_4} \rTo (E^{6,-4}_1)^{\perm_4}$ is surjective with kernel ${w_4}_*(\Lambda^3 \Omega^1_X \tens  \Omega^1_X)
\simeq (E^{5, -4}_2)^{\perm_4}$. Finally, we have that $(E^{6,-5}_1)^{\perm_4} \simeq 
(E^{6,-5}_2)^{\perm_4} \simeq {w_4}_*(S^{3,1,1} \Omega^1_X)$. 
These are the only relevant differences in the page $E_2$. Analogously to corollary \ref{crl: A}, one can prove that the
 second differential $(d_2^{3, -3})^{\perm_4} : {w_3}_*(\Lambda^3 \Omega^1_X \boxtimes \mc{I}_{\Delta_2}) {\rOnto} {w_4}_*(\Lambda^3 \Omega^1_X \tens \Omega^1_X)$ is surjective with kernel ${w_3}_*(\Lambda^3 \Omega^1_X \boxtimes \mc{I}^2_{\Delta_2}) \simeq (E^{3, -3}_3)^{\perm_4}$; moreover, similarly to  proposition \ref{pps: mapC} one can show that the third differential $(d_3^{3, -3})^{\perm_4} : {w_3}_*(\Lambda^3 \Omega^1_X \boxtimes \mc{I}^2_{\Delta_2})  {\rOnto} 
 {w_4}_*(S^{3,1,1} \Omega^1_X)$ is surjective. Hence, at the page $E_3$ all differences with the $2$-dimensional case happen along the diagonal, so they are not relevant for the complex in the statement of theorem \ref{thm: inv4}. 
\end{remark}

To finish this subsection, we present two immediate byproducts of the proof of theorems \ref{thm: inv3} and  \ref{thm: inv4}. 
\begin{crl}\label{crl: invprod}For $n$ equal to $3$ or $4$, the $\perm_n$-invariants of the product ideal $\prod_{I \in E_{K_n}} \mc{I}_{\Delta_I} = \mc{I}_{\Delta_{12}} \cdot \cdots \cdot \mc{I}_{\Delta_{n-1, n}}$  and 
of the ideal $\mc{I}_{\Delta_n}$ coincide: 
$$ \Big( \mc{I}_{\Delta_{12}} \cdot \dots \cdot \mc{I}_{\Delta_{n-1, n}} \Big)^{\perm_n} \simeq (\mc{I}_{\Delta_n})^{\perm_n} = 
\Big(  \mc{I}_{\Delta_{12}} \cap \dots \cap \mc{I}_{\Delta_{n-1, n}} \Big)^{\perm_n} \;.$$
\end{crl}
\begin{proof}The invariants $\big( \mc{I}_{\Delta_{12}} \cdot \dots \cdot \mc{I}_{\Delta_{n-1, n}} \big)^{\perm_n}$ of the product of ideals coincide
with the term $\pi_*(E^{0,0}_{\infty})^{\perm_n}$ of the spectral sequence $\pi_*(E^{p,q}_1)^{\perm_n}$ above, since the abutment in degree $0$ is the tensor product $\mc{I}_{\Delta_{12}}  \tens \cdots \tens \mc{I}_{\Delta_{n-1, n}}$ and since 
$\pi_*(E^{0,0}_{\infty})$, being the first graded sheaf for the natural filtration on the abutment,  is the image of the natural morphism $( \mc{I}_{\Delta_{12}}  \tens \cdots \tens \mc{I}_{\Delta_{n-1, n}} )^{\perm_n} \rTo ( \FS_{X^n} )^{\perm_n} \simeq \FS_{S^n X} $. But it is evident from the proof of theorems \ref{thm: inv3} and \ref{thm: inv4}
that $\pi_*(E^{0,0}_{\infty})^{\perm_n} \simeq \pi_*(E^{0,0}_{2})^{\perm_n} \simeq (\mc{I}_{\Delta_n})^{\perm_n}$. 
\end{proof}
\begin{remark}It seems to us an interesting question whether the statement of corollary \ref{crl: invprod} is true for general $n$;  in some contexts (for example when taking inverse images) the product of ideals is better behaved that the intersection; therefore, knowing that, at least at level of invariants, the two coincide, might turn out useful in some applications.
\end{remark}

\begin{remark}When taking the tensor product of ideals, things are clearly different. Over $S^3 X$, we have the isomorphism $ \big(\mc{I}_{\Delta_{12}} \tens \mc{I}_{\Delta_{13}} \tens \mc{I}_{\Delta_{23}} \big)^{\perm_3} \simeq \big(\mc{I}_{\Delta_{12}} \cdot \mc{I}_{\Delta_{13}} \cdot \mc{I}_{\Delta_{23}} \big)^{\perm_3}$ only if $\dim X \leq 2$, while 
over $S^4 X$ the two sheaves are definitely not isomorphic also for $X$ of dimension $2$; their difference, in this case, is measured by the sheaf ${w_4}_* (K_X^3)$, as the next exact sequence proves: 
$$ 0 {\rTo} {w_4}_* (K_X^3) \rTo \Big (\mc{I}_{\Delta_{12}}  \tens \cdots \tens \mc{I}_{\Delta_{34}} \Big)^{\perm_4} 
\rTo (\mc{I}_{\Delta_{4}} )^{\perm_4} \rTo 0  \;.$$
\end{remark} 
\subsection{Twisting by the line bundle \protect $\mc{D}_L$}\label{subsec: DL}
Let now $F$ be a $\perm_n$-equivariant coherent sheaf over $X^n$. By definition of the line bundle $\mc{D}_L$ on the symmetric variety $S^n X$ (which is valid for $X$ of arbitrary dimension, see remark \ref{rmk: DL}),  using projection formula  and taking $\perm_n$-invariants we have the following equation: 
$$ (\pi^{\perm_n}_*F) \tens \mc{D}_L \simeq \pi^{\perm_n}_*( F \tens L^{\boxtimes n}) \;.$$
Because of this fact, all results proved in this section continue to work when we tensorize the sheaf $\mc{I}_{\Delta_n} $ with 
a line bundle of the form $L \boxtimes \cdots \boxtimes L$, or its invariants by $\mc{D}_L$. In particular we have 
that for $n$ equal to $3$ or $4$, the sheaf of invariants $\pi_*(\mc{I}_{\Delta_n} \tens L^{\boxtimes n })^{\perm_n} \simeq (\mc{I}_{\Delta_n})^{\perm_n} \tens \mc{D}_L$ is resolved by 
the complex $\mbb{I}_n^\bullet \tens \mc{D}_L$; in other words 
\begin{crl}\label{crl: inv34L}Let $X$ be a smooth algebraic variety. Over $S^3X$ and $S^4 X$, respectively, we have resolutions 
\begin{align*} 
0 \rTo \pi_*(\mc{I}_{\Delta_3} \tens L^{\boxtimes 3})^{\perm_3} \rTo  \mc{D}_L  \rTo & {w_2}_*(L^2 \boxtimes L) {\rTo} 
{w_3}_*( \Omega^1_X \tens L^3) \rTo 0 \\ 
 & \\
 0 \rTo \pi_*(\mc{I}_{\Delta_4} \tens L^{\boxtimes 4})^{\perm_3} \rTo   \mc{D}_L  \rTo & {w_2}_*(L^2 \boxtimes \mc{D}_L)_0 \rTo \\ \rTo &  {w_3}_*(( \Omega^1_X \tens L^3 \boxtimes L)   \tens \mc{I}_{\Delta_2})_0 {\rTo} {w_4}_*(S^3 \Omega^1_X \tens L^4) \rTo 0 
\end{align*}
\end{crl}

\section{Applications}
From now on $X$ will always be a smooth algebraic surface. 
\subsection{Cohomology of \protect $(\det L^{[n]})^2$ for low $n$. }

\begin{theorem}Let $X$ be a smooth quasi-projective surface and $L$ and $A$ two line bundles over $X$. 
Then, for $n = 3$ or $n =4$, the cohomology $H^*(X^{[n]}, (\det L^{[n]})^{\tens 2} \tens \mc{D}_A)$ is computed 
by the spectral sequence 
$$ E^{p,q}_1 := H^q(S^n X, \mbb{I}^p_n \tens \mc{D}^{\tens 2}_L \tens \mc{D}_A ) \;.$$
\end{theorem}\begin{proof}
By corollary \ref{crl: projectionformula} and by lemma \ref{lmm: inv2k-1}, we have 
$$  \B{R} \mu_* ( (\det L^{[n]} )^{\tens 2} \tens \mc{D}_A ) \simeq  \big ( \mc{I}_{\Delta_n}^2    \big)^{\perm_n} \tens \mc{D}_L^{\tens 2} \tens \mc{D}_A \simeq  \big ( \mc{I}_{\Delta_n}    \big)^{\perm_n} \tens \mc{D}_L^{\tens 2} \tens \mc{D}_A \;.$$Hence, applying the functor $\B{R}\Gamma$ on both sides, we get 
$$ H^*(X^{[n]}, (\det L^{[n]})^{\tens 2} \tens \mc{D}_A) \simeq \B{R} \Gamma \B{R} \mu_* ( (\det L^{[n]} )^{\tens 2} \tens \mc{D}_A ) \simeq H^* 
(S^n X, ( \mc{I}_{\Delta_n} )^{\perm_n} \tens \mc{D}_L^{\tens 2} \tens \mc{D}_A) \;.$$The spectral sequence in the statement computes the hypercohomology of the complex $\mbb{I}^\bullet_n \tens \mc{D}_L^{\tens 2} \tens \mc{D}_A$, which is a resolution of the sheaf $( \mc{I}_{\Delta_n}    \big)^{\perm_n} \tens \mc{D}_L^{\tens 2} \tens \mc{D}_A$, by theorems \ref{thm: inv3} and \ref{thm: inv4}. \end{proof}
An immediate application yields the computation of the Euler-Poincar\'e characteristic of~$(\det L^{[n]})^2~\tens~\mc{D}_A$. 
\begin{crl}\label{crl: Euler}Let $X$ a smooth projective surface and $L$ and $A$ line bundles over $X$. For $n =3$ and $n =4$ we have the following formulas for the Euler-Poincar\'e characteristic of $(\det L^{[n]})^2 \tens \mc{D}_A$ over the HIlbert scheme $X^{[n]}$:\begin{align*} \chi(X^{[3]}, (\det L^{[3]})^2 \tens \mc{D}_A ) = & \: \binom{\chi(L^2 \tens A)+2}{3} - \chi(L^4 \tens A^2)\chi(L^2 \tens A)+ \chi(\Omega^1_X \tens L^6 \tens A^3) \\ 
 \chi(X^{[4]}, (\det L^{[4]})^2 \tens \mc{D}_A ) = & \: \binom{\chi(L^2 \tens A)+3}{4} -\chi(L^4 \tens A^2)\binom{\chi(L^2 \tens A)+ 1}{2} + 
 \binom{\chi(L^4 \tens A^2)}{2} + \\ & \qquad + \chi(\Omega^1_X \tens L^6 \tens A^3)\chi(L^2 \tens A) - \chi(\Omega^1_X \tens L^8 \tens A^4) \\ & \qquad \qquad  \qquad \qquad -
 \chi(K_X \tens L^8 \tens A^4) - \chi(S^3 \Omega^1_X \tens L^8 \tens A^4)
\end{align*}
\end{crl}
We now mention an effective vanishing result for the cohomology of $(\det L^{[n]})^k \tens \mc{D}_A$, for any $n$~and~$k$. 
\begin{remark}We recall that a line bundle $L$ on a smooth projective surface $X$ is called $m$-very ample if, for any $\xi \in X^{[m+1]}$, the restriction map $H^0(L)  \rTo H^0(L_\xi)$ is surjective. The property of being $m$-very ample generalizes tha fact of being very ample, since ``$1$-very ample" means exactly ``very ample". After results of \cite{BeltramettiSommese1991} and \cite{CataneseGoettsche1990}, one can prove that $\det L^{[n]}$ is globally generated if $L$ is $(n-1)$-very ample, and that it is actually very ample if $L$ is $n$-very ample (see also \cite[Cor. 5.10]{Scalaarxiv2015}). 
\end{remark}
\begin{crl}Let $X$ be a smooth projective surface and $L$ and $A$ two line bundles over $X$ such that $L^k \tens A \tens K_X^{-1}$ is a product $\tens_{i=1}^k B_i$ of line bundles $B_i$, with $B_1$ $n$-very ample and $B_j$ $(n-1)$-very ample, for $j = 2, \dots, k$. Then we have the vanishing 
$$ H^i(X^{[n]}, ( \det L^{[n]})^{\tens k}  \tens \mc{D}_A)=0 \qquad \qquad \mbox{for  $i>0$} \;.$$In particular the statement is true if $L^k \tens A \tens K_X^{-1}$ 
is a product of $k(n-1)+1$ very ample line bundles over $X$. 
\end{crl}
\begin{proof}One has just to note that, in the hypothesis of the corollary, and since $\mc{D}_{K_X} \simeq K_{X^{[n]}}$, 
$( \det L^{[n]})^{\tens k}  \tens \mc{D}_A \simeq \big( \tens_{i=1}^k \det B_i^{[n]} \big) \tens K_{X^{[n]}}$. Then one uses 
Kodaira vanishing, since all line bundles $\det B_i^{[n]}$ are nef and $\det B_1^{[n]}$ is very ample. The last statement follows from the fact that a product of $l$ $1$-very ample line bundles is $l$-very ample  \cite{BeltramettiSommese1991}. 
\end{proof}

\subsection{Regularity of \protect $\mc{I}_{\Delta_n}^k$ and \protect $(\mc{I}^k_{\Delta_n})^{\perm_n}$}
\begin{notat}If $s \in \mbb{Q}$ is a rational number, we denote with $[s]$ its integral part, or round-down, and with $\lceil s \rceil$ its round-up. 
\end{notat}
\begin{theorem}\label{thm: regularity1}Let $X$ be a smooth projective surface and $L$ be a line bundle over $X$. Let $n \in \mbb{N}$, $n \geq 2$. Let $m \in \mbb{N}$ be an integer with the property 
\begin{itemize}\item[a)] $L^m \tens K_X^{-1} = \tens_{i=1}^{2 [(k+1)/2] }B_i$, with 
$B_1$ $n$-very ample and with $B_j$ $(n-1)$-very ample, for $j > 1$. \end{itemize}
Then we have the vanishing
$$H^i(S^n X, ( \mc{I}^k_{\Delta_n} )^{\perm_n} \tens \mc{D}^m_{L}) = 0 \qquad \qquad \mbox{for $i>0$} \;.$$If, moreover, $L$ is very ample on $X$,  then $(\mc{I}^k_{\Delta_n})^{\perm_n}$ is $(m+2n)$-regular with respect 
to $L \boxtimes \cdots \boxtimes L$. Therefore, if $m_0$ is the minimum of $m$ such that condition $a)$ is true, we have the  upper bound 
$$ \reg((\mc{I}^k_{\Delta_n})^{\perm_n}) \leq m_0 + 2n $$for the regularity of the ideal sheaf $(\mc{I}^k_{\Delta_n})^{\perm_n}$ with respect to 
$\mc{D}_L$.
\end{theorem}\begin{remark}We note that condition a) is true in particular if $L^m \tens K_X^{-1}$ is a tensor product $2n [(k+1)/2] - 2[(k+1)/2] +1$ very ample line bundles over $X$. If this holds and, additonaly, the line bundle $L$ is very ample, then the ideal $(\mc{I}^k_{\Delta_n})^{\perm_n} $ is $(m+2n)$-regular with respect to $\mc{D}_L$. 
\end{remark}
\begin{proof}[Proof of theorem \ref{thm: regularity1}]
For the first statement, by lemma \ref{lmm: inv2k-1} and by corollary \ref{crl: projectionformula}, we have that
$$(\mc{I}^k_{\Delta_n})^{\perm_n} \tens \mc{D}^m_L \simeq (\mc{I}^{2  [(k+1)/2]}_{\Delta_n})^{\perm_n}\tens \mc{D}^m_L
 = \B{R} \mu_*(  ( \det \FS_X^{[n]}) ^{\tens 2 [(k+1)/2]} \tens \mc{D}^m_L) \;.$$Hence
 $$ H^i(S^n X, (\mc{I}^k_{\Delta_n})^{\perm_n} \tens \mc{D}^m_L) \simeq H^i(X^{[n]}, \mc{D}_{L^m}( - 2 [(k+1)/2] e))$$and we conclude by 
 \cite[Propositon 5.14]{Scalaarxiv2015}. 

For the second part, from the first statement and from the fact that, under the hypothesis, $\mc{D}_L$ is very ample,  
we immediately have that $H^i(S^n X, ( \mc{I}^k_{\Delta_n} )^{\perm_n} \tens \mc{D}_L^{\tens m+l} ) = 0$ for $i>0$ and for $l \geq 0$. 
Consequently, $( \mc{I}^k_{\Delta_n} )^{\perm_n}$ has to be $(m+2n)$-regular with respect to $\mc{D}_L$. 
\end{proof}

\begin{theorem}\label{thm: regularity2}Let $X$ be a smooth projective surface and $L$ be a line bundle over $X$. Let $n \in \mbb{N}$, $2 \leq n \leq 7$. Let $m \in \mbb{N}$  be an integer with the property 
\begin{itemize}\item[b)] $L^m \tens K_X^{-1} = \tens_{i=1}^{k+1}B_i$, with 
$B_1$ $n$-very ample and with $B_j$ $(n-1)$-very ample, for $j > 1$. \end{itemize}
Then we have the vanishing 
$$ H^i(X^n, \mc{I}^k_{\Delta_n} \tens (L^m \boxtimes \cdots \boxtimes L^m) ) = 0 \qquad \qquad \mbox{for $i>0$} \;.$$If, moreover, $L$ is very ample, 
then the ideal sheaf $\mc{I}^k_{\Delta_n}$ is $(m+2n)$-regular with respect to $L \boxtimes \cdots \boxtimes L$. Therefore, if $m_0$ is the minimum of $m$ such that condition $b)$ is true, we have the  upper bound
$$ \reg(\mc{I}^k_{\Delta_n}) \leq m_0 + 2n $$for the regularity of the ideal sheaf $\mc{I}^k_{\Delta_n}$ with respect to 
$L \boxtimes \cdots \boxtimes L$. \end{theorem}
\begin{remark}We note that condition b) is true if $L^m \tens K_X^{-1}$ is a tensor product of $(k+1)n -k$ very ample line bundles over $X$. If this holds and, additionaly, $L$ is very ample, then, for $2 \leq n \leq 7$, the ideal sheaf $\mc{I}^k_{\Delta_n}$ is $(m+2n)$ regular with respect to $L \boxtimes \cdots \boxtimes L$. 
\end{remark}
\begin{proof}[Proof of theorem \ref{thm: regularity2}]The first statement follows immediately by \cite[Propositon 5.15]{Scalaarxiv2015}
and by the fact that, by \cite[Theorem 2.12]{Scala2015isospectral},  $B^n$ has log-canonical singularities for $n \leq 7$. As for the second, its proof is analogous
to the proof of the similar statement for the regularity of $(\mc{I}^k_{\Delta_n})^{\perm_n}$ in theorem \ref{thm: regularity1}. 
\end{proof}

The previous regularity results are nicer to state when $X$ has Picard number one. 
\begin{crl}Let $X$ be a smooth  projective surface with Picard group $\Pic(X) \simeq \mbb{Z} B$, where $B$ is the ample generator. Let $r$ be the minimum positive power of $B$ such that $B^r$ is very ample. Suppose, moreover, that $K_X \simeq B^w$, for some integer $w$. Then we have the following. 
\begin{itemize}
\item The sheaf $(\mc{I}_{\Delta_n}^k)^{\perm_n}$ is $(m+2n)$-regular with respect to $\mc{D}_{B^r} $, if $m \geq 2n [(k+1)/2] - 2[(k+1)/2] +1
+w/r $. Hence, with respect to $\mc{D}_{B^r}$, 
$$ \reg( (\mc{I}_{\Delta_n}^k)^{\perm_n}) \leq  2n ( [(k+1)/2] +1) - 2[(k+1)/2] +1
 +\lceil w/r  \rceil  \;.$$
\item 
 If $2 \leq n \leq 7$, the sheaf $\mc{I}_{\Delta_n}^k$ is $(m+2n)$-regular with respect to $B^r \boxtimes \cdots \boxtimes B^r$, if $m \geq (k+1)n-k +w/r $. Hence, with respect to $B^r \boxtimes \cdots \boxtimes B^r$,  
$$ \reg(\mc{I}_{\Delta_n}^k) \leq (k+3)n-k +\lceil w/r  \rceil  \;.$$
\end{itemize}
\end{crl}
\begin{remark}If $X = \mbb{P}_2$, taking $B= \FS_{\mbb{P}_2}(1)$, we can say, more simply, that, 
$ \reg( (\mc{I}_{\Delta_n}^k)^{\perm_n}) \leq  2n ( [(k+1)/2] +1) - 2[(k+1)/2] -2$ and,  
for $2 \leq n \leq 7$,  $\reg(\mc{I}_{\Delta_n}^k) \leq (k+3)(n-1)$. 
\end{remark}
\begin{remark}The results proven in this subsection for $\mc{I}^k_{\Delta_n}$ are valid for $2 \leq n \leq 7$. We expect them to hold also for $n = 8$, since $B^n$ should have log-canonical singularities also in this case \cite[Conjecture~2]{Scala2015isospectral}. However, we don't know, and it seems to us an interesting question, if the previous bound is still a good upper bound for 
the regularity of $\mc{I}_{\Delta_n}^k$, for general $n$, or, if not, what would be a good one. 
The proof we gave here can't go through in general since we proved that $B^n$ does not have log-canonical singularities for $n \geq 9$ \cite[Theorem 2.12]{Scala2015isospectral}. 
\end{remark}

\subsection{The sheaves \protect $\mc{L}^\mu(-2\mu\Delta)$.}\label{subsec: sheaves}
Let $n, k \in \mbb{N}$, $n \geq 2$, and  let $\mu$ be a partition of $k$ of length $l(\mu) \leq n$. The symmetric group $\perm_n$ acts naturally on the set of compositions of $k$ supported in $\{1, \dots, n \}$. 
Indicate with $L^\mu$ the line bundle on $X^n$ defined by 
$L^\mu := \tens_{i=1}^n p_i^* L^{\tens \mu_i}$, where $p_i: X^n \rTo X$ is the projecton onto the $i$-th factor. 
In \cite{Scalaarxiv2015} we defined sheaves $\mc{L}^\mu (-2 \mu \Delta)$ over the symmetric variety $S^n X$ as 
$$ \mc{L}^\mu (- 2 \mu \Delta) := \pi_* \Big(L^\mu \tens \bigcap_{1 \leq i < j  \leq l(\mu)} \mc{I}_{\Delta_{i,j}}^{2 \mu_j} \Big)^{\Stab_{\perm_n}(\mu)} $$where we see $\mu$ as a composition of $k$ supported in $\{1, \dots, n \}$. 
If $\mu_2 = \dots = \mu_{l(\mu)} = l$, we denote more simply this sheaf with $\mc{L}^{\mu}(-2 l \Delta)$. 
We also use sheaves $\mc{L}^{\mu}(-m \Delta)$, for an integer $m \in \mbb{N}$, whose definition is analogous.  
The interest in such sheaves comes from the fact that we believe they could be in all generality 
 the graded sheaves for a natural filtration on the direct image $\mu_*(S^k L^{[n]})$ 
of symmetric powers of tautological bundles on $X^{[n]}$ via the Hilbert-Chow morphism.
\begin{remark}The sheaf $\mc{L}^\mu (- 2 \mu \Delta)$ over the symmetric variety $S^nX$ is closely related to the same sheaf over $S^{l(\mu)} X$; more precisely, if $v_l$ is the finite morphism $v_l: S^lX \times S^{n-l} X \rTo S^n X$, sending the couple of $0$-cycles $(x,y)$ to $x+y$, we can write the isomorphism of sheaves over $S^n X$:
$$ \mc{L}^\mu (- 2 \mu \Delta) \simeq {v_{l(\mu)}}_* (\mc{L}^\mu (- 2 \mu \Delta) \boxtimes \FS_{S^{n-l(\mu)} X}) \;,$$and, in general, if $A$ is a line bundle over $X$, then
$ \mc{L}^\mu (- 2 \mu \Delta) \tens \mc{D}_A \simeq {v_{l(\mu)}}_* (\mc{L}^\mu (- 2 \mu \Delta) \tens \mc{D}_A \boxtimes \mc{D}_A )$. 
\end{remark}
\begin{remark}Let  $\lambda = (r, \dots, r)$ and set $l = l(\lambda)$. Then  it is immediate to see that the sheaf 
$\mc{L}^{\lambda}(-2r \Delta)$ over $S^{l} X$ 
is isomorphic to $$\mc{L}^{\lambda}(-2r \Delta) \simeq (\mc{I}^{2r}_{\Delta_l} )^{\perm_l} \tens \mc{D}_{L^r} \simeq \mu_*((\det \FS_{X}^{[l]})^{\tens 2r}) \tens \mc{D}_{L^r}  \;.$$If $n \geq l$, over $S^n X$, in general we have that 
$\mc{L}^{\lambda}(-2r \Delta) \tens \mc{D}_A \simeq {v_l}_*( (\mc{I}^{2r}_{\Delta_l} )^{\perm_l} \tens \mc{D}_{L^r \tens A} 
\boxtimes \mc{D}_A ) \simeq {v_l}_*( \mu_*((\det \FS_{X}^{[l]})^{\tens 2r} ) \tens \mc{D}_{L^r \tens A} \boxtimes \mc{D}_A) $. 
\end{remark}We come to our results. Denote the partition $(1, \dots, 1)$ of $l$ with $1^{l}$ (in exponential notation). As a consequence of the previous remark, as well as of theorems \ref{thm: inv3}, \ref{thm: inv4}, and of subsection \ref{subsec: DL}, we obtain the following
\begin{crl}For $l = 3, 4$, and $n \geq l$, the sheaf $\mc{L}^{1^l}(-2 \Delta) \tens \mc{D}_A$ over $S^n X$  is resolved by the complex ${v_l}_*(\mbb{I}^\bullet_l \tens \mc{D}_{L \tens A} \boxtimes \mc{D}_A)$: 
$$ \mc{L}^{1^l}(-2 \Delta) \tens \mc{D}_A \simeq^{\rm qis} {v_l}_*(\mbb{I}^\bullet_l \tens \mc{D}_{L \tens A} \boxtimes \mc{D}_A) \;.$$
\end{crl}Similarly, thanks to proposition \ref{pps: 211} we obtain 
\begin{crl}The sheaf $\mc{L}^{2,1,1}(- \Delta) \tens \mc{D}_A$ is resolved over $S^n X$ by the complex
\begin{multline}  0 {\rTo} {w_3}_*({v_1}_*(L^2 \tens A \boxtimes \mc{D}_{L \tens A}) \boxtimes \mc{D}_A ) {\rTo} \\ {\rTo }{w_3}_*([ 
{v_2}_*(L^3 \tens A^2 \boxtimes L \tens A) \oplus {v_2}_*(L^2 \tens A^2 \boxtimes L^2 \tens A)
]_0\boxtimes \mc{D}_A) {\rTo}  \\ {\rTo} {w_3}_*( {v_3}_*( \Omega^1_X \tens L^4 \tens A^3) \boxtimes \mc{D}_A) \rTo 0  \;.\end{multline}\end{crl}

It would be now immediate to give a formula for the Euler-Poincar\'e characteristic of $\mc{L}^{1^l}(-2 \Delta) \tens \mc{D}_A$, using corollary \ref{crl: inv34L} or corollary \ref{crl: Euler}, and of $\mc{L}^{2,1,1}(- \Delta) \tens \mc{D}_A$ . We leave this to the reader. 

\paragraph{The sheaf \protect $\mc{L}^{2,1,1}(-2 \Delta)$}
To finish this section we will describe the sheaf $\mc{L}^{2,1,1}(-2 \Delta)$, which is important for the work \cite{Scalaarxiv2015}. 
If $A \subseteq \{1, \dots, n \}$ and if $\mu$ is a composition of some integer $l$ supported in $\{1, \dots, n \}$, we define $\mu_{A}$ as 
the compositon  coinciding with $\mu$ over the set $A$, and with $0$ over $\{1, \dots, n \} \setminus A$. We define with $|\mu_A| = 
\sum_{i=1}^n \mu_A(i)$. 
In \cite[Remark 4.6]{Scalaarxiv2015} we defined a natural $\Stab_{\perm_n}(\mu)$-equivariant differential 
$$ d^l_\Delta: L^\mu \tens \mc{I}^l_{\Delta_{l(\mu)}}  \rTo \bigoplus_{\substack{|I|=2 \\ I \subseteq \{1, \dots, l(\mu) \}} } L^\mu \tens \mc{I}^l_{\Delta_I}/\mc{I}^{l+1}_{\Delta_I} \simeq  \bigoplus_{\substack{|I|=2 \\ I \subseteq \{1, \dots, l(\mu) \}} } (S^l \Omega^1_X \tens L^{|\mu_I|})_{I} \tens L^{\mu_{\bar{I}}}$$and an invariant version over $S^n X$: 
$$ d^l_\Delta: \mc{L}^\mu(-l \Delta) := \pi_*(L^\mu \tens \mc{I}^l_{\Delta_{l(\mu)}})^{\Stab_{\perm_n}(\mu)} \rTo \pi_* \Big( 
\bigoplus_{\substack{|I|=2 \\ I \subseteq \{1, \dots, l(\mu) \}} } (S^l \Omega^1_X \tens L^{|\mu_I|})_{I} \tens L^{\mu_{\bar{I}}} \Big)^{\Stab_{\perm_n}(\mu)} $$whose kernel is $\mc{L}^\mu(-(l+1)\Delta)$. 
Denote with $\mc{K}^{1}_{(1)(1)}(- l \Delta)$ the sheaf 
$$ \mc{K}^{1}_{(1)(1)}(- l \Delta) := \pi_*(  (  \Omega^1_X \tens L^3)_{\{12\}} \tens p_3^*L  \tens \mc{I}^l_{\Delta_{23}} )^{ \perm(4, \dots, n )} \;; $$it will be denoted it just with $\mc{K}^1_{(1)(1)}$ if $l = 0$. It is clear that, for $\mu = (2,1,1)$, 
\begin{align*}  \pi_* \Big( \bigoplus_{\substack{|I|=2 \\ I \subseteq \{1, \dots, l(\mu) \}}  } (\Omega^1_X \tens L^{|\mu_I|})_{I} \tens L^{\mu_{\bar{I}}} \Big)^{\Stab_{\perm_n}(\mu)}  & \: \simeq \mc{K}^1_{(1)(1)} 
\oplus \pi_*\big( (\Omega^1_X \tens L^{2})_{\{23\}} \tens p_1^*L^2  \big)^{\perm(2,3) \times \perm(4, \dots, n))} \\ & \: \simeq \mc{K}^1_{(1)(1)} 
\;, \end{align*}since $\perm(2,3)$ acts with a sign on the sheaf $(\Omega^1_X \tens L^2)_{\{23\}}$. 
With these notations, we can prove the following fact. 

\begin{pps}\label{prop: exactL211}We have the exact sequence over $S^n X$: 
\begin{multline} 0 \rTo \mc{L}^{2,1,1}(-2 \Delta) \rTo \mc{L}^{2,1,1}(-\Delta) \rTo \mc{K}^1_{(1)(1)}(-2 \Delta) \rTo \\ {\rTo} {w_3}_*( (S^3 \Omega^1_X \tens L^4) \boxtimes \FS_{S^{n-3}X} ) \rTo 0 \;,\end{multline}where the third map 
is the differential $d^1_{\Delta}$ and the fourth one is the composition: 
\begin{multline} \K^1_{(1)(1)} (-2 \Delta)  = \pi_*\Big[ ( (\Omega^1_X \tens L^3)_{\{12\}}  \tens p_3^* L  \tens \mc{I}^2_{\Delta_{23}} \Big]^{ \perm(\{4, \dots, n\}) } \! \! \! \! \! \! \!\rTo \\ \rTo \pi_*\Big[ ( (\Omega^1_X \tens L^3)_{\{12\}}  \tens p_3^* L  \tens \mc{I}^2_{\Delta_{23}} / \mc{I}^3_{\Delta_{23}} \Big]^{\perm(\{4, \dots, n\})} \simeq \\ 
\simeq {w_3}_* \big ( (\Omega^1_X \tens S^2 \Omega^1_X \tens L^4) \boxtimes \FS_{S^{n-3}X }  \big)  {\rTo} 
{w_3}_* \big (   (S^3 \Omega^1_X \tens L^4) \boxtimes \FS_{S^{n-3}X}  \big)  \;.\end{multline}
\end{pps}
\begin{proof}
It is clear that, by construction, the differential $d^1_\Delta$ takes values in $\K^1_{(1)(1)}(-2 \Delta)$; it is also clear that 
$\mathcal{L}^{2,1,1}(-2 \Delta)$ is the kernel of the third map. Moreover, by construction, the map $\mc{K}^1_{(1)(1)}(-2 \Delta) {\rTo} {w_3}_*( (S^3 \Omega^1_X \tens L^4) \boxtimes \FS_{S^{n-3}X} ) $ is surjective. Hence it remains to prove that 
the sequence 
\begin{equation}\label{eq: sequenceK} \mc{L}^{2,1,1}(-\Delta) \rTo \mc{K}^1_{(1)(1)}(-2 \Delta) {\rTo} {w_3}_*( (S^3 \Omega^1_X \tens L^4) \boxtimes \FS_{S^{n-3}X} ) \end{equation}is exact. By GAGA principle it is actually sufficient to prove the exactness 
of the sequence (\ref{eq: sequenceK}) for $X = \mbb{C}^2$ and $L$ trivial.  More precisely, let $(S^n X)_{\rm an}$ be the complex analytic space associated to the complex algebraic variety $S^n X$. Then the natural morphism $((S^n X)_{\rm an}, \FS_{(S^n X)_{\rm an}}) \rTo (S^n X, \FS_{S^n X})$ is faithfully flat, by GAGA principle. This implies, in particular, that the sequence (\ref{eq: sequenceK}) is exact over $S^n X$ if and only if the induced sequence of complex analytic sheaves  
\begin{equation}\label{eq: sequenceKan}  \mc{L}^{2,1,1}(-\Delta)_{\rm an} \rTo \mc{K}^1_{(1)(1)}(-2 \Delta)_{\rm an} {\rTo} {w_3}_*( (S^3 \Omega^1_X \tens L^4) \boxtimes \FS_{S^{n-3}X} )_{\rm an}
\end{equation}is exact over $(S^n X)_{\rm an}$. But 
this is holds if and only  if it holds for an arbitrary small\footnote{here we mean: if it holds over 
any open set $V_j$ of an open cover $\{ V_j \}_j$ of $(S^n X)_{\rm an}$, where each $V_j$ is chosen to be sufficiently small} open set of $(S^n X)_{\rm an}$ in the complex topology. Now a sufficiently small open set of $(S^n X)_{\rm an}$ in the complex topology is always biholomorphic to a sufficiently small open set of $(S^n \mbb{C}^2)_{\rm an}$, in the complex topology. Hence it is sufficient to prove that the sequence (\ref{eq: sequenceKan}) is exact analytically over $(S^n \mbb{C}^2)_{\rm an}$ and $L$ trivial, but this is equivalent, invoking GAGA principle again, to proving the same fact algebraically over $S^n \mbb{C}^2$ and $L$ trivial.

It is also easy to see that it is sufficient to prove the statement for $n=3$. 
In this case set $A = \mbb{C}[x,y]$, $A^{\tens 3} = \mbb{C}[x_1, x_2 , x_3, y_1, y_2, y_3]$. 
Identifying coherent sheaves with modules, 
it is sufficient to prove that 
the sequence of $S^3A$-modules
\begin{equation}\label{eq: sequenceL211} (\mc{I}_{\Delta_3})^{\perm(\{2,3\})} \rTo^{d^1_\Delta} ( \Omega^1_{A} \tens_{\mbb{C}} A)(-2 \Delta) \rTo S^3 \Omega^1_A \end{equation}is exact, where 
we wrote briefly $( \Omega^1_{A} \tens_{\mbb{C}} A)(-2 \Delta) $ for $( \Omega^1_{A} \tens_{\mbb{C}} A) \tens_{A \tens A} \mc{I}^2_{\Delta_2}$, where $\mc{I}_{\Delta_2}$ is the ideal of the diagonal in $A \tens A$ and where $A \tens A$ acts componentwise on $\Omega^1_{A} \tens_{\mbb{C}} A$.

The ideal
$\mc{I}_{\Delta_3}$ of the big diagonal in $X^3$ equals the ideal 
$ \langle \mc{I}_{\Delta_{12}}  \mc{I}_{\Delta_{13}}  \mc{I}_{\Delta_{13}}, q \rangle$, where $q$ is the quadraric polynomial
$q = (x_2-x_1)(y_3-y_1) - (y_2 - y_1)(x_3-x_1)$. The  quadric $q$ is anti-invariant for $\perm(\{2,3\})$, hence 
$\mc{I}_{\Delta_3}^{\perm(\{2,3\})} = ( \mc{I}_{\Delta_{12}}  \mc{I}_{\Delta_{13}}  \mc{I}_{\Delta_{13}} )^{\perm(\{2,3\})}$. Writing down 
all nine degree-$3$ generators of $\mc{I}_{\Delta_{12}}  \mc{I}_{\Delta_{13}}  \mc{I}_{\Delta_{13}}$ and taking $\perm(\{2,3\})$-invariants, 
we get that $\mc{I}_{\Delta_3}^{\perm(\{2,3\})}$ is generated, up to elements of degree $4$, by 
$q(x_3 -x_2)$ and $q(y_3 -y_2)$. Now it is easy to see that 
the image of $d^1_\Delta$ contains $(\Omega^1_A \tens_{\mbb{C}} A)(-3 \Delta)$. 
Indeed, if $ \alpha, \beta \in \mbb{N}$, $\alpha + \beta >1$, we have
\begin{align*} 
d^1_\Delta \Big[ [(x_2 -x_1)(x_3 -x_1)^\alpha & (y_3 -y_1)^\beta - \\ &  (x_3-x_1)(x_2-x_1)^\alpha (y_2 -y_1)^\beta ](x_3-x_2) \Big] = (x_3 -x_1)^{\alpha+1}(y_3 -y_1)^{\beta} dx \\
d^1_\Delta \Big[ [(x_2 -x_1)(x_3 -x_1)^\alpha & (y_3 -y_1)^\beta - \\ & (x_3-x_1)(x_2-x_1)^\alpha (y_2 -y_1)^\beta ](y_3-y_2) \Big] = (x_3 -x_1)^{\alpha}(y_3 -y_1)^{\beta+1} dx \\
d^1_\Delta \Big[ [(y_2 -y_1)(x_3 -x_1)^\alpha & (y_3 -y_1)^\beta - \\ &  (y_3-y_1)(x_2-x_1)^\alpha (y_2 -y_1)^\beta ](x_3-x_2) \Big] = (x_3 -x_1)^{\alpha+1}(y_3 -y_1)^{\beta} dy \\
d^1_\Delta \Big[ [(y_2 -y_1)(x_3 -x_1)^\alpha & (y_3 -y_1)^\beta - \\ & (y_3-x_1)(x_2-x_1)^\alpha (y_2 -y_1)^\beta ](y_3-y_2) \Big] = (x_3 -x_1)^{\alpha}(y_3 -y_1)^{\beta+1} dy
\end{align*}Consider now an element of the form 
$$ \tau = (a dx + bdy) (x_3 - x_1)^2 + (cdx  + e dy) (x_3 -x_1)(y_3 -y_1) + (fdx + g dy)(y_3 -y_1)^2 $$in 
$(\Omega^1_A \tens_{\mbb{C}} A)(-2 \Delta)$. The image of $\tau$ in $S^3 \Omega^1_A$ is 
$a dx^3 + (b +c )dx^2 dy + (e +f) dx dy^2 + g dy^3$. If $\tau$ is in the kernel of the map $(\Omega^1_A \tens_{\mbb{C}} A)(-2 \Delta) \rTo S^3 \Omega^1_A$, then $a = g =0$ and $b=-c$, $e =-f$ and $\tau$ is of the form
$\tau = b (x_3 -x_1)^2 dy - b (x_3 -x_1)(y_3 -y_1) dy + e (x_3 -x_1)(y_3-y_1)dy - e (y_3 -y_1)^2 dx$. But is now easy to see 
that $\tau$ is a linear combination of $d^1_\Delta [ q (x_3 -x_2) ]$ and $d^1_\Delta[q (y_3 -y_2)]$ and hence in the image of 
$d^1_\Delta$. These facts show that the sequence (\ref{eq: sequenceL211}) is exact. 
\end{proof}

\appendix
\section{Appendix: Determination of higher differentials in the spectral sequence of invariants 
}

We will determine here explicitely the higher differentials in the spectral sequence of invariants $(E^{p,q}_1)^{\perm_n}$ for $n = 3, 4$, appearing in theorem \ref{thm: inv3}, \ref{pps: 211} and \ref{thm: inv4}. To fix ideas, we will always suppose that $\dim X = 2$, 
but everything can be straightforwardly generalized in the case $X$ is a smooth algebraic variety of arbitrary dimension. 
\begin{remark}\label{rmk: A1}In order to prove that a certain higher differential $d_r^{\perm_n} : (E^{p,q}_r)^{\perm_n} \rTo (E^{p+r,q-r+1}_r)^{\perm_n}$ has a certain expression, we  define explicitely another morphism between the same coherent sheaves of invariants, say $D_r: (E^{p,q}_r)^{\perm_n} \rTo (E^{p+r,q-r+1}_r)^{\perm_n}$,  
and then we prove that the two maps coincide. By GAGA principle (as done in proposition \ref{prop: exactL211}), or, alternatively,  by localization an completion, in order to compare the two maps we can always reduce the problem to the case $X = \mbb{C}^2$, where computations are much easier. 
\end{remark}
\begin{remark}\label{rmk: A2}Over $X^n = (\mbb{C}^2)^n = \Spec \mbb{C}[x_1, \dots, x_n, y_1, \dots, y_n]$ we resolve the sheaves $\mc{K}^1_{I}$, $I \subseteq \{1, \dots, n \}$ 
with Koszul complexes $K^\bullet_I(F_I, s_I)$, where $F_I$ is the trivial rank $2$ bundle, with global frame $\gamma_{I}, \delta_{I}$, and with 
global section $s_I = x_I \gamma_I + y_I \delta_I$, where, if $I = \{i, j \}$, $i <j$,  we denoted briefly $x_I $ and $y_I$  the differences $x_j - x_i$ and  $y_j-y_i$, respectively.  
We then build a term by term free resolution $R^{\bullet, \bullet}_I$ of the complex $\mc{K}^\bullet_I$. The spectral sequence $E^{p,q}_1$ can be seen as the spectral sequence associated to the bicomplex $L^{\bullet, \bullet} := \Tens_{|I|=2}R^{\bullet, \bullet}_I$, where the tensor product is taken respecting the lexicographic order of the multi-indexes $I$: we see it as a bicomplex with respect to the sum of the first indexes and the sum of the second. We denote with $\partial$ and $\delta$  the (commuting) horizontal and  vertical differentials, respectively. 
It is straightforward to see that, remembering the notation used in the proof of proposition \ref{pps: iGamma}, 
\begin{align*} L^{p, \bullet} \simeq \; & \bigoplus_{I_1, \dots , I_p \subseteq \{1, \dots,  n \} } K^\bullet(F_{I_1}, s_{I_1}) \tens 
\cdots \tens K^\bullet(F_{I_p}, s_{I_p}) \\ \simeq \; &   \bigoplus_{I_1, \dots , I_p \subseteq \{1, \dots, n  \} } K^\bullet(F_{I_1} \oplus \cdots \oplus F_{I_p}, s_{I_1} \oplus \cdots \oplus s_{I_p} ) \\ \; & \; \bigoplus_{\Gamma \in \mc{G}_{p,n} } K^\bullet(F_\Gamma, s_\Gamma) \end{align*}where the direct sums are over distinct $I_1, \dots, I_p \subseteq \{1, \dots, n \}$, in lexicographic order. 
\end{remark}
\begin{remark}\label{rmk: A3}Let $X = \mbb{C}^2$. After the description of $Q_\Gamma$ given in 
 \ref{pps: torgamma}, it is practical to think of $F_I$ as $\mbb{C}^2 \tens \rho_I$ with $\gamma_I  = e_1 \tens e_I$ and $\delta_I = e_2 \tens e_I$. The quotient bundle $Q_\Gamma$ can then be seen as 
 $$ Q_\Gamma \simeq \FS_{\Delta_{\Gamma}} \tens ( \mbb{C}^2 \tens q_\Gamma)  \;.$$
 The isomorphism $Q^*_\Gamma \simeq (\Omega^1_X \tens q_\Gamma)_\Gamma \simeq \FS_{\Delta_{\Gamma}} \tens (\Omega^1_{\mbb{C}^2} \tens q_\Gamma)$ is given by identifying, over $\Delta_\Gamma$, 
 the vector $e_1 \tens v$, for $v \in q_\Gamma$,  with $dx \tens v$, with  and $e_2 \tens v$ with $dy \tens v$. 
 Of course, since every bundle here is trivial and the representation 
 $q_\Gamma$ is autodual, $Q_\Gamma \simeq Q^*_\Gamma$. 
\end{remark}
\begin{notat}\label{notat: A4}Suppose that the graph $\Gamma \in \mc{G}_{p,n}$ contains an oriented $3$-cycle $K_3(H)$. We will identify the $3$-cycle 
$K_3(H)$ with the sequence of its vertices written in order according to the orientation. Hence we write $e_H$ for the corresponding vector in $q_\Gamma$. 
Moreover, we write  $\gamma_{H} = e_1 \tens e_H$ and $\delta_H = e_2 \tens e_H$ the corresponding vectors in $\mbb{C}^2 \tens q_\Gamma$, which can be seen as elements of $Q_\Gamma$, by the previous remark 
\end{notat}

\paragraph{Determination of the map $D$}

\begin{notat}For $J \subseteq \{1, \dots, n \}$ a cardinality $2$ multi-index and for $i \in \mbb{N}^*$, we denote with $d^i_{\Delta_J}$ the $i$-th order differential 
$ d^i_{\Delta_J}: \mc{I}_{\Delta_J}^i  \rTo \mc{I}^i_{\Delta_J}/\mc{I}^{i+1}_{\Delta_J}$ and with $r_{J}$ the restriction
$ r_{J}: \FS_{X^n} \rTo \FS_{\Delta_J}$. Sometimes we see the operator $d^i_{\Delta_J}$ as taking values in $(S^i \Omega^1_X)_{J}$, via the isomorphism $ \mc{I}^i_{\Delta_J}/\mc{I}^{i+1}_{\Delta_J} \simeq S^i N^*_{\Delta_J} \simeq (S^i \Omega^1_X)_{J}$. 
\end{notat}
\begin{remark}\label{rmk: normalbundles}Let $Y, Z$ be smooth subvarieties of a smooth variety $M$ intersecting transversely in a smooth subvariety $Y \cap Z$. It is easy to show that $N_{Y \cap Z / Y} \simeq N_{Z/M} \trest_{Y \cap Z}$. 
\end{remark}

\begin{remark}Let $I$ and $J$ be two distinct cardinality $2$ multi-indexes in $\{1, \dots, n \}$, such that $I \cap J \neq \emptyset$. Let $I = \{i, j \}$, $J = \{j, k \}$, with $j < k$ and $i \neq k$. 
The diagonal $\Delta_J$ will be identified to $X^{\{1, \dots, n\} \setminus \{k \}} \simeq X^{n-1}$; therefore
$\Delta_I$ defines a pairwise diagonal  in $\Delta_J \simeq X^{n-1}$, that we still indicate with $\Delta_I$. 
By remark \ref{rmk: normalbundles}  we can define the composition 
$$ d_{\Delta_I} \circ r_J : \mc{I}_{\Delta_{I \cup J}} \rTo^{r_J} \mc{I}_{\Delta_{I \cup J}} /\mc{I}_{\Delta_J} \simeq \mc{I}_{\Delta_I /\Delta_J} \rTo^{d_{\Delta_I}} N^*_{\Delta_I /\Delta_J} \simeq N^*_{\Delta_I /X^n} \trest_{\Delta_{I \cup J} } \;.$$

\end{remark}
\begin{remark}Recall that $E^{3, -1}_2 \simeq \oplus_{H \subseteq \{1, \dots, n \}, |H| = 3}Q^*_{K_3(H)} \tens \cap_{J \not \subseteq H} \mc{I}_{\Delta_J}$. Now each of the sheaves $Q^*_{K_3(H)} \tens \cap_{J \not \subseteq H} \mc{I}_{\Delta_J}$ can be indentified with 
 $(\Omega^1_X \boxtimes \mc{I}_{\Delta_{n-2}})_{K_3(H)}$ (see notations \ref{notat: omegadelta} and \ref{notat: igamma}). 
 Hence elements in $Q^*_{K_3(H) }  \cap_{J \not \subseteq H} \mc{I}_{\Delta_J} $ can be identified with differential forms in $\Omega^1_X \boxtimes \FS_{X^{n-3}}$ 
 over the product $X \times X^{n-3}$ vanishing on the diagonal $\Delta_{n-2}$ in $\Delta_H \simeq X \times X^{n-3}$. For brevity's sake, we will denote the sheaf $(\Omega^1_X \boxtimes \mc{I}_{\Delta_{n-2}})_{K_3(H)}$ more briefly with $(\Omega^1_X \boxtimes \mc{I}_{\Delta_{n-2}})_{H}$. 
\end{remark}Recalling notation $ (\oplus_{|I|=2} \FS_{\Delta_I} )_0 $ for the kernel of $d_1:  \oplus_{|I|=2} \FS_{\Delta_I} \rTo \oplus_{I \neq J, |I|=|J|=2} \FS_{\Delta_{I \cup J}}$, we have the following
\begin{lemma}\label{lmm: d2a}
The morphism of coherent sheaves 
$\tilde{D}:  (\oplus_{|I|=2} \FS_{\Delta_I} )_0 \rTo \oplus_{|H|=3} (\Omega^1_X \boxtimes \mc{I}_{\Delta_{n-2}})_{H}$, given by
$$ \tilde{D}( (f_L)_{L})_{H} = d_{\Delta_{I}}r_{J}( \tilde{f}_{J} -  \tilde{f}_{I}) -   d_{\Delta_{I}} r_{K} (  \tilde{f}_{K} - \tilde{f}_{I})
$$ --- where $E_{K_3(H)} = \{I, J, K \}$,  $I < J < K$, and where $\tilde{f}_L$ are liftings to $\FS_{X^n}$ of functions $f_L$ in $\FS_{\Delta_L}$ ---
descends to a morphism of coherent sheaves $\tilde{D}: E^{1,0}_2 \rTo E^{3, -1}_2$, which coincides with 
the differential $d_2$ of the spectral sequence $E^{p,q}_1$.  
\end{lemma}

\begin{proof}It is easy to prove that the formula for $\tilde{D}$ well defines a morphism of sheaves $(\oplus_{|I|=2} \FS_{\Delta_I} )_0 \rTo    \oplus_{|H|=3}  (\Omega^1_X \boxtimes \mc{I}_{\Delta_{n-2}})_{H}$ 
as in the statement, which is clearly zero on $d_1(E^{0,0}_1)$ and hence induces a morphism of sheaves 
$\tilde{D}: E^{1,0}_2 \rTo E^{3,-1}_2$, by lemma \ref{lmm: E1terms}. 

We now prove that $\tilde{D}  = d_2$: we put ourselves in the situation explained in remarks \ref{rmk: A1}, \ref{rmk: A2}, \ref{rmk: A3}. Let $f_I$, $|I|=2$,  functions in $\FS_{\Delta_{I}}$ 
such that $(f_I)_{I}$ is in $ E^{1,0}_1$. For all $I$, let $\tilde{f}_I$ be regular functions in $\FS_{X^n}$ restricting to $f_{I}$.  
The element $(\tilde{f}_I)_I$ is in $L^{1,0}$ and its image $\partial ( (\tilde{f}_I)_I ) \in L^{2,0}$ is zero when projected to 
$E^{2,0}_1$.  In other word we have that 
$$ \partial( (\tilde{f}_I)_I)_{I_1 \cup I_2} = \epsilon_{I_1, I_1 \cup I_2} \tilde{f}_{I_1} + \epsilon_{I_2, I_1 \cup I_2} \tilde{f}_{I_2}  \in \mc{I}_{\Delta_{I_1 \cup I_2}} $$for pairs of cardinality $2$-multi-indexes $I_1, I_2$ with $I_1 \neq I_2$, where we indicated graphs with two edges just as a union of these. Therefore the element $ \partial( (f_I)_I)_{I_1 \cup I_2}$ can be lifted to $L^{2,-1}$ as 
$$    \epsilon_{I_1, I_1 \cup I_2} \tilde{f}_{I_1} + \epsilon_{I_2, I_1 \cup I_2} \tilde{f}_{I_2}  = a_{I_1, I_2} x_{I_1} + b_{I_1, I_2} y_{I_1} + c_{I_1, I_2} x_{I_2} + d_{I_1, I_2} y_{I_2} = \delta(w_{I_1, I_2})$$for elements 
$w_{I_1, I_2} = a_{I_1, I_2} \gamma^*_{I_1} + b_{I_1, I_2} \delta^*_{I_1} + c_{I_1, I_2} \gamma^*_{I_2} + d_{I_1, I_2} \delta^*_{I_2}$, 
where, according to remark \ref{rmk: A2}, we indicated with $\gamma^*_{I_i}, \delta^*_{I_i}$ a frame of $F^*_{I_i}$. 
The image in $L^{3,-1}$ via $\partial$ of liftings $w_{I_1, I_2}$ will represent the image of $d_2$. The complex $L^{3, \bullet}$
is now a direct sum of Koszul complexes $K^\bullet(F_\Gamma, s_\Gamma)$, where $\Gamma$ is a simple 
graph with $3$ edges and without isolated vertices. But if $\Gamma$ is acyclic then the correspondent Koszul complex is acyclic in negative degree, and the $\Gamma$-component of the image via $\partial$ will be zero in vertical cohomology and hence in $E^{3,-1}_2$. 
Hence 
we are just interested in components of the second differential $d_2$ indexed by $3$-cycles, determined by cardinality $3$ multi-indexes. Suppose now 
$H$ is such a multi-index and that $\{I, J, K \}$ are the edges of the corresponding $3$-cycles, with $I < J < K$. 
From $( \partial \circ \partial ((\tilde{f}_I)_I))_{K_3(H)} = 0$ we deduce 
$$ 
\partial( (\tilde{f}_I)_I)_{I \cup J} 
-\partial( (\tilde{f}_I)_I)_{I \cup K} + 
\partial( (\tilde{f}_I)_I)_{J \cup K} = 0 \;,$$which implies that 
\begin{gather*} ( a_{I, J} - a_{I, K} ) x_I + ( c_{I, J} + a_{J, K})x_J  + (- c_{I, K} + c_{J, K})x_K =0 \\
( b_{I, J} - b_{I, K} ) y_I + ( d_{I, J} +  b_{J, K})y_J  + (- d_{I, K} +  d_{J, K})y_K =0
\end{gather*}and hence, since $   x_I - x_J +x_K= 0$ and 
$y_I -y_J + y_K= 0$,   that
\begin{subequations} \label{subeq: rel}
\begin{align}  ( a_{I, J} - a_{I, K} ) = \: & (- c_{I, K} + c_{J, K})  = - (c_{I, J} + a_{J, K})  \\ 
 ( b_{I, J} - b_{I, K} ) = \: & (- d_{I, K} +  d_{J, K}) = - (d_{I, J} + b_{J, K}) 
 \end{align}
 \end{subequations}Finally the image of $d_2$ is represented in $L^{3,-1}$ by 
 $   w_{I \cup J} -  w_{I \cup K} +  w_{J \cup K} $, which is equal to 
 \begin{multline*}  (a_{I, J} - a_{I, K} ) \gamma^*_I + ( c_{I, J} + a_{J, K})\gamma^*_J  + (- c_{I, K} +  c_{J, K}) \gamma^*_K + \\ + 
( b_{I, J} - b_{I, K} ) \delta^*_I + ( d_{I, J} + b_{J, K}) \delta^*_J  + (- d_{I, K} +  d_{J, K}) \delta^*_K \qquad 
 \end{multline*}and which, using  relations (\ref{subeq: rel}),  can be rewritten as 
 $$  ( a_{I, J} - a_{I, K} ) \gamma^*_H + ( b_{I, J} - b_{I, K} ) \delta^*_H \;,$$which, by notation \ref{notat: A4} and in the identifications explained in remark \ref{rmk: A3}, is precisely the formula in the statement. \end{proof}

Over an affine open set of the form $U^n$ or $S^n U$, with $U = \Spec(A)$, we will identify sheaves with their modules of global sections. In particular, we will denote with 
$\Omega^1_A (- \Delta_{n-2}) = (\Omega^1_A \tens A^{n-3}) \tens \mc{I}_{\Delta_{n-2}}$ the module of sections of the sheaf $(\Omega^1_U \boxtimes \FS_{U^{n-3}}) \tens \mc{I}_{\Delta_{n-2}}$ over $U^n$, where $\mc{I}_{\Delta_{n-2}}$ is the ideal of the big diagonal in $U^{\tens n-2}$ and $\Omega^1_A \tens A^{n-3}$ is seen as a $A^{\tens n-2}$-module (see notation \ref{notat: omegadelta}) and hence a $A^{\tens n}$-module, via the contraction of the first three factors $A^{\tens n} \rTo A \tens A^{n-3}$. 
We denote with $\oplus_{|I|=2}(A \tens A^{n-2})_0$ the module of global sections of the sheaf $( \oplus_{|I|=2} \FS_{\Delta_I})_0$ over $U^n$ and 
with $(A \tens S^{n-2}A)_0$ its $\perm_n$-invariants over $S^n U$. 
\begin{notat}If $w_1 \tens \cdots \tens w_l$ is an element of $A^{\tens l}$, and $1 \leq i \leq l$, we indicate with $\widehat{w_i}$ the element 
$w_1 \tens \cdots \tens w_{i-1} \tens w_{i+1} \tens \cdots w_l \in A^{\tens l-1}$. We use an analogous notation  for an element $w_1 . \cdots . w_l \in S^l A$. 
\end{notat}
\begin{crl}\label{crl: B}Over an affine open set $U^n = \Spec A^{\tens n}$, the differential  $d_2: E^{1,0}_2 \rTo E^{3,-1}_2$  is induced by the map $ \tilde{D}: \big ( \oplus_{|I|=2} A \tens A^{\tens n-2} \big)_0 \rTo
\oplus_{ H \subseteq \{1, \dots, n \}, |H|=3} \Omega^1_A (-\Delta_{n-2})$, determined by  
$$ \tilde{D} ( (f_L)_L)_{H} =
ad u_{k-2} \tens \widehat{u_{k-2}} + b dv_{j-1} \tens \widehat{v_{j-1}} -   w_i dc \tens \widehat{w_i}$$where 
 $H = \{ i, j, k \}$, $i < j < k$, $I =\{ i,j \}, J = \{i, k\}, K =\{ j,k\} $, and where $f_{I} = a \tens u_1 \tens \cdots \tens u_{n-2}$, $f_J = b \tens v_1 \tens\cdots \tens v_{n-2}$, $f_K = c \tens w_1 \tens \cdots \tens w_{n-2}$. 
\end{crl}

\begin{pps}\label{pps: D}The invariant differential $d_2 : (E^{1,0}_2)^{\perm_n} \rTo (E^{3,-1}_2)^{\perm_n}$ is induced locally, over an affine open set of the form $S^n U$, by the map $D: (A \tens S^{n-2} A )_0 \rTo \Omega^1_A (- \Delta_{n-2})^{\perm_{n-3}}$ determined by 
$$ D (a \tens b_1 . \dots . b_{n-2} ) = \sum_{i=1}^{n-2} ( 2 a db_i - b_i d a ) \tens \widehat{b_i} \;.$$Here the group $\perm_{n-3}$ acts on the factor $A^{n-3}$ of the tensor product $A \tens A^{\tens n-3}$. 
\end{pps}

\paragraph{Determination of the map $A$ for $n =4$.}

In what follows we use the second convention of notation \ref{notat: 3cycle}. Any graph of the kind $C_4 \cup L$ has a distinguished edge $L$ (the only edge whose vertices are of degree $3$) and 
 therefore it can be decomposed uniquely as a union  of two $3$-cycles $H$ and $K$, determined by two cardinality $3$ multi-indexes $H$ and $K$ such that $H \cap K = L$. In what follows we will write such a graph just as $H \cup K$, instead of $K_3(H) \cup K_3(K)$. 
 
If $H$ is a cardinality $3$ multi-index, say $H = \{i, j , k \}$, with $i =  \min H$, we identify $\Delta_H$ with $X^{\{1, \dots, 4 \} \setminus \{j, k \} } \simeq X^2$. Moreover, suppose $K$ is a $3$-cycle $K = \{i_1, i_2, i_3 \}$ with 
 $i_1 < i_2 < i_3$. We say that a simple path in $K$ is positively oriented if, in the orientation of the path, the vertex following $i_1$ is $i_2$, negatively oriented if it is not positively oriented. When writing the coefficient $\eta_{I, K}$ for $I$ and edge of $K$ (see notation \ref{notat: etaIgamma}), we will always assume that $K$ is positively oriented. 
 
We introduce a general sign $\epsilon_{\Gamma, \Gamma^\prime}$ for a couple $(\Gamma, \Gamma^\prime)$ where  $\Gamma$ is a subgraph of $\Gamma^\prime$. If $E_{\Gamma^\prime} \setminus E_{\Gamma} = \{I_1, \dots, I_l \}$, in lexicographic order, 
then $\epsilon_{\Gamma, \Gamma^\prime} = \prod_{j=0}^{l-1} \epsilon_{\Gamma \cup I_1 \cup \cdots I_j, \Gamma \cup  I_1 \cup \cdots I_{j+1}}$.

According to these facts and notations, we have the first 
 \begin{lemma}\label{lmm: d2}Consider the map $\tilde{A}: \oplus_{|H|=3} (\Omega^1_X \boxtimes \mc{I}_{\Delta_2})_H \rTo \oplus_{\Gamma \in \mc{G}_{5,4}} \Lambda^2(\Omega^1_X \tens q_{\Gamma})_\Gamma$, whose component $\tilde{A}^{\Gamma}_H$ is  zero if $H$ is not a subgraph of $\Gamma$ and is defined by the formula 
  $$ \tilde{A}^{\Gamma}_H(\omega \tens f) =  
  - \epsilon_{H, H \cup K}  \eta_{I, K}(\omega \tens e_H ) \wedge (d_\Delta f \tens e_K) 
  $$if $\Gamma = H \cup K$, for some $3$-cycle $K$, with $I = \min E_{H \cup K} \setminus E_H$, where  $\omega \in \Omega^1_X$, $f \in \mc{I}_{\Delta_2}$ and where $e_H$, $e_K$ are base elements in $q_{H \cup K}$. 
  Then the image of $\tilde{A}$ is in $\ker d_1$; hence it induces a map $\tilde{A}: E^{3,-1}_2 \rTo E^{5, -2}_2$, which coincides with the second differential
$d_2$. 
  \end{lemma}
 \begin{proof}It is clear that the map $\tilde{A}$ is well defined and it is easy to see that its image lies in $\ker d_1$: this yields the good definition of the map $\tilde{A}:  E^{3,-1}_2 \rTo E^{5, -2}_2$. 
 We just have to prove that it coincides with $d_2$. 
 We put ourselves in the situation of remarks $\ref{rmk: A1}, \ref{rmk: A2}, \ref{rmk: A3}$. 
 Consider a differential form $\omega \tens f$ in 
 $(\Omega^1_X \boxtimes  \mc{I}_{\Delta_2})_H$, where $\omega \in \Omega^1_X$ and $f \in \mc{I}_{\Delta_2}$. 
 In what follows we identify the element $\omega \tens f \in (\Omega^1_X \boxtimes \mc{I}_{\Delta_2})_H$ with the element in $ \oplus_{|L|=3} (\Omega^1_X \boxtimes \mc{I}_{\Delta_2})_L$ whose $H$-component is $\omega \tens f$ and whose every other $L$-component, with $L \neq H$ is zero. 
 
To prove the statement, it is sufficient to compute the component $d_2(\omega \tens f)_{H \cup K}$ of the second differential, where $K = \{I, J, L \}$, where it will always be assumed that $I < J$ in the lexicographic order. The form $\omega \tens f$ can be rewritten as 
 $$ \omega \tens f = h (dx \tens 1) + g(dy \tens 1)$$over $X \times X \simeq \Delta_H$, where $h, g \in \mc{I}_{\Delta_2}$. Now 
 $\mc{I}_{\Delta_2}$ is generated over $X \times X$ by (classes in $\FS_{\Delta_H}$ of ) regular functions $x_I, y_I$ in $\FS_{X^4}$: hence we can lift $h$ and $g$ to regular functions (which we will still call $h$ and $g$) $h = a x_I + b y_I$, $g = c x_I + d y_I$ in $\FS_{X^4}$. 
 By remark \ref{rmk: A3},  the differential form $\omega \tens f$ can be represented, in 
 $L^{3,-1}$ as
 \begin{align*} \omega \tens f =  \: & h \gamma^*_H + g \delta^*_H 
  =  (a x_I + b y_I) \gamma^*_H + (c x_I + d y_I) \delta^*_H \;. 
  \end{align*}Since the element $\omega \tens f$ was chosen in $E^{3, -1}_2 = \ker (\partial : E^{3,-1}_1 \rTo E^{4, -1}_1)$, this means that 
  the vertical cohomology class of $\partial (\omega \tens f)$ in $E^{4, -1}_1 = H^{-1}_{\delta}(L^{4, \bullet})$ is zero. This is equivalent to saying 
  that $ \partial(\omega \tens f)_{H \cup I}$ and $\partial(\omega \tens f)_{H \cup J}$ 
   come from  elements
  in $L^{4, -2}$; for the first we can we can write 
 \begin{equation}\label{eq: HUI} \partial(\omega \tens f)_{H \cup I} = \epsilon_{H, H \cup I}(\omega \tens f) = \delta \Big( \epsilon_{H, H \cup I} \big( (a \gamma^*_I + b \delta^*_I) \wedge \gamma^*_H +(c \gamma^*_I + d \delta^*_I) \wedge \delta^*_H \big) \Big) \end{equation}For the second we have, taking into account that, 
  by definitions of the signs $\eta_{I, K}$ and omitting for brevity's sake the index $K$, we can write 
  $\eta_I x_I + \eta_J x_J + \eta_L x_L = 0$:
   \begin{multline} \label{mult: HUJ} \partial(\omega \tens f)_{H \cup J}= \epsilon_{H, H \cup J} ( \omega \tens f)  = 
  \delta \Big ( \epsilon_{H, H \cup J} \big( (a (- \eta_I \eta_J \gamma^*_J - \eta_I \eta_L \gamma^*_L) +  \\
  b  (- \eta_I \eta_J \delta^*_J - \eta_I \eta_L \delta^*_L) ) \wedge \gamma^*_H + 
  (c (- \eta_I \eta_J \gamma^*_J - \eta_I \eta_L \gamma^*_L) + \\ +d (- \eta_I \eta_J \delta^*_J - \eta_I \eta_L \delta^*_L) ) \wedge \delta^*_H \big) 
   \Big) \;.
  \end{multline}
  Now the element between parenthesis in (\ref{eq: HUI}) has image 
  $$ \epsilon_{H \cup I, H \cup K} \epsilon_{H, H \cup I} \big( (a \gamma^*_I + b \delta^*_I) \wedge \gamma^*_H + (c \gamma^*_I + d \delta^*_I) \wedge \delta^*_H  \big) $$via 
  $\partial$ in $L^{5, -2}$, while the element between parenthesis in (\ref{mult: HUJ}) has image
\begin{multline*} \epsilon_{H \cup J, H \cup K} \epsilon_{H, H \cup J} \big( (a (- \eta_I \eta_J \gamma^*_J - \eta_I \eta_L \gamma^*_L) +  
  b  (- \eta_I \eta_J \delta^*_J - \eta_I \eta_L \delta^*_L) ) \wedge \gamma^*_H + \\
  (c (- \eta_I \eta_J \gamma^*_J - \eta_I \eta_L \gamma^*_L) + d (- \eta_I \eta_J \delta^*_J - \eta_I \eta_L \delta^*_L) ) \wedge \delta^*_H  \big) \end{multline*}via $\partial$ in $L^{5, -2}$. The sum of the two terms is given by 
  \begin{multline*} 
  \epsilon_{H \cup I, H \cup K} \epsilon_{H, H \cup I}\Big( (a (\gamma^*_I - \eta_I \eta_J \gamma^*_J - \eta_I \eta_L \gamma^*_L) +  
  b  (\delta^*_I - \eta_I \eta_J \delta^*_J - \eta_I \eta_L \delta^*_L) ) \wedge \gamma^*_H + \\
  (c (\gamma^*_I - \eta_I \eta_J \gamma^*_J - \eta_I \eta_L \gamma^*_L) + d (\delta^*_I - \eta_I \eta_J \delta^*_J - \eta_I \eta_L \delta^*_L) ) \wedge \delta^*_H \Big)
  \end{multline*}since we can easily see that $\epsilon_{H \cup J, H \cup K} \epsilon_{H, H \cup I} = - \epsilon_{H \cup I, H \cup K} \epsilon_{H, H \cup J}$ because $ \epsilon_{H , H \cup I} = \epsilon_{H \cup J, H \cup K}$ and 
  $\epsilon_{H, H \cup J} = - \epsilon_{H \cup I, H \cup K}$. 
  Note now that $\gamma^*_K = \eta_I \gamma^*_I + \eta_J \gamma^*_J + \eta_L \gamma^*_L$, $\delta^*_K =  \eta_I \delta^*_I + \eta_J \delta^*_J + \eta_L \delta^*_L$. By lemma \cite[Lemma A.3]{Scalaarxiv2015}, we have that $d_2 (\omega \tens f)_{H \cup K}$ is represented by 
  the vertical cohomology class of 
  $$ \epsilon_{H ,  H \cup K}  \eta_{I, K} \big( ( a \gamma^*_K +  
  b   \delta^*_K ) \wedge \gamma^*_H + \\
  (c  \gamma^*_K + d \delta^*_K ) \wedge \delta^*_H  \big) \;.$$
  Since we identified the classes of $\gamma^*_H$, $\gamma^*_K$, $\delta^*_H$, $\delta^*_K$ with 
  $dx \tens e_H$, $dx \tens e_K$, $dy \tens e_H$, $dy \tens e_K$ in $\Omega^1_X \tens q_{H \cup K}$, respectively, we obtain the formula in the statement. \end{proof}
 
 \begin{remark}\label{rmk: handy}For future use the following computation wil turn out handy. 
 Let $X = \mbb{C}^2$. Consider the differential form 
 $ \omega \tens f$, as above, but now lift it to the element $(a x_M + b y_M) \gamma^*_H + (c x_M + d y_M) \delta^*_H$
 $\in L^{3,-1}$, for functions $a, b, c, d \in \FS_{X^4}$ and let $K$ a $3$-cycle as above with edges $\{I, J, L \}$,  such that the edge $M$ satisfies
 $M \not  \in H$, $M \not \in K$. 
 The stairway process in order to compute 
 the component $d_2(\omega \tens f)_{H \cup K}$ provides the element $$
   \epsilon_{H, H \cup K} \eta_{L, K}  \big( \gamma^*_H \wedge (a \gamma^*_K + b \delta^*_K) + \delta^*_H \wedge (c \gamma^*_K + d \delta^*_K) \big) = 
  \epsilon_{H, H \cup K} \eta_{L, K} (\omega \tens e_H) \wedge (d_\Delta f \tens e_K)  \in L^{5, -2} \;,$$up to elements coming from $L^{4, -2}$.  
\end{remark}
 
 As an immediate corollary, taking $\perm_4$-invariants, and using Danila's lemma for morphisms, we have 
 \begin{crl}\label{crl: A}The $\perm_4$-invariant higher differential $A = d^{\perm_4}_2: (E^{3, -1}_2)^{\perm_4} \rTo ( E^{5, -2}_2)^{\perm_4} $ takes an element $\omega \tens f$ in ${w_3}_*(\Omega^1_X \boxtimes \mc{I}_{\Delta_2})$ to the element $ \omega \wedge d_\Delta f$ in ${w_4}_*(\Lambda^2 \Omega^1_X)$. 
 \end{crl}
 \begin{proof}Indeed it is sufficient to take $H = \{1, 2, 3 \}$ and $K = \{1, 3, 4 \}$. Then $L = \{1, 3 \}$,   $I = \{1, 4 \}$, $J = \{3, 4\}$. 
 In order to compute the invariant differential $d^{\perm_4}_2 (\omega \tens f)$ of an element $\omega \tens f$, by \cite[Lemma A.1]{Scala2009D},
 we just have to compute $d_2 ( \omega \tens f + (24)_* \omega \tens f ) =  (\omega \tens e_H) \wedge (d_\Delta f \tens e_K) 
 + (\omega \tens e_K) \wedge (d_\Delta f \tens e_H) $ --- where we omit writing  the push-forward $\pi_*$ --- 
 but this can be identified with $ \omega \tens d_\Delta f - d_\Delta f \tens \omega = 
    \omega \wedge d_\Delta f$ in 
 $\Omega^1_X \tens \Omega^1_X \subseteq \Lambda^2( \Omega^1_X \tens q_{H \cup K})$.   \end{proof}

 \paragraph{Determination of the map $C$.}
 \begin{remark}\label{rmk: x14dx} Recall notation \ref{notat: kerA}. 
 Let $X = \mbb{C}^2$.  The differential forms $x_{14}dx$, $y_{14} dy$ are in the image of 
 $D: {w_2}_*(\FS_X \boxtimes \FS_{S^2 X})_0 {\rTo} 
 {w_3}_*(\Omega^1_X \boxtimes \mc{I}_{\Delta_2})_0 
 $. Indeed, by  corollary \ref{crl: B}, it is clear that $A (x_{14} dx ) = dx \wedge dx = 0 = dy \wedge dy = A (y_{14} dy)$. So both differential forms belong to ${w_3}_*(\Omega^1_X \boxtimes \mc{I}_{\Delta_2})_0$. On the other hand we have that 
 $x \tens x .1 \in {w_2}_*(\FS_X \boxtimes \FS_{S^2 X})_0$, since $ d_1^{\perm_4}(x \tens x .1) = 2 x \tens x   - 2 x \tens x = 0$ and, by proposition \ref{pps: D},  
 $D (x \tens x.1) = (2x dx - x dx ) \tens 1  - dx \tens x  = (dx \tens 1 )(x \tens 1 - 1 \tens x)$, which can be identified with 
 $x_{14} dx$. Similarly  $y_{14}dy  = D(y \tens y.1)$, and $y \tens y .1 \in  {w_2}_*(\FS_X \boxtimes \FS_{S^2 X})_0$. 
 \end{remark}
 
 \begin{remark}Since we have surjective maps $E^{6, -3}_1 \rOnto E^{6, -3}_2$ and $E^{6,-3}_2 \rOnto E^{6,-3}_3$, and since $E^{6,-3}_1 \simeq \Lambda^3 Q^*_{K_4}$, 
 we can see 
  $E^{6, -3}_3$ as a quotient of the bundle $\Lambda^3 Q^*_{K_4}$ over the small diagonal $\Delta_{1234}$. 
 If $a$ is an element of $\Lambda^3 Q^*_{K_4}$, we denote with $[a]$ the 
 class of its image in $E^{6,-3}_3$. 
 \end{remark}
 
 \begin{remark}The natural composition $  {w_3}_*(\Omega^1_X \boxtimes \mc{I}^2_{\Delta_2})
  {\rTo} {w_3}_*(\Omega^1_X \boxtimes \mc{I}_{\Delta_2})_0 \rOnto ( E^{3,-1}_3)^{\perm_4}$ is surjective. 
 \end{remark}\begin{proof}By GAGA principle, it is sufficient to prove the statement for $X = \mbb{C}^2$. But in this case 
 it follows by remark \ref{rmk: x14dx} and by corollary \ref{crl: A}; indeed
 a differential form $\omega \tens f \in {w_3}_*(\Omega^1_X \boxtimes \mc{I}_{\Delta_2})$ is in the kernel of $d_2^{\perm_4}$ if and only if it is of the form $\omega \tens f = a x_{14} dx + b y_{14} dy + \omega_1 \tens g$,  where $g$ is in $\mc{I}^2_{\Delta_2}$. But now the term $a x_{14} dx + b y_{14} dy$ is zero in $( E^{3,-1}_3)^{\perm_4}$, because of remark \ref{rmk: x14dx}. 
  \end{proof}
 
 By the previous remark, we can represent any element in $( E^{3, -1}_3)^{\perm_4}$ by an element in ${w_3}_*(\Omega^1_X \boxtimes \mc{I}_{\Delta_2}^2)$. In the proof of the next proposition we will use the following notation: if $I$ is a cardinality $2$-multi-index in $\{1, \dots, 4 \}$, we will indicate with $\Gamma(\widehat{I})$ the graph obtained by the complete graph $K_4$ removing the edge $I$, that is 
 $V_{\Gamma(\widehat{I})} = \{1, \dots, 4 \}$, $E_{\Gamma(\widehat{I})} = E_{K_4} \setminus \{I \}$. 
 \begin{pps}\label{pps: mapC}Consider the map $C: {w_3}_*(\Omega^1_X \boxtimes \mc{I}_{\Delta_2}^2) {\rTo} {w_4}_*(S^3 \Omega^1_X)$ defined by the formula $$C (\omega \tens f) =  \sym(\omega \tens d^2_\Delta f) \;,$$where $\omega \in \Omega^1_X$ and $f \in \mc{I}^2_{\Delta_2}$. It descends to a map $C: E^{3,-1}_3 \rTo E^{6,-3}_3$, which coincides, up to a constant, with $d_3^{\perm_4}$. 
  \end{pps}
  \begin{proof}It is clear that the formula induces a well defined map $C: E^{3,-1}_3 \rTo E^{6,-3}_3$. It is sufficient to prove that 
  this map coincides up to constants with the invariant differential $d^{\perm_4}_3$. We put ourselves in the situation explained in 
  remarks \ref{rmk: A1}, \ref{rmk: A2}, \ref{rmk: A3}. For brevity's sake, we indicate with $H_0$ the cardinality $3$-multi-index $\{1, 2, 3 \}$ and the associated $3$-cycle. 
  We identify $(E^{3, -1}_2)^{\perm_4}$ with $ {w_3}_*(\Omega^1_X \boxtimes \mc{I}_{\Delta_2}) \simeq \pi_*( (\Omega^1_X )_{H_0} \tens  \mc{I}_{\Delta_{14}})$; hence
  ${w_3}_*(\Omega^1_X \boxtimes \mc{I}_{\Delta_2}^2)$ can be identified with $\pi_*( (\Omega^1_X)_{H_0} \tens \mc{I}^2_{\Delta_{14}})\subseteq \ker d_2^{\perm_4}$. 
  By Danila's lemma for morphisms, if $\omega \tens f \in \pi_*( (\Omega^1_X)_{H_0} \tens \mc{I}^2_{\Delta_{14}})$, then 
  \begin{align*} d_3^{\perm_4}(\omega \tens f) = \: & d_3 \big( \sum_{[\tau] \in \perm_4 / \perm_3 } \tau_* (\omega \tens f) \big)
  =  \sum_{[\tau] \in \perm_4 / \perm_3 } \tau_* d_3 (\omega \tens f) \;,
  \end{align*}where, when writing $d_3(\omega \tens f)$ we think of $\omega \tens f$ as an element in $ \ker d_2 \subseteq 
  \oplus_{H} (\Omega^1_X \boxtimes \mc{I}_{\Delta_2})_H$. 
  
  Hence we just need to compute $d_3(\omega \tens f)$ in $E^{6, -3}_3$. 
The element 
  $\omega \tens f \in \pi_*( (\Omega^1_X)_{H_0} \tens \mc{I}^2_{\Delta_{14}})$ can be written as 
  $\omega \tens f = h(dx \tens 1) + g (dy \tens 1)$, where $h, g \in \mc{I}^2_{\Delta_{14}}$. Writing $h = a x_{14}^2 + b x_{14} y_{14} + c y_{14}^2$ and $g = a^\prime x_{14}^2 + b^\prime x_{14} y_{14} + c^\prime y_{14}^2$, we see that is sufficient 
  to compute the image for $d_3$ of differential forms of the kind 
  $\alpha x_{14}^2 dx$, $\alpha x_{14} y_{14} dx$, $\alpha y_{14}^2 dx$, $\alpha x_{14}^2 dy$, $\alpha x_{14} y_{14} dy$, $\alpha y_{14}^2 dy$, for an arbitrary function $\alpha \in \FS_{X \times X}$, of course, thinking of these differential forms as elements in $\ker d_2$.  Let's begin with $\alpha x_{14}^2 dx$. We have that $d_2(\alpha x_{14}^2 dx) = 0$ in $E^{5, -2}_2$: this means that its components are zero:
  \begin{align*}   d_2(\alpha x_{14}^2 dx)_{\Gamma(\widehat{14})} = 0 \:, \hspace{1.3cm}
   d_2(\alpha x_{14}^2 dx)_{\Gamma(\widehat{24})} = 0 \:,\hspace{1.3cm}
  d_2(\alpha x_{14}^2 dx)_{\Gamma(\widehat{34})} = 0 \;.
  \end{align*}By lemma \ref{lmm: d2}, we have that, in terms of representants in $\Lambda^2(Q^*_{\Gamma(\widehat{24})})$ and 
  $\Lambda^2(Q^*_{\Gamma(\widehat{34})})$ \begin{gather*} 
  d_2(\alpha x_{14}^2 dx)_{\Gamma(\widehat{24})} = [  \alpha x_{14} \gamma^*_{123} \wedge \gamma^*_{134}] \:, \hspace{1.5cm}
  d_2(\alpha x_{14}^2 dx)_{\Gamma(\widehat{34})} = [ \alpha x_{14} \gamma^*_{123} \wedge \gamma^*_{124} ] \;.
  \end{gather*}In order to compute the representant of $d_2(\alpha x_{14}^2 dx)_{\Gamma(\widehat{14})}$ in $L^{5, -2}$, we invoke remark 
  \ref{rmk: handy} and we find, in terms of representants in $\Lambda^2(Q^*_{\Gamma(\widehat{14})})$: 
  $$ d_2(\alpha x_{14}^2 dx)_{\Gamma(\widehat{14})} = [  -\alpha x_{14} \gamma^*_{123} \wedge \gamma^*_{234}] \;.$$We now lift the elements we found to $L^{5, -3}$. We get
  \begin{gather*}   \alpha x_{14} \gamma^*_{123} \wedge \gamma^*_{134} = \delta(  \alpha \gamma^*_{14} \wedge \gamma^*_{123} \wedge \gamma^*_{134} ) \in \Lambda^3(\mbb{C}^2 \tens W_{\Gamma(\widehat{24})} ) \\ 
   \alpha x_{14} \gamma^*_{123} \wedge \gamma^*_{124} = \delta(  \alpha \gamma^*_{14} \wedge \gamma^*_{123} \wedge \gamma^*_{124} ) \in \Lambda^3(\mbb{C}^2 \tens W_{\Gamma(\widehat{34})} ) \\
    - \alpha x_{14} \gamma^*_{123} \wedge \gamma^*_{234} = \delta(   -\alpha (\gamma^*_{12} + \gamma^*_{24})  \wedge \gamma^*_{123} \wedge \gamma^*_{234} ) \in \Lambda^3(\mbb{C}^2 \tens W_{\Gamma(\widehat{14})} )  \;.
  \end{gather*}
 By \cite[Lemma A.3]{Scalaarxiv2015}, 
  the term $d_3(\alpha x_{14}^2 dx)$ in $E^{6,-3}_3$ is represented by the sum of the images of the preceding liftings via the horizontal differential $\partial$: hence:
  $$ d_3(\alpha x_{14}^2 dx) = [  \alpha \gamma^*_{14} \wedge \gamma^*_{123} \wedge(  \gamma^*_{134} - \gamma^*_{124}) +
   \alpha ( -\gamma^*_{12} - \gamma^*_{24})  \wedge \gamma^*_{123} \wedge \gamma^*_{234} )] $$
   Note now that $\gamma^*_{234} = \gamma^*_{123} + \gamma^*_{134} - \gamma^*_{124}$ and that 
   the term $( -\gamma^*_{12} - \gamma^*_{24}) \wedge \gamma^*_{123} \wedge \gamma^*_{123}$ is zero in $E^{6, -3}_3$ since 
   it comes from something in $L^{5, -3}_3$. Hence we get that 
   \begin{align*} d_3(\alpha x_{14}^2 dx) = \: & [  - \alpha (\gamma^*_{12} -\gamma^*_{14}  + \gamma^*_{24})  \wedge \gamma^*_{123} \wedge(  \gamma^*_{134} - \gamma^*_{124} )]  
   =  [ \alpha \gamma^*_{123} \wedge \gamma^*_{124} \wedge \gamma^*_{134}]
\end{align*}where again we simplified terms coming from $L^{5, -3}$.   The term $ \alpha \gamma^*_{123} \wedge \gamma^*_{124} \wedge \gamma^*_{134}$ belongs to $\Lambda^3(Q^*_{K_4}) \simeq \Lambda^3(\mbb{C}^2 \tens q_{K_4})$ and can be identified with 
$\alpha (dx \tens e_{123}) \wedge (dx \tens e_{124}) \wedge (dx \tens e_{134}) =   \alpha (dx)^3 \tens (e_{123} \wedge e_{124} \wedge e_{134}) \in S^3 \Omega^1_X \tens \Lambda^3 q_{K_4} \subseteq \Lambda^3(\Omega^1_X \tens q_{K_4})$. 

When computing $d_3(\alpha x_{14} y_{14} dx)$, we have, analogously to the previous case that 
$d_2(\alpha x_{14} y_{14} dx)_{\Gamma(\widehat{24})}$ is represented in $L^{5, -2}$ 
by  $  \alpha x_{14} \gamma^*_{123} \wedge \delta^*_{134} $ and hence we have the lifting 
$ \alpha x_{14} \gamma^*_{123} \wedge \delta^*_{134}  = \delta(  \alpha \gamma^*_{14} \wedge \gamma^*_{123} \wedge \delta^*_{134})$; 
moreover $d_2(\alpha x_{14} y_{14} dx)_{\Gamma(\widehat{34})}$ is represented by 
$  \alpha x_{14} \gamma^*_{123} \wedge \delta^*_{134}$ and hence can be lifted to $  \alpha \gamma^*_{14} \wedge \gamma^*_{123} \wedge \delta^*_{134}$ in $L^{5, -2}$; finally, by remark \ref{rmk: handy}, 
$$ d_2(\alpha x_{14} y_{14} dx)_{\Gamma(\widehat{14})} = [ -\alpha x_{14} (\delta^*_{123} - \delta^*_{234}) \wedge \gamma^*_{123} ] = 
[  \alpha x_{14} \gamma^*_{123} (\delta^*_{123}- \delta^*_{234})]$$and we have the lifting 
$  \alpha x_{14} \gamma^*_{123} (\delta^*_{123}- \delta^*_{234}) =   \alpha(\gamma^*_{12} + \gamma^*_{24}) \wedge \gamma^*_{123} \wedge 
(\delta^*_{123} - \delta^*_{234}) $. Then, using that $\delta^*_{234} = \delta^*_{123} + \delta^*_{134} - \delta^*_{124}$, we get that
 $d_3(\alpha x_{14} y_{14} dx)$ is given by the class 
\begin{align*} d_3(\alpha x_{14} y_{14} dx) = \: & [ \alpha \gamma^*_{123} (\delta^*_{134} - \delta^*_{124}) \wedge (\gamma^*_{14} - \gamma^*_{12} - \gamma^*_{24}) ] 
=   [  \alpha \gamma^*_{123} \wedge \gamma^*_{124} \wedge \delta^*_{134}]
\end{align*}were we simplified elements coming from $L^{5, -3}$. 
Analogously $$d_3(\alpha y^2_{14} dx) =  [  \alpha \gamma^*_{123} \wedge \delta^*_{124} \wedge \delta^*_{134}] \;.$$
The computation of all other elements is done by symmetry. We than finally have that, for a differential form $\omega \tens f$ in $\pi_*((\Omega^1_X)_{H_0} \tens \mc{I}_{\Delta_{14}})$ as in the beginning, the differential
$ d_3(\omega \tens f)$  is represented by the class of  $ \omega \tens d^2_\Delta f$ in $( \Omega^1_X \tens S^2 \Omega^1_X)_{K_4} 
\subseteq  \Lambda^3 (\Omega^1_X \tens q_{K_4})_{K_4}$. 

To finish the computation we remark the following two facts. Firstly, the terms $\gamma^*_{123} \wedge \gamma^*_{124} \wedge \delta^*_{134}$, 
satifies the equality $[\gamma^*_{123} \wedge \gamma^*_{124} \wedge \delta^*_{134}] = [\gamma^*_{123} \wedge \delta^*_{124} \wedge \gamma^*_{134}]$. Indeed, simplifying at each step elements coming from $L^{5,-3}$, 
\begin{align*}[\gamma^*_{123} \wedge \gamma^*_{124} \wedge \delta^*_{134}] = \: & [\gamma^*_{123} \wedge (\gamma^*_{123} + \gamma^*_{134} - \gamma^*_{234}) \wedge \delta^*_{134}] =  - [\gamma^*_{123} \wedge \gamma^*_{234} \wedge \delta^*_{134} ] \\
 = \: & - [\gamma^*_{123} \wedge \gamma^*_{234} \wedge (\delta^*_{234} - \delta^*_{123} + \delta^*_{124})] 
 = -  [\gamma^*_{123} \wedge \gamma^*_{234} \wedge \delta^*_{124}] \\
 = \: & -  [\gamma^*_{123} \wedge (\gamma^*_{123}  + \gamma^*_{134} - \gamma^*_{124}) \wedge \delta^*_{124})] \\ 
 = \: & - [\gamma^*_{123} \wedge \gamma^*_{134} \wedge \delta^*_{124}] = [\gamma^*_{123} \wedge \delta^*_{124} \wedge \gamma^*_{134}]
 \end{align*}
The same is true for any nontrivial triple wedge products of vectors associated to $3$-cycles $\gamma^*_{H}$, $\delta^*_K$. 
Secondly, the class in $E^{6, -3}_3$ of any such nontrivial triple wedge product is $\perm_4$-invariant. The proof of this fact is similar to that of the previous fact.  Hence, up to positive  constants
$$ d_3^{\perm_4}(\omega \tens f) =  \sym (\omega \tens d^2_\Delta f) $$in ${w_4}_*(S^3 \Omega^1_X \tens \Lambda^3 q_{K_4})^{\perm_4} \simeq {w_4}_*(S^3 \Omega^1_X) $. 
  \end{proof}

\vspace{1cm}
\noindent
{Departamento de  Matem\'atica, Puc-Rio, Rua Marqu\^es S\~ao Vicente 225, 22451-900 G\'avea, Rio de Janeiro, RJ, Brazil}

\noindent 
{\it Email address:} {\tt lucascala@mat.puc-rio.br}


\footnotesize


\begin{thebibliography}{BKR01}

\bibitem[BKR01]{BridgelandKingReid2001}
Tom Bridgeland, Alastair King, and Miles Reid.
\newblock The {M}c{K}ay correspondence as an equivalence of derived categories.
\newblock {\em J. Amer. Math. Soc.}, 14(3):535--554 (electronic), 2001.

\bibitem[BS91]{BeltramettiSommese1991}
M~Beltrametti and Andrew~J Sommese.
\newblock Zero cycles and k-th order embeddings of smooth projective surfaces.
\newblock In {\em Problems in the theory of surfaces and their classification,
  Symposia Math}, volume~32, pages 33--48, 1991.

\bibitem[CG90]{CataneseGoettsche1990}
Fabrizio Catanese and Lothar G{\oe}ttsche.
\newblock d-very-ample line bundles and embeddings of {H}ilbert schemes of
  0-cycles.
\newblock {\em Manuscripta Mathematica}, 68(1):337--341, 1990.

\bibitem[Dan01]{Danila2001}
Gentiana Danila.
\newblock Sur la cohomologie d'un fibr\'e tautologique sur le sch\'ema de
  {H}ilbert d'une surface.
\newblock {\em J. Algebraic Geom.}, 10(2):247--280, 2001.

\bibitem[Die10]{DiestelGT}
Reinhard Diestel.
\newblock {\em Graph theory}, volume 173 of {\em Graduate Texts in
  Mathematics}.
\newblock Springer, Heidelberg, fourth edition, 2010.

\bibitem[DN89]{drezetNarasimhan1989}
J.-M. Drezet and M.~S. Narasimhan.
\newblock Groupe de {P}icard des vari\'et\'es de modules de fibr\'es
  semi-stables sur les courbes alg\'ebriques.
\newblock {\em Invent. Math.}, 97(1):53--94, 1989.

\bibitem[FH91]{FultonHarrisRT}
William Fulton and Joe Harris.
\newblock {\em Representation theory}, volume 129 of {\em Graduate Texts in
  Mathematics}.
\newblock Springer-Verlag, New York, 1991.
\newblock A first course, Readings in Mathematics.

\bibitem[GAP15]{GAP4}
The GAP~Group.
\newblock {\em {GAP -- Groups, Algorithms, and Programming, Version 4.7.8}},
  2015.

\bibitem[Hai99]{Haiman1999}
Mark Haiman.
\newblock Macdonald polynomials and geometry.
\newblock In {\em New perspectives in algebraic combinatorics (Berkeley, CA,
  1996--97)}, volume~38 of {\em Math. Sci. Res. Inst. Publ.}, pages 207--254.
  Cambridge Univ. Press, Cambridge, 1999.

\bibitem[Hai01]{Haiman2001}
Mark Haiman.
\newblock Hilbert schemes, polygraphs and the {M}acdonald positivity
  conjecture.
\newblock {\em J. Amer. Math. Soc.}, 14(4):941--1006 (electronic), 2001.

\bibitem[Hai02]{Haiman2002}
Mark Haiman.
\newblock Vanishing theorems and character formulas for the {H}ilbert scheme of
  points in the plane.
\newblock {\em Invent. Math.}, 149(2):371--407, 2002.

\bibitem[Leh99]{Lehn1999}
Manfred Lehn.
\newblock Chern classes of tautological sheaves on {H}ilbert schemes of points
  on surfaces.
\newblock {\em Invent. Math.}, 136(1):157--207, 1999.

\bibitem[mof]{moflw96705}
http://mathoverflow.net/questions/96705.

\bibitem[Sca09]{Scala2009D}
Luca Scala.
\newblock {C}ohomology of the {H}ilbert scheme of points on a surface with
  values in representations of tautological bundles.
\newblock {\em Duke Math. J.}, 150(2):211--267, 2009.

\bibitem[Sca15a]{Scalaarxiv2015}
Luca Scala.
\newblock {H}igher symmetric powers of tautological bundles on {H}ilbert
  schemes of points on a surface.
\newblock {\em {\rm arXiv: 1502.07595v1}}, 2015.

\bibitem[Sca15b]{Scala2015isospectral}
Luca Scala.
\newblock Singularities of the {I}sospectral {H}ilbert {S}cheme.
\newblock {\em {\rm arXiv: 1510.03071}}, 2015.

\bibitem[Ser77]{SerreLRFG}
Jean-Pierre Serre.
\newblock {\em Linear representations of finite groups}.
\newblock Springer-Verlag, New York-Heidelberg, 1977.
\newblock Translated from the second French edition by Leonard L. Scott,
  Graduate Texts in Mathematics, Vol. 42.

\end{thebibliography}
\end{document}